\documentclass[a4paper]{amsart}

\usepackage[french, english]{babel} 
\usepackage[T1]{fontenc}
\usepackage[ansinew]{inputenc}

\title[Quadratic differential equations]{Quadratic differential equations : partial Gelfand-Shilov smoothing effect and null-controllability}

\author{Paul Alphonse}
\address{Paul Alphonse, Univ Rennes, CNRS, IRMAR - UMR 6625, F-35000 Rennes}
\email{paul.alphonse@univ-rennes1.fr}



\usepackage[top=3cm, bottom=2cm, left=3cm, right=3cm]{geometry}		

\usepackage{enumitem}

\usepackage{array}											

\usepackage{amsfonts}
\usepackage{amsmath}
\usepackage{amsthm}
\usepackage{amssymb}

\usepackage{bbm}

\usepackage{stmaryrd}

\usepackage[scr]{rsfso}

\usepackage{comment}

\usepackage{hyperref}

\usepackage{xcolor}

\usepackage{constants}
\newconstantfamily{c}{symbol=c}

\numberwithin{equation}{section}

\newtheorem{thm}{Theorem}[section]
\newtheorem{prop}[thm]{Proposition}
\newtheorem{lem}[thm]{Lemma}
\newtheorem{cor}[thm]{Corollary}
\theoremstyle{definition}
\newtheorem{dfn}[thm]{Definition}
\newtheorem{ex}[thm]{Example}

\DeclareMathOperator{\Supp}{Supp}
\DeclareMathOperator{\Reelle}{Re}
\DeclareMathOperator{\Op}{Op}
\DeclareMathOperator{\Tr}{Tr}
\DeclareMathOperator{\Ker}{Ker}
\DeclareMathOperator{\Imag}{Im}

\DeclareMathOperator{\Span}{Span}
\DeclareMathOperator{\Rank}{Rank}
\DeclareMathOperator{\GL}{GL}
\DeclareMathOperator{\Diag}{Diag}
\DeclareMathOperator{\Log}{Log}
\DeclareMathOperator{\Ran}{Ran}

\begin{document}

\sloppy

\selectlanguage{english}

\begin{abstract}
We study the partial Gelfand-Shilov regularizing effect and the exponential decay for the solutions to evolution equations associated to a class of accretive non-selfadjoint quadratic operators, which fail to be globally hypoelliptic on the whole phase space. By taking advantage of the associated Gevrey regularizing effects, we study the null-controllability of parabolic equations posed on the whole Euclidean space associated to this class of possibly non-globally hypoelliptic quadratic operators. We prove that these parabolic equations are null-controllable in any positive time from thick control subsets. This thickness property is known to be a necessary and sufficient condition for the null-controllability of the heat equation posed on the whole Euclidean space. Our result shows that this geometric condition turns out to be a sufficient one for the null-controllability of a large class of quadratic differential operators.
\end{abstract}

\maketitle

\section{Introduction}
\label{intro}

\subsection{Miscellaneous facts about quadratic operators}We study in this work quadratic operators, that is the pseudodifferential operators
\begin{equation}
	q^w(x,D_x)u(x) = \frac1{(2\pi)^n}\int_{\mathbb R^{2n}}e^{i(x-y)\cdot\xi}q\bigg(\frac{x+y}2,\xi\bigg)u(y)dyd\xi,
\end{equation}
defined by the Weyl quantization of complex-valued quadratic symbols
$$q:\mathbb R^n_x\times\mathbb R^n_{\xi}\rightarrow\mathbb C,$$
on the phase space $\mathbb R^n_x\times\mathbb R^n_{\xi}$, with $n\geq1$. These non-selfadjoint operators are only differential operators since the Weyl quantization of the quadratic symbols $x^{\alpha}\xi^{\beta}$, with $(\alpha,\beta)\in\mathbb N^{2n}$, $\vert\alpha+\beta\vert = 2$, is given by
\begin{equation}
	(x^{\alpha}\xi^{\beta})^w = \Op^w(x^{\alpha} \xi^{\beta}) = \frac12\big(x^{\alpha} D^{\beta}_x+D^{\beta}_xx^{\alpha}\big),
\end{equation}
with $D_x = i^{-1}\partial_x$. It is known from \cite{MR1339714} (pp. 425-426) that the maximal closed realization of a quadratic operator $q^w(x,D_x)$ on $L^2(\mathbb R^n)$, that is the operator equipped with the domain
\begin{equation}\label{04122017E3}
	D(q^w) = \big\{u\in L^2(\mathbb R^n) : q^w(x,D_x)u\in L^2(\mathbb R^n)\big\},
\end{equation}
where $q^w(x,D_x)u$ is defined in the distribution sense, coincides with the graph closure of its restriction to the Schwartz space
$$q^w(x,D_x):\mathscr S(\mathbb R^n)\rightarrow\mathscr S(\mathbb R^n).$$ 

Classically, to any quadratic form defined on the phase space $q:\mathbb R^n_x\times\mathbb R^n_{\xi}\rightarrow\mathbb C$ is associated a matrix $F\in M_{2n}(\mathbb C)$ called its Hamilton map, or its fundamental matrix, which is defined as the unique matrix satisfying the identity
\begin{equation}\label{04122017E4}
	\forall X,Y\in\mathbb R^{2n},\quad q(X,Y) = \sigma(X,FY),
\end{equation}
with $q(\cdot,\cdot)$ the polarized form associated to the quadratic form $q$, and $\sigma$ the standard symplectic form given by
\begin{equation}\label{04122017E5}
	\sigma((x,\xi),(y,\eta)) = \langle\xi,y\rangle -\langle x,\eta\rangle,\quad (x,y),(\xi,\eta)\in\mathbb C^{2n},
\end{equation}
and where $\langle\cdot,\cdot\rangle$ denotes the inner product on $\mathbb C^n$ defined by
$$\langle x,y\rangle = \sum_{j=0}^nx_jy_j,\quad x =(x_1,\ldots,x_n),\ y=(y_1,\ldots,y_n)\in\mathbb C^n.$$
Note that $\langle\cdot,\cdot\rangle$ is linear in both variables but not sesquilinear. By definition, $F$ is given by
\begin{equation}\label{21062018E6}
	F = JQ,
\end{equation}
where $Q\in S_{2n}(\mathbb C)$ is the symmetric matrix associated to the bilinear form $q(\cdot,\cdot)$,
\begin{equation}\label{29062018E3}
	\forall X,Y\in\mathbb R^{2n},\quad q(X,Y) = \langle X,QY\rangle,
\end{equation}
and $J\in\GL_{2n}(\mathbb R)$ stands for the symplectic matrix
$$J = \begin{pmatrix}
	0_n & I_n \\
	-I_n & 0_n
\end{pmatrix}\in\GL_{2n}(\mathbb R),$$
with $0_n\in M_n(\mathbb R)$ the null matrix and $I_n\in M_n(\mathbb R)$ the identity matrix. We notice that a Hamilton map is always skew-symmetric with respect to the symplectic form, since
\begin{equation}\label{04122017E6}
	\forall X,Y\in\mathbb R^{2n},\quad \sigma(X,FY) = q(X,Y) = q(Y,X) = \sigma(Y,FX) = -\sigma(FX,Y),
\end{equation}
by symmetry of the polarized form and skew-symmetry of $\sigma$.

When the real part of the symbol is non-negative $\Reelle q\geq0$, the quadratic operator $q^w(x,D_x)$ equipped with the domain \eqref{04122017E3} is shown in \cite{MR1339714} (pp. 425-426) to be maximal accretive and to generate a strongly continuous contraction semigroup $(e^{-tq^w})_{t\geq0}$ on $L^2(\mathbb R^n)$. 
Moreover, for all $t\geq0$, $e^{-tq^w}$ is a pseudodifferential operator whose Weyl symbol is a tempered distribution $p_t\in\mathscr S'(\mathbb R^{2n})$. More specifically, this symbol is a $L^{\infty}(\mathbb R^{2n})$ function explicitly given by the Mehler formula
\begin{equation}\label{04122017E8}
	p_t(X) = \frac1{\sqrt{\det(\cos(tF))}}e^{-\sigma(X,\tan(tF)X)}\in L^{\infty}(\mathbb R^{2n}),\quad X\in\mathbb R^{2n},
\end{equation}
whenever the condition $\det(\cos(tF))\ne 0$ is satisfied, see \cite{MR1339714} (Theorem 4.2), with $F$ the Hamilton map of $q$. For example, the Schr\"odinger operator $i(D^2_x+x^2)$ generates a group $(e^{-it(D^2_x+x^2)})_{t\in\mathbb R}$ whose elements are pseudodifferential operators, and their Weyl symbols are respectively given by
$$(x,\xi)\mapsto \frac1{\cos t}e^{-i(\xi^2+x^2)\tan t}\in L^{\infty}(\mathbb R^{2n}),$$
when $\cos t\ne 0$, whereas when $t=\frac{\pi}{2} + k\pi$, with $k\in\mathbb Z$, it is given by the Dirac mass
$$(x,\xi)\mapsto i(-1)^{k+1}\pi\delta_0(x,\xi)\in\mathscr S'(\mathbb R^{2n}).$$
This example is taken from \cite{MR1339714} (p. 427) and shows that the condition $\det(\cos(tF))\ne 0$ is not always satisfied for any $t\geq0$.

The notion of singular space associated to any complex-valued quadratic form $q:\mathbb R^n_x\times\mathbb R^n_{\xi}\rightarrow\mathbb C$ defined on the phase space, introduced in \cite{MR2507625} (formula (1.1.15)) by M. Hitrik and K. Pravda-Starov, is defined as the following finite intersection of kernels
\begin{equation}\label{04122017E7}
	S = \bigcap_{j=0}^{2n-1}\Ker(\Reelle F(\Imag F)^j)\cap\mathbb R^{2n},
\end{equation}
where $\Reelle F$ and $\Imag F$ stand respectively for the real and imaginary parts of the Hamilton map $F$ associated to the quadratic symbol $q$,
$$\Reelle F = \frac12(F+\overline F)\quad \text{and}\quad \Imag F = \frac1{2i}(F-\overline F).$$
According to \eqref{04122017E7}, we may consider $0\le k_0\le 2n-1$ the smallest integer satisfying 
\begin{equation}\label{22062018E1}
	S = \bigcap_{j=0}^{k_0}\Ker(\Reelle F(\Imag F)^j)\cap\mathbb R^{2n}.
\end{equation}
When the quadratic symbol has a non-negative real part $\Reelle q\geq0$, the singular space can be defined in an equivalent way as the subspace in the phase space where all the Poisson brackets
$$H_{\Imag q}^k\Reelle q = \left[\frac{\partial\Imag q}{\partial\xi}\cdot\frac{\partial}{\partial x}-\frac{\partial\Imag q}{\partial x}\cdot\frac{\partial}{\partial \xi}\right]^k\Reelle q,\quad k\geq0,$$
are vanishing
$$S = \big\{X\in\mathbb R^{2n} : (H^k_{\Imag q}\Reelle q)(X) = 0,\ k\geq0\big\}.$$
This dynamical definition shows that the singular space corresponds exactly to the set of points $X\in\mathbb R^{2n}$, where the real part of the symbol $\Reelle q$ under the flow of the Hamilton vector $H_{\Imag q}$ associated with its imaginary part 
\begin{equation}\label{28062018E1}
	t\mapsto\Reelle q(e^{tH_{\Imag q}}X),
\end{equation}
vanishes to any order at $t=0$. This is also equivalent to the fact that the function \eqref{28062018E1} is identically zero on $\mathbb R$. 

As pointed out in \cite{MR2507625, MR3342487, MR2752935, MR3244980}, the singular space is playing a basic role in understanding the spectral and hypoelliptic properties of non-elliptic quadratic operators, as well as the spectral and pseudospectral properties of certain classes of degenerate doubly characteristic pseudodifferential operators \cite{MR2753626, MR3137478}. For example, when the singular space of $q$ is equal to zero $S = \{0\}$, the quadratic operator $q^w(x,D_x)$ is shown in \cite{MR2752935} (Theorem 1.2.1) to be hypoelliptic and to enjoy global subelliptic estimates of the type
$$\exists C>0, \forall u\in\mathscr S(\mathbb R^n),\quad \big\Vert\langle(x,D_x)\rangle^{\frac2{2k_0+1}}u\big\Vert_{L^2(\mathbb R^n)}
\le C\big[\Vert q^w(x,D_x)u\Vert_{L^2(\mathbb R^n)} + \Vert u\Vert_{L^2(\mathbb R^n)}\big],$$
where $$\langle(x,D_x)\rangle^2 = 1+\vert x\vert^2+\vert D_x\vert^2,$$ and $0\le k_0\le 2n-1$ is the smallest integer such that \eqref{22062018E1} holds.

The notion of singular space also allows to understand the propagation of Gabor singularities for solutions to parabolic equations associated to accretive quadratic operators. The Gabor wave front set (or Gabor singularities) $WF(u)$ of a tempered distribution $u$ measures the directions in the phase space in which a tempered distribution does not behave like a Schwartz function. We refer the reader e.g. to \cite{MR3880300} (Section 5) for the definition and the basic properties of the Gabor wave front set. We only recall here that the Gabor wave front set of a tempered distribution is empty if and only if this distribution is a Schwartz function:
$$\forall u\in\mathscr S'(\mathbb R^n),\quad WF(u) = \emptyset \Leftrightarrow u\in\mathscr S(\mathbb R^n).$$
The following microlocal inclusion is proven in \cite{MR3756858} (Theorem 6.2):
\begin{equation}\label{25042018E14}
	\forall u\in L^2(\mathbb R^n),\forall t>0,\quad WF(e^{-tq^w}u)\subset e^{tH_{\Imag q}}(WF(u)\cap S)\subset S,
\end{equation}
where $(e^{tH_{\Imag q}})_{t\in\mathbb R}$ is the flow generated  by the Hamilton vector field 
$$H_{\Imag q} = \frac{\partial\Imag q}{\partial\xi}\cdot\frac{\partial}{\partial x}-\frac{\partial\Imag q}{\partial x}\cdot\frac{\partial}{\partial \xi}.$$
This result shows that the singular space $S$ contains all the directions in the phase space in which the semigroup $(e^{-tq^w})_{t\geq0}$ does not regularize in the Schwartz space $\mathscr S(\mathbb R^n)$. The microlocal inclusion \eqref{25042018E14} was shown to hold as well for other types of wave front sets, as Gelfand-Shilov wave front sets \cite{MR3649471}, or polynomial phase space wave front sets \cite{MR3767155}. 

\subsection{Smoothing properties of semigroups generated by accretive quadratic operators} Given $q:\mathbb R^n_x\times\mathbb R^n_{\xi}\rightarrow\mathbb C$ a complex-valued quadratic form with a non-negative real part $\Reelle q\geq0$, we study in the first part of this work the smoothing effects of the semigroup $(e^{-tq^w})_{t\geq0}$ generated by the quadratic operator $q^w(x,D_x)$ associated to $q$.

When the singular space of $q$ is equal to zero, 
\begin{equation}\label{28062018E5}
	S = \{0\},
\end{equation}
the microlocal inclusion \eqref{25042018E14} implies that the semigroup $(e^{-tq^w})_{t\geq0}$ is smoothing in the Schwartz space $\mathscr S(\mathbb R^n)$, 
$$\forall u\in L^2(\mathbb R^n), \forall t>0,\quad e^{-tq^w}u\in\mathscr S(\mathbb R^n).$$
However, this result does not provide any control of the Schwartz seminorms for small times and does not describe how they blow up as time tends to zero. In the work \cite{MR3841852}, this regularizing property was sharpened and under the assumption \eqref{28062018E5}, the semigroup $(e^{-tq^w})_{t\geq0}$ was shown to be actually smoothing for any positive time in the Gelfand-Shilov space $S^{1/2}_{1/2}(\mathbb R^n)$ and some asymptotics for the associated seminorms are given for small times $0<t\ll1$. We refer the reader to Subsection \ref{GSreg} in Appendix where the Gelfand-Shilov spaces $S^{\mu}_{\nu}(\mathbb R^n)$, with $\mu+\nu\geq1$, are defined. More precisely, \cite{MR3841852} (Proposition 4.1) states that when \eqref{28062018E5} holds, there exist some positive constants $t_0>0$ and $C_0>0$ such that for all $0\le t\le t_0$ and $u\in L^2(\mathbb R^n)$,
$$\big\Vert e^{\frac{t^{2k_0+1}}{C_0}(D_x^2+x^2)}e^{-tq^w}u\big\Vert_{L^2(\mathbb R^n)}\le C_0 \Vert u\Vert_{L^2(\mathbb R^n)},$$
where $0\le k_0\le 2n-1$ is the smallest integer such that \eqref{22062018E1} holds. From the work \cite{MR3841852} (Estimate (4.19)), this implies that there exists a positive constant $C>1$ such that for all $0<t\le t_0$, $(\alpha,\beta)\in\mathbb N^{2n}$ and $u\in L^2(\mathbb R^n)$,
\begin{equation}\label{05122017E2}
	\big\Vert x^{\alpha}\partial^{\beta}_x(e^{-tq^w}u)\big\Vert_{L^2(\mathbb R^n)}\le \frac{C^{1+\vert\alpha\vert + \vert\beta\vert}}{t^{\frac{2k_0+1}2(\vert\alpha\vert + \vert\beta\vert+2n)}}\ (\alpha!)^{\frac12}\ (\beta!)^{\frac12}\ \Vert u\Vert_{L^2(\mathbb R^n)}.
\end{equation}
By using the Sobolev embedding theorem, we notice that this result provides the existence of a positive constant $C>1$ such that for all $0<t\le t_0$, $(\alpha,\beta)\in\mathbb N^{2n}$ and $u\in L^2(\mathbb R^n)$,
$$\big\Vert x^{\alpha}\partial^{\beta}_x(e^{-tq^w}u)\big\Vert_{L^{\infty}(\mathbb R^n)}\le \frac{C^{1+\vert\alpha\vert + \vert\beta\vert}}{t^{\frac{2k_0+1}2(\vert\alpha\vert + \vert\beta\vert + 2n+s)}}\ (\alpha!)^{\frac12}\ (\beta!)^{\frac12}\ \Vert u\Vert_{L^2(\mathbb R^n)},$$
where $s>n/2$ is a fixed integer, see \cite{MR3841852} (Theorem 1.2).

More generally, when the singular space $S$ of $q$ is possibly non-zero but still has a symplectic structure, that is, when the restriction of the canonical symplectic form to the singular space $\sigma_{\vert S}$ is non-degenerate, the above result \eqref{05122017E2} can be easily extended but only when differentiating the semigroup in the directions of the phase space given by the symplectic orthogonal complement of the singular space 
$$S^{\sigma\perp} = \big\{X\in\mathbb R^{2n} : \forall Y\in S,\quad \sigma(X,Y) = 0\big\}.$$
Indeed, when the singular space $S$ has a symplectic structure, it is proven in \cite{MR3710672} (Subsection 2.5) that the quadratic form $q$ writes as $q = q_1 + q_2$ with $q_1$ a purely imaginary-valued quadratic form defined on $S$ and $q_2$ another one defined on $S^{\sigma\perp}$ with a non-negative real part and a zero singular space. The symplectic structures of $S$ and $S^{\sigma\perp}$ imply that the operators $q^w_1(x,D_x)$ and $q^w_2(x,D_x)$ do commute as well as their associated semigroups
$$\forall t>0,\quad e^{-tq^w} = e^{-tq^w_1}e^{-tq^w_2} = e^{-tq^w_2}e^{-tq^w_1}.$$
Moreover, since $\Reelle q_1=0$, $(e^{-tq^w_1})_{t\geq0}$ is a contraction semigroup on $L^2(\mathbb R^n)$ and the partial smoothing properties of the semigroup $(e^{-tq^w})_{t\geq0}$ can be deduced from a symplectic change of variables and the result known for zero singular space can be applied to the semigroup $(e^{-tq^w_2})_{t\geq0}$. We refer the reader to \cite{MR3710672} (Subsection 2.5) for more details about the reduction by tensorization of the non-zero symplectic case to the case when the singular space is zero.

\begin{ex}\label{ex5} We consider the Kramers-Fokker-Planck operator acting on $L^2(\mathbb R^{2n}_{x,v})$,
\begin{equation}\label{13122017E1}
	K = -\Delta_v + \frac14\vert v\vert^2 +\langle v,\nabla_x\rangle - \langle\nabla_x V(x),\nabla_v\rangle,\quad (x,v)\in\mathbb R^{2n},
\end{equation}
with a quadratic external potential
\begin{equation}\label{08052018E1}
	V(x) = \frac12a\vert x\vert^2,\quad a\in\mathbb R,
\end{equation}
where $\vert\cdot\vert$ denotes the Euclidean norm on $\mathbb R^n$. This operator writes as $K = q^w(x,v,D_x,D_v)$, where 
\begin{equation}\label{29062018E5}
	q(x,v,\xi,\eta) = \vert\eta\vert^2 + \frac14\vert v\vert^2 + i\left(\langle v,\xi\rangle - a\langle x,\eta\rangle\right),\quad (x,v,\xi,\eta)\in\mathbb R^{4n},
\end{equation}
is a non-elliptic complex-valued quadratic form with a non-negative real part, whose Hamilton map is given by
\begin{equation}\label{29062018E7}
F = \frac12\begin{pmatrix}
	0_n & iI_n & 0_n & 0_n \\
	-aiI_n & 0_n & 0_n & 2I_n \\
	0_n & 0_n & 0_n & aiI_n \\
	0_n & -\frac12I_n & -iI_n & 0_n
\end{pmatrix}.
\end{equation}
When $a\ne0$, a simple algebraic computation shows that its singular space is 
$$S = \Ker(\Reelle F)\cap\Ker(\Reelle F(\Imag F))\cap\mathbb R^{4n} = \{0\}.$$
Therefore, the integer $0\le k_0\le 4n-1$ defined in \eqref{22062018E1} is equal to $1$ and it follows from \eqref{05122017E2} that there exist some positive constants $t_0>0$ and $C>0$ such that for all $0<t\le t_0$, $(\alpha,\beta,\gamma,\delta)\in\mathbb N^{4n}$ and $u\in L^2(\mathbb R^{2n})$,
$$\big\Vert x^{\alpha}v^{\beta}\partial^{\gamma}_x\partial^{\delta}_v(e^{-tK}u)\big\Vert_{L^2(\mathbb R^{2n})}\le \frac{C^{1+\vert\alpha\vert + \vert\beta\vert+\vert\gamma\vert+\vert\delta\vert}}{t^{\frac32(\vert\alpha\vert + \vert\beta\vert+\vert\gamma\vert+\vert\delta\vert)}}\ (\alpha!)^{\frac12}\ (\beta!)^{\frac12}\ (\gamma!)^{\frac12}\ (\delta!)^{\frac12}\ \Vert u\Vert_{L^2(\mathbb R^{2n})}.$$
When $a=0$, the singular space of $q$ is
\begin{equation}\label{11122017E4}
	S = \Ker(\Reelle F)\cap\Ker\left(\Reelle F(\Imag F)\right)\cap\mathbb R^{4n} = \mathbb R^n_x\times\{0_{\mathbb R^n_v}\}\times\{0_{\mathbb R^n_{\xi}}\}\times\{0_{\mathbb R^n_{\eta}}\},
\end{equation}
and the integer $0\le k_0\le 4n-1$ defined in \eqref{22062018E1} is also equal to $1$. In particular, when $a=0$, the singular space $S$ of $q$ has not a symplectic structure since its symplectic orthogonal complement is given by 
$$S^{\sigma\perp} = \mathbb R^n_x\times\mathbb R^n_v\times\{0_{\mathbb R^n_{\xi}}\}\times\mathbb R^n_{\eta},$$ 
and no general theory nor known regularizing results apply for the Kramers-Fokker-Planck semigroup $(e^{-tK})_{t\geq0}$.
\end{ex}

In the present work, we consider quadratic operators $q^w(x,D_x)$ whose symbols are complex-valued quadratic forms $q:\mathbb R^n_x\times\mathbb R^n_{\xi}\rightarrow\mathbb C$ with a non-negative real part $\Reelle q\geq0$ and a singular space $S$ spanned by elements of the canonical basis of $\mathbb R^{2n}$ which can also possibly fail to be symplectic, as the Kramers-Fokker-Planck operator without external potential (case $a=0$). In these degenerate cases when $S\ne\{0\}$, with $S$ possibly non-symplectic, we cannot expect that the semigroup $(e^{-tq^w})_{t\geq0}$ enjoys Gelfand-Shilov smoothing properties in all variables as in \eqref{05122017E2} and we aim in the first part of this work at studying in which specific directions of the phase space the semigroup does enjoy partial Gelfand-Shilov smoothing properties. 

To describe the regularizing effects of the semigroup $(e^{-tq^w})_{t\geq0}$ when the singular space $S$ is spanned by elements of the canonical basis of $\mathbb R^{2n}$, we need to introduce the following notation:

\begin{dfn}\label{1} Let $n\geq1$ be a positive integer, $J$ be a subset of $\{1,\ldots,n\}$ and $E$ be a subset of $\mathbb C$. We define $E^n_J$ as the subset of $E^n$ whose elements $x\in E^n_J$ satisfy 
$$\forall j\notin J,\quad x_j = 0.$$
By convention, we set $E^n_J=\{0\}$ when $J$ is empty.
\end{dfn}

The main result of this article is the following:

\begin{thm}\label{20112017T1} Let $q:\mathbb R^n_x\times\mathbb R^n_{\xi}\rightarrow\mathbb C$ be a complex-valued quadratic form with a non-negative real part $\Reelle q\geq0$. We assume that there exist some subsets $I,J\subset\{1,\ldots,n\}$ such that the singular space $S$ of $q$ satisfies $S^{\perp} = \mathbb R^n_I\times\mathbb R^n_J$, the orthogonality being taken with respect to the canonical Euclidean structure of $\mathbb R^{2n}$. We also assume that the inclusion $S\subset\Ker(\Imag F)$ holds, where $F$ denotes the Hamilton map of $q$. Then, there exist some positive constants $C>1$ and $0<t_0<1$ such that for all $(\alpha,\beta)\in\mathbb N^n_I\times\mathbb N^n_J$, $0<t\le t_0$ and $u\in L^2(\mathbb R^n)$,
$$\big\Vert x^{\alpha}\partial^{\beta}_x(e^{-tq^w}u)\big\Vert_{L^2(\mathbb R^n)}
\le\frac{C^{1+\vert\alpha\vert+\vert\beta\vert}}{t^{(2k_0+1)(\vert\alpha\vert + \vert\beta\vert + s)}}\ (\alpha!)^{\frac12}\ (\beta!)^{\frac12}\ \Vert u\Vert_{L^2(\mathbb R^n)},$$
where $0\le k_0\le 2n-1$ is the smallest integer such that \eqref{22062018E1} holds and $s = 9n/4+2\lfloor n/2\rfloor + 3$.
\end{thm}

In all this work, $\lfloor\cdot\rfloor$ stands for the floor function. Moreover, we denote
$$\vert\alpha\vert = \sum_{i\in I}\alpha_i,\quad \vert\beta\vert = \sum_{j\in J}\beta_j,\quad \alpha! = \prod_{i\in I}\alpha_i!,\quad\beta! = \prod_{j\in J}\beta_j!,$$
for all $I,J\subset\{1,\ldots,n\}$ and $(\alpha,\beta)\in\mathbb N^n_I\times\mathbb N^n_J$, with the convention that a sum taken over the empty set is equal to $0$, and a product taken over the empty set is equal to $1$.

Theorem \ref{20112017T1} shows that once the singular space $S$ of $q$ is spanned by elements of the canonical basis of $\mathbb R^{2n}$ and satisfies the algebraic condition $S\subset\Ker(\Imag F)$, with $F$ the Hamilton map of $q$, the semigroup $(e^{-tq^w})_{t\geq0}$ enjoys partial Gelfand-Shilov smoothing properties, with a control in $\mathcal O(t^{-(2k_0+1)(\vert\alpha\vert+\vert\beta\vert+s)})$ of the seminorms as $t\rightarrow0^+$, where $s = 9n/4+2\lfloor n/2\rfloor + 3$. The power $(2k_0+1)(\vert\alpha\vert+\vert\beta\vert+s)$ is not expected to be sharp. It would be interesting to understand if the upper bound in Theorem \ref{20112017T1} may be sharpened in $\mathcal O(t^{-\frac{2k_0+1}2(\vert\alpha\vert+\vert\beta\vert+2n)})$ as in the estimate \eqref{05122017E2}.

\begin{ex} Theorem \ref{20112017T1} applies in particular for quadratic operators $q^w(x,D_x)$ associated to complex-valued quadratic forms with non-negative real parts and zero singular spaces $S = \{0\}$. The result of Theorem \ref{20112017T1} allows to recover the Gelfand-Shilov regularizing properties of the associated semigroup $(e^{-tq^w})_{t\geq0}$ for small times given by \eqref{05122017E2}, up to the power of the time variable $t$ which is less precise in the above statement.
\end{ex}

Let us state the fact that Theorem \ref{20112017T1} applies for quadratic operators with non-symplectic singular spaces:

\begin{ex}\label{ex1} Let $K$ be the Kramers-Fokker-Planck operator without external potential:
$$K = -\Delta_v + \frac14 \vert v\vert^2 +\langle v,\nabla_x\rangle,\quad (x,v)\in\mathbb R^{2n}.$$
We recall from \eqref{29062018E5} that the Weyl symbol of $K$ is the quadratic form $q$ defined by
\begin{equation}\label{29062018E6}
	q(x,v,\xi,\eta) = \vert\eta\vert^2 + \frac14\vert v\vert^2 + i\langle v,\xi\rangle,\quad (x,v,\xi,\eta)\in\mathbb R^{4n}.
\end{equation}
Moreover, the Hamilton map and the singular space of $q$ are respectively given from \eqref{29062018E7} and \eqref{11122017E4} by
$$F = \frac12\begin{pmatrix}
	0_n & iI_n & 0_n & 0_n \\
	0_n & 0_n & 0_n & 2I_n \\
	0_n & 0_n & 0_n & 0_n \\
	0_n & -\frac12I_n & -iI_n & 0_n
\end{pmatrix},$$
and
$$S = \Ker(\Reelle F)\cap\Ker\left(\Reelle F(\Imag F)\right)\cap\mathbb R^{4n} = \mathbb R^n_x\times\{0_{\mathbb R^n_v}\}\times\{0_{\mathbb R^n_{\xi}}\}\times\{0_{\mathbb R^n_{\eta}}\}.$$
Notice that here, the singular space has not a symplectic structure. Since $S^{\perp} = \mathbb R^{2n}_I\times\mathbb R^{2n}_J$, with $I = \{n+1,\ldots,2n\}$ and $J = \{1,\ldots,2n\}$, the orthogonality being taken with respect to the canonical Euclidean structure of $\mathbb R^{2n}$, and that the inclusion $S\subset\Ker(\Imag F)$ holds, Theorem \ref{20112017T1} shows that there exist some positive constants $C>1$ and $0<t_0<1$ such that for all $(\alpha,\beta,\gamma)\in\mathbb N^{3n}$, $0<t\le t_0$, and $u\in L^2(\mathbb R^{2n})$,
\begin{equation*}
	\big\Vert v^{\alpha}\partial^{\beta}_x\partial^{\gamma}_v(e^{-tK}u)\big\Vert_{L^2(\mathbb R^{2n})}
	\le\frac{C^{1+\vert\alpha\vert+\vert\beta\vert + \vert\gamma\vert}}{t^{3(\vert\alpha\vert + \vert\beta\vert + \vert\gamma\vert + (13n)/2+3)}}\ (\alpha!)^{\frac12}\ (\beta!)^{\frac12}\ (\gamma!)^{\frac12}\ \Vert u\Vert_{L^2(\mathbb R^{2n})}.
\end{equation*}
We refer the reader to Section 5 for an extension of this result to generalized Ornstein-Uhlenbeck operators.
\end{ex}

It is still an open question to know if the algebraic condition $S\subset\Ker(\Imag F)$ on the Hamilton map $F$ and the singular space $S$ of $q$ in Theorem \ref{20112017T1} can be weakened or simply removed. However, as pointed out by the following particular example, there exists a class of complex-valued quadratic forms $q:\mathbb R^n_x\times\mathbb R^n_{\xi}\rightarrow\mathbb C$ with non-negative real parts $\Reelle q\geq0$ such that the result of Theorem \ref{20112017T1} holds with a sharp upper-bound as in \eqref{05122017E2} even if the assumption $S\subset\Ker(\Imag F)$ fails.

\begin{ex} We consider the complex-valued quadratic form $q:\mathbb R^n_x\times\mathbb R^n_{\xi}\rightarrow\mathbb C$ defined by
\begin{equation}\label{27062018E3}
	q(x,\xi) = \frac12\vert Q^{\frac12}\xi\vert^2 - i\langle Bx,\xi\rangle,\quad (x,\xi)\in\mathbb R^{2n},
\end{equation}
where $B$ and $Q$ are real $n\times n$ matrices, with $Q$ symmetric positive semidefinite, $B$ and $Q^{\frac12}$ satisfying an algebraic condition called the Kalman rank condition. We refer the reader to Section \ref{GOU} for the definition of the Kalman rank condition and the calculus of the Hamilton map and the singular space of $q$ respectively given by 
\begin{equation}\label{17072018E1}
F = \frac12\begin{pmatrix}
	-iB & Q \\
	0 & iB^T
\end{pmatrix}\quad \text{and}\quad S = \mathbb R^n\times\{0\}.
\end{equation}
Notice that $S^{\perp} = \mathbb R^n_I\times\mathbb R^n_J$, with $I = \emptyset$ and $J=\{1,\ldots,n\}$, the orthogonality being taken with respect to the canonical Euclidean structure of $\mathbb R^{2n}$, and that the inclusion $S\subset\Ker(\Imag F)$ holds if and only if $\mathbb R^n\subset\Ker B$. Therefore, the inclusion $S\subset\Ker(\Imag F)$ does not hold when $B\ne0$. However, by explicitly computing $e^{-tq^w}u$ for all $t\geq0$ and $u\in L^2(\mathbb R^n)$ and exploiting the Kalman rank condition, J. Bernier and the author proved in \cite{AB} (Theorem 1.2) that there exist some positive constants $C>1$ and $t_0>0$ such that for all $\beta\in\mathbb N^n$, $0<t\le t_0$ and $u\in L^2(\mathbb R^n)$,
\begin{equation}\label{17072018E2}
	\big\Vert\partial^{\beta}_x(e^{-tq^w}u)\big\Vert_{L^2(\mathbb R^n)}
\le \frac{C^{1+\vert\beta\vert}}{t^{\frac{2k_0+1}{2}\vert\beta\vert}}\ (\beta!)^{\frac12}\ \Vert u\Vert_{L^2(\mathbb R^n)},
\end{equation}
where $0\le k_0\le 2n-1$ is the smallest integer such that \eqref{22062018E1} holds. Moreover, the estimates \eqref{17072018E2} show that for all $(\alpha,\beta)\in\mathbb N^n_I\times\mathbb N^n_J$, $0<t\le t_0$ and $u\in L^2(\mathbb R^n)$,
$$\big\Vert x^{\alpha}\partial^{\beta}_x(e^{-tq^w}u)\big\Vert_{L^2(\mathbb R^n)}
\le \frac{C^{1+\vert\alpha\vert+\vert\beta\vert}}{t^{\frac{2k_0+1}2(\vert\alpha\vert+\vert\beta\vert)}}\ (\alpha!)^{\frac12}\ (\beta!)^{\frac12}\ \Vert u\Vert_{L^2(\mathbb R^n)},$$
since $\mathbb N^n_I = \{0\}$ and $\mathbb N^n_J = \mathbb N^n$.
\end{ex}

\subsection{Null-controllability of parabolic equations associated with frequency-hypoelliptic accretive quadratic operators} As an application of the smoothing result provided by Theorem \ref{20112017T1}, we study the null-controllability of parabolic equations posed on the whole Euclidean space associated with a general class of accretive quadratic operators
\begin{equation}\label{08122017E1}
\left\{\begin{aligned}
	& \partial_tf(t,x) + q^w(x,D_x)f(t,x) = u(t,x)\mathbbm 1_{\omega}(x),\quad (t,x)\in(0,+\infty)\times\mathbb R^n, \\
	& f(0) = f_0\in L^2(\mathbb R^n),
\end{aligned}\right.
\end{equation}
where $\omega\subset\mathbb R^n$ is a Borel subset with positive Lebesgue measure, $\mathbbm 1_{\omega}$ is its characteristic function, and $q^w(x,D_x)$ is an accretive quadratic operator whose Weyl symbol is a complex-valued quadratic form $q:\mathbb R^n_x\times\mathbb R^n_{\xi}\rightarrow\mathbb C$ with a non-negative real part $\Reelle q\geq0$.

\begin{dfn}[Null-controllability] Let $T>0$ and $\omega$ be a Borel subset of $\mathbb R^n$ with positive Lebesgue measure. Equation \eqref{08122017E1} is said to be null-controllable from the set $\omega$ in time $T$ if, for any initial datum $f_0\in L^2(\mathbb R^n)$, there exists $u\in L^2((0,T)\times\mathbb R^n)$, supported in $(0,T)\times\omega$, such that the mild solution of \eqref{08122017E1} satisfies $f(T,\cdot) = 0$.
\end{dfn}

When the singular space of $q$ is reduced to zero $S = \{0\}$, K. Beauchard, P. Jaming and K. Pravda-Starov proved in the recent work \cite{BJKPS} (Theorem 2.2) that the parabolic equation \eqref{08122017E1} associated to $q^w(x,D_x)$ is null-controllable in any positive time $T>0$, once the control subset $\omega\subset\mathbb R^n$ is thick. The thickness of a subset of $\mathbb R^n$ is defined as follows:

\begin{dfn}\label{10092018D1} Let $\gamma\in(0,1]$ and $a = (a_1,\ldots,a_n)\in(\mathbb R^*_+)^n$. Let $P = [0,a_1]\times\ldots\times[0,a_n]\subset\mathbb R^n$. A subset $\omega\subset\mathbb R^n$ is called $(\gamma,a)$-thick if it is measurable and 
$$\forall x\in\mathbb R^n,\quad \vert\omega\cap(x+P)\vert\geq\gamma\prod_{j=1}^na_j,$$
where $\vert\omega\cap(x+P)\vert$ stands for the Lebesgue measure of the set $\omega\cap(x+P)$. A subset $\omega\subset\mathbb R^n$ is called thick if there exist $\gamma\in(0,1]$ and $a\in(\mathbb R^*_+)^n$ such that $\omega$ is $(\gamma,a)$-thick.
\end{dfn}

No general result of null-controllability for the equation \eqref{08122017E1} is known up to now when the singular space of $q$ is non-zero.  However, when the quadratic form $q$ is defined by \eqref{27062018E3}, where $B$ and $Q$ are real $n\times n$ matrices, with $Q$ symmetric positive semidefinite, $B$ and $Q^{\frac12}$ satisfying the Kalman rank condition, we recall that the singular space of $q$ is $S=\mathbb R^n\times\{0\}$ (in particular, $S$ is non-zero), and J. Bernier and the author proved in \cite{AB} (Theorem 1.8) that the parabolic equation \eqref{08122017E1} is null-controllable in any positive time from thick control subsets. Moreover, when $B = 0_n$ and $Q=2I_n$, the quadratic form $q$ is given by $q(x,\xi) = \vert\xi\vert^2$ and \eqref{08122017E1} is the heat equation posed on the whole space :
\begin{equation}\label{25042018E15}
\left\{\begin{aligned}
	& \partial_tf(t,x) - \Delta_xf(t,x) = u(t,x)\mathbbm 1_{\omega}(x),\quad (t,x)\in(0,+\infty)\times\mathbb R^n, \\
	& f(0) = f_0\in L^2(\mathbb R^n).
\end{aligned}\right.
\end{equation}
Recently, M. Egidi and I. Veselic in \cite{MR3816981} and G. Wang, M. Wang, C. Zhang and Y. Zhang in \cite{MR3950015} proved independently that the thickness of the control set $\omega$ is not only a sufficient condition for the null-controllability of the heat equation \eqref{25042018E15}, but also a necessary condition.

In this work, we investigate the sufficient geometric conditions on the singular space $S$ of $q$ which allow to obtain positive null-controllability results for the parabolic equation \eqref{08122017E1} when the control subset $\omega\subset\mathbb R^n$ is thick. The suitable class of symbols $q$ to consider is the following class of diffusive quadratic forms: 

\begin{dfn}[Diffusive quadratic form]\label{2} Let $q:\mathbb R^n_x\times\mathbb R^n_{\xi}\rightarrow\mathbb C$ be a complex-valued quadratic form. We say that $q$ is diffusive if there exists a subset $I\subset\{1,\ldots,n\}$ such that $S^{\perp} = \mathbb R^n_I\times\mathbb R^n_{\xi}$, the orthogonality being taken with respect to the canonical Euclidean structure of $\mathbb R^{2n}$.
\end{dfn}

\begin{ex} Any complex-valued quadratic form $q:\mathbb R^n_x\times\mathbb R^n_{\xi}\rightarrow\mathbb C$ whose singular space is equal to zero $S = \{0\}$ is diffusive, since $S^{\perp} = \mathbb R^n_I\times\mathbb R^n_{\xi}$, with $I = \{1,\ldots,n\}$.
\end{ex}

\begin{ex}\label{exKFP} Let $q:\mathbb R^{2n}_x\times\mathbb R^{2n}_{\xi}\rightarrow\mathbb C$ be the complex-valued quadratic form defined by \eqref{29062018E6}. We recall from \eqref{11122017E4} that the singular space of $q$ is 
$$S = \mathbb R^n_x\times\{0_{\mathbb R^n_v}\}\times\{0_{\mathbb R^n_{\xi}}\}\times\{0_{\mathbb R^n_{\eta}}\}.$$
Therefore, $S^{\perp} = \mathbb R^{2n}_I\times\mathbb R^{2n}_{\xi,\eta}$, with $I = \{n+1,\ldots,2n\}$, which proves that $q$ is diffusive.
\end{ex}

It follows from Theorem \ref{20112017T1} that when the quadratic form $q$ is diffusive and $S\subset\Ker(\Imag F)$, where $F$ and $S$ denote respectively the Hamilton map and the singular space of $q$, the semigroup $(e^{-tq^w})_{t\geq0}$ generated by $q^w(x,D_x)$ is smoothing in the Gevrey space $G^{\frac12}(\mathbb R^n)$. More precisely, there exist some positive constants $C>0$  and $0<t_0<1$ such that for all $0<t\le t_0$, $\alpha\in\mathbb N^n$ and $u\in L^2(\mathbb R^n)$,
\begin{equation}\label{27062018E2}
	\big\Vert\partial^{\alpha}_x(e^{-tq^w}u)\big\Vert_{L^2(\mathbb R^n)}
	\le\frac{C^{1+\vert\alpha\vert}}{t^{(2k_0+1)(\vert\alpha\vert + s)}}\ (\alpha!)^{\frac12}\ \Vert u\Vert_{L^2(\mathbb R^n)},
\end{equation}
where $0\le k_0\le 2n-1$ is the smallest integer such that \eqref{22062018E1} holds and $s = 9n/4+2\lfloor n/2\rfloor+3$. This regularizing property has a key role to prove the following result on the null-controllability of the parabolic equation \eqref{08122017E1}:

\begin{thm}\label{05122017T2} Let $q:\mathbb R^n_x\times\mathbb R^n_{\xi}\rightarrow\mathbb C$ be a complex-valued quadratic form with a non-negative real part $\Reelle q\geq0$. We assume that $q$ is diffusive and that its singular space $S$ satisfies $S\subset\Ker(\Imag F)$, where $F$ is the Hamilton map of $q$. If $\omega\subset\mathbb R^n$ is a thick set, then the parabolic equation
$$\left\{\begin{array}{l}
	\partial_tf(t,x) + q^w(x,D_x)f(t,x) = u(t,x)\mathbbm1_{\omega}(x),\quad (t,x)\in(0,+\infty)\times\mathbb R^n, \\[5pt]
	f(0) = f_0\in L^2(\mathbb R^n),
\end{array}\right.$$
with $q^w(x,D_x)$ being the quadratic differential operator defined by the Weyl quantization of the symbol $q$, is null-controllable from the set $\omega$ in any positive time $T>0$.
\end{thm}

Theorem \ref{05122017T2} allows to consider more degenerate cases than the one when the singular space of $q$ is zero $S = \{0\}$. Indeed, when $S=\{0\}$, the estimates \eqref{05122017E2} show that the semigroup $(e^{-tq^w})_{t\geq0}$ is regularizing in any positive time in the Gelfand-Shilov space $S^{1/2}_{1/2}(\mathbb R^n)$, that is, is regularizing in the whole phase space, while when $q$ is only assumed to be diffusive and when $S\subset\Ker(\Imag F)$, with $F$ the Hamilton map of $q$, the semigroup $(e^{-tq^w})_{t\geq0}$ is only smoothing in the Gevrey space $G^{\frac12}(\mathbb R^n)$ (in the sense that the estimates \eqref{27062018E2} hold). In particular, Theorem \ref{05122017T2} extends \cite{BJKPS} (Theorem 2.2).

\begin{ex} We consider the Kramers-Fokker-Planck equation without external potential posed on the whole space 
\begin{equation}\label{29062018E8}
\left\{\begin{array}{l}
	\partial_tf(t,x,v) +Kf(t,x,v) = u(t,x,v)\mathbbm1_{\omega}(x,v),\quad (t,x,v)\in(0,+\infty)\times\mathbb R^{2n}, \\[5pt]
	f(0) = f_0\in L^2(\mathbb R^{2n}),
\end{array}\right.
\end{equation}
where the operator $K$ is defined by
$$K = - \Delta_v + \frac14\vert v\vert^2 + \langle v,\nabla_x\rangle,\quad (x,v)\in\mathbb R^{2n}.$$
As noticed in Example \ref{ex1}, the Hamilton map $F$ and the singular space $S$ of the quadratic form $q:\mathbb R^{2n}_{x,v}\times\mathbb R^{2n}_{\xi,\eta}\rightarrow\mathbb C$ defined in \eqref{29062018E6}, Weyl symbol of the operator $K$, satisfy the condition $S\subset\Ker(\Imag F)$. Moreover, the quadratic form $q$ is diffusive according to Example \ref{exKFP}. It therefore follows from Theorem \ref{05122017T2} that the equation \eqref{29062018E8} is null-controllable from any thick control subset $\omega\subset\mathbb R^{2n}$ in any positive time $T>0$.
As above, we refer to Section 5 for a generalization of this result to generalized Ornstein-Uhlenbeck equations.
\end{ex}

By the Hilbert Uniqueness Method, see \cite{MR2302744} (Theorem 2.44), the null-controllability of the equation \eqref{08122017E1} is equivalent to the observability of the adjoint system
\begin{equation}\label{27062018E4}
\left\{\begin{aligned}
	& \partial_tg(t,x) + (q^w(x,D_x))^*g(t,x) = 0,\quad (t,x)\in(0,+\infty)\times\mathbb R^n, \\
	& g(0) = g_0\in L^2(\mathbb R^n).
\end{aligned}\right.
\end{equation}
We recall the definition of the notion of observability:

\begin{dfn}[Observability]\label{3} Let $T>0$ and $\omega$ be a Borel subset of $\mathbb R^n$ with positive Lebesgue measure. Equation \eqref{27062018E4} is said to be observable from the set $\omega$ in time $T$ if there exists a constant $C_T>0$ such that, for any initial datum $g_0\in L^2(\mathbb R^n)$, the mild solution of \eqref{27062018E4} satisfies
\begin{equation}\label{05122017E1}
	\Vert g(T,x)\Vert^2_{L^2(\mathbb R^n)}\le C_T\int_0^T\Vert g(t,x)\Vert^2_{L^2(\omega)}dt.
\end{equation}
\end{dfn}

The $L^2(\mathbb R^n)$-adjoint of the quadratic operator $(q^w(x,D_x),D(q^w))$ is given by the quadratic operator $(\overline q^w(x,D_x),D(\overline q^w))$, whose Weyl symbol is the complex conjugate of the symbol $q$. Moreover, the Hamilton map of $\overline q$ is $\overline F$, where $F$ is the Hamilton map of $q$. This implies that $q$ and $\overline q$ do have the same singular space. As a consequence, the assumptions of Theorem \ref{05122017T2} hold for the quadratic operator $q^w(x,D_x)$ if and only if they hold for its $L^2(\mathbb R^n)$-adjoint operator $\overline q^w(x,D_x)$. We deduce from the Hilbert Uniqueness Method that the result of null-controllability given by Theorem \ref{05122017T2} is therefore equivalent to the following observability estimate: 

\begin{thm}\label{08122017T1} Let $q:\mathbb R^n_x\times\mathbb R^n_{\xi}\rightarrow\mathbb C$ be a complex-valued quadratic form with a non-negative real part $\Reelle q\geq0$. We assume that $q$ is diffusive and that its singular space $S$ satisfies $S\subset\Ker(\Imag F)$, where $F$ is the Hamilton map of $q$. If $\omega\subset\mathbb R^n$ is a thick subset, there exists a positive constant $C>1$ such that for all $T>0$ and $g\in L^2(\mathbb R^n)$,
\begin{equation}\label{08122017E9}
	\big\Vert e^{-Tq^w}g\big\Vert^2_{L^2(\mathbb R^n)}\le C\exp\left(\frac C{T^{2(2k_0+1)}}\right)\int_0^T\big\Vert e^{-tq^w}g\big\Vert^2_{L^2(\omega)}\ dt,
\end{equation}
where $0\le k_0\le 2n-1$ is the smallest integer such that \eqref{22062018E1} holds.
\end{thm}

As the results of Theorems \ref{05122017T2} and \ref{08122017T1} are equivalent, we only need to prove Theorem \ref{08122017T1}, and the proof of this observability estimate is based on the regularizing effect \eqref{27062018E2} while using a Lebeau-Robbiano strategy.

\subsubsection{Outline of the work} In Section \ref{regeffects}, we study a family of time-dependent pseudodifferential operators whose symbols are models of the Mehler symbols given by formula \eqref{04122017E8}. Thanks to the Mehler formula, the properties of these operators allow to prove Theorem \ref{20112017T1} in Section \ref{refquad}. The Section \ref{controllability} is devoted to the proof of Theorem \ref{08122017T1}. An application to the study of generalized Ornstein-Uhlenbeck operators is given in Section \ref{GOU}. Section \ref{Appendix} is an appendix devoted to the proofs of some technical results.

\section{Regularizing effects of time-dependent pseudodifferential operators}
\label{regeffects}

Let $T>0$ and $q_t:\mathbb R^n\times\mathbb R^n\rightarrow\mathbb C$ be a time-dependent complex-valued quadratic form whose coefficients depend continuously on the time variable $0\le t\le T$. We assume that there exist some positive constants $0<T^*<T$ and $c>0$, a positive integer $k\geq1$, and $I,J\subset\{1,\ldots,n\}$ such that
\begin{equation}\label{08052018E2}
	\forall t\in[0,T^*], \forall X\in\mathbb R^n_I\times\mathbb R^n_J,\quad (\Reelle q_t)(X)\geq ct^k\vert X\vert^2,
\end{equation}
and 
\begin{equation}\label{19042018E5}
	\forall t\in[0,T^*], \forall X\in\mathbb R^n\times\mathbb R^n,\quad q_t(X) = q_t(X_{I,J}),
\end{equation}
where $X_{I,J}$ stands for the component in $\mathbb R^n_I\times\mathbb R^n_J$ of the vector $X\in\mathbb R^n\times\mathbb R^n$ according to the decomposition $\mathbb R^n\times\mathbb R^n = (\mathbb R^n_I\times\mathbb R^n_J)\oplus^{\perp}(\mathbb R^n_I\times\mathbb R^n_J)^{\perp}$, the orthogonality being taken with respect to the canonical Euclidean structure of $\mathbb R^n\times\mathbb R^n$, and where the notation $\mathbb R^n_I\times\mathbb R^n_J$ is defined in Definition \ref{1}. This section is devoted to the study of the regularizing effects of the pseudodifferential operators $(e^{-q_t})^w$ acting on $L^2(\mathbb R^n)$ defined by the Weyl quantization of the symbols $e^{-q_t}$. The main result of this section is the following:

\begin{thm}\label{18042018T1} Let $T>0$ and $q_t:\mathbb R^n\times\mathbb R^n\rightarrow\mathbb C$ be a time-dependent complex-valued quadratic form satisfying \eqref{08052018E2} and \eqref{19042018E5}, and whose coefficients depend continuously on the time variable $0\le t\le T$. Then, there exist some positive constants $C>1$ and $0<t_0<\min(1,T^*)$ such that for all $(\alpha,\beta)\in\mathbb N^n_I\times\mathbb N^n_J$, $0< t\le t_0$ and $u\in L^2(\mathbb R^n)$, 
$$\big\Vert x^{\alpha}\partial^{\beta}_x(e^{-q_t})^wu\big\Vert_{L^2(\mathbb R^n)}\le\frac{C^{1+\vert\alpha\vert+\vert\beta\vert}}{t^{k(\vert\alpha\vert+\vert\beta\vert+s)}}\ (\alpha!)^{\frac12}\ (\beta!)^{\frac12}\ \Vert u\Vert_{L^2(\mathbb R^n)},$$
where $0<T^*<T$, $k\geq1$, $I,J\subset\{1,\ldots,n\}$ are defined in \eqref{08052018E2} and \eqref{19042018E5}, and $s = 9n/4+2\lfloor n/2\rfloor + 3$.
\end{thm}

The proof of Theorem \ref{18042018T1} is based on symbolic calculus. In the following of this section, we study the Weyl symbol of the operator $x^{\alpha}\partial^{\beta}_x(e^{-q_t})^w$, that is, the symbol $x^{\alpha}\ \sharp\ (i\xi)^{\beta}\ \sharp\ e^{-q_t(x,\xi)}$, where $\sharp$ denotes the Moyal product defined for all $a,b$ in proper symbol classes by
$$(a\ \sharp\ b)(x,\xi) = \left.e^{\frac{i}{2}\sigma(D_x,D_{\xi};D_y,D_{\eta})}a(x,\xi)\ b(y,\eta)\right\vert_{(x,\xi) = (y,\eta)},$$
see e.g. (18.5.6) in \cite{MR2304165}.

\subsection{Gelfand-Shilov type estimates} In this first subsection, we study the properties of the time-dependent symbol $e^{-q_t}$, where $q_t:\mathbb R^N\rightarrow\mathbb C$ is a time-dependent complex-valued quadratic form whose coefficients are continuous functions of the time variable $0\le t\le T^*$, with $T^*>0$, satisfying that there exist a positive constant $c>0$ and a positive integer $k\geq1$ such that
\begin{equation}\label{11092017E1}
	\forall t\in[0,T^*],\forall X\in\mathbb R^N,\quad (\Reelle q_t)(X) \geq c t^k\vert X\vert^2.
\end{equation}

\begin{lem}\label{18042018L2} Let $T>0$ and $q_t:\mathbb R^N\rightarrow\mathbb C$ be a time-dependent complex-valued quadratic form satisfying \eqref{11092017E1}, and whose coefficients depend continuously on the time variable $0\le t\le T$. Then, there exist some positive constants $0<t_0<\min(1,T^*)$ and $c_1, c_2>0$ such that the Fourier transform of the symbol $e^{-q_t}$ satisfies the estimates
$$\forall t\in(0,t_0),\forall\Xi\in\mathbb R^N,\quad\big\vert\widehat{e^{-q_t}}(\Xi)\big\vert
\le \frac{c_1}{t^{\frac{kN}2}}e^{-c_2t^{2k}\vert\Xi\vert^2},$$
where $0<T^*<T$ and $k\geq1$ are defined in \eqref{11092017E1}.
\end{lem}

\begin{proof} Let $0<t\le T^*$. Since $\Reelle q_t$ satisfies \eqref{11092017E1}, the spectral theorem allows to diagonalize $\Imag q_t$ with respect to $\Reelle q_t$. More precisely, there exist $(e_{1,t},\ldots,e_{N,t})$ a basis of $\mathbb R^N$ and $\lambda_{1,t},\ldots,\lambda_{N,t}\in\mathbb R$ some real numbers satisfying that for all $1\le j,k\le N$,
\begin{equation}\label{23042018E4}
	(\Reelle q_t)(e_{j,t},e_{k,t}) = \delta_{j,k}\quad \text{and}\quad (\Imag q_t)(e_{j,t},e_{k,t}) = \lambda_{j,t}\delta_{j,k},
\end{equation}
with $\Reelle q_t(\cdot,\cdot)$ and $\Imag q_t(\cdot,\cdot)$ the polarized forms associated to the quadratic forms $\Reelle q_t$ and $\Imag q_t$ respectively, and where $\delta_{j,k}$ denotes the Kronecker delta. Let $P_t\in\GL_N(\mathbb R)$ be the matrix associated to the change of basis mapping the canonical basis of $\mathbb R^N$ to $(e_{1,t},\ldots,e_{N,t})$. Moreover, we consider $D_t\in M_N(\mathbb R)$ the diagonal matrix given by
\begin{equation}\label{19042018E1}
	D_t = \Diag(\lambda_{1,t},\ldots,\lambda_{N,t}),
\end{equation}
and $S_t\in M_N(\mathbb C)$ the complex matrix defined by
\begin{equation}\label{19042018E2}
	S_t = I_N + iD_t = \Diag(1+i\lambda_{1,t},\ldots,1+i\lambda_{N,t}).
\end{equation}
With these notations, we have
\begin{equation}\label{23042018E5}
	\forall X\in\mathbb R^N,\quad (\Reelle q_t)(P_tX) = \vert X\vert^2,\quad (\Imag q_t)(P_tX) = \langle D_tX,X\rangle,
\end{equation}
and therefore, $q_t\circ P_t$ is given by
\begin{equation}\label{23042018E6}
	\forall X\in\mathbb R^N,\quad q_t(P_tX) = \langle S_tX,X\rangle.
\end{equation}
Then, we compute thanks to the substitution rule and \eqref{23042018E6} that for all $\Xi\in\mathbb R^N$,
$$\widehat{e^{-q_t}}(\Xi) = \int_{\mathbb R^N}e^{-i\langle X,\Xi\rangle}e^{-q_t(X)}dX
= \vert \det P_t\vert\int_{\mathbb R^N}e^{-i\langle P_tX,\Xi\rangle}e^{-\langle S_tX,X\rangle}dX.$$
We observe that $S_t$ is a symmetric non-singular matrix satisfying that $\Reelle S_t\geq0$. It follows from \cite{MR1996773} (Theorem 7.6.1) that for all $\Xi\in\mathbb R^N$,
\begin{equation}\label{20062018E9}
	\widehat{e^{-q_t}}(\Xi) = \vert\det P_t\vert\ \frac{\pi^{\frac N2}}{(\det S_t)^{\frac12}}\ e^{-\frac14\langle S_t^{-1}P_t^T\Xi,P_t^T\Xi\rangle},
\end{equation}
where
$$(\det S_t)^{\frac12} = \prod_{j=1}^Ne^{\frac12\Log(1+i\lambda_{j,t})},$$
with $\Log$ the principal determination of the complex logarithm in $\mathbb C\setminus\mathbb R_-$. We consider $\Delta_t\in M_N(\mathbb R)$ the real diagonal matrix defined by
\begin{equation}\label{20062018E8}
	\Delta_t = \Reelle (S_t^{-1}) = \Diag\Big(\frac1{1+\lambda_{1,t}^2},\ldots,\frac{1}{1+\lambda_{N,t}^2}\Big).
\end{equation}
Since $P_t$ is a real matrix, we have that for all $\Xi\in\mathbb R^N$,
\begin{equation}\label{20062018E10}
	\big\vert e^{-\frac14\langle S_t^{-1}P_t^T\Xi,P_t^T\Xi\rangle}\big\vert = e^{-\frac14\langle\Delta_tP_t^T\Xi,P_t^T\Xi\rangle}.
\end{equation}
Moreover, both $P_t$ and $\Delta_t$ are non-degenerate, and it follows that for all $\Xi\in\mathbb R^N$,
\begin{equation}\label{20062018E11}
	\langle\Delta_tP_t^T\Xi,P_t^T\Xi\rangle = \vert\Delta_t^{\frac12}P^T_t\Xi\vert^2 
	\geq \big[\Vert(P^T_t)^{-1}\Vert^{-1}\Vert\Delta_t^{-\frac12}\Vert^{-1}\vert\Xi\vert\big]^2
	= \big[\Vert P^{-1}_t\Vert^{-1}\Vert\Delta_t^{-\frac12}\Vert^{-1}\vert\Xi\vert\big]^2.
\end{equation}
We deduce from \eqref{20062018E9}, \eqref{20062018E10} and \eqref{20062018E11} that for all $\Xi\in\mathbb R^N$,
\begin{equation}\label{20062018E12}
	\big\vert\widehat{e^{-q_t}}(\Xi)\big\vert\le\vert\det P_t\vert\ \frac{\pi^{\frac{N}{2}}}{\vert\det S_t\vert^{\frac12}}\
e^{-\frac14\big[\Vert P^{-1}_t\Vert^{-1}\Vert\Delta_t^{-\frac12}\Vert^{-1}\big]^2\vert\Xi\vert^2}.
\end{equation}
The following of the proof consists in bounding the time-dependent terms appearing in the right hand-side of \eqref{20062018E12}, that is
$\vert\det S_t\vert$, $\vert\det P_t\vert$, $\Vert P_t^{-1}\Vert$, and $\Vert\Delta_t^{-\frac12}\Vert$. \\[5pt]
\textbf{1.} It follows from \eqref{19042018E2} that the determinant of the time-dependent matrix $S_t$ is bounded from below in the following way:
\begin{equation}\label{13072018E1}
	\forall t\in(0,T^*],\quad \vert\det S_t\vert = \prod_{j=1}^N\vert1+i\lambda_{j,t}\vert\geq1.
\end{equation}
\textbf{2.} We notice from \eqref{11092017E1} and \eqref{23042018E5} that there exists a positive constant $c>0$ such that for all $0<t\le T^*$ and $X\in\mathbb R^N$,
$$\vert X\vert^2 = (\Reelle q_t)(P_tX)\geq ct^k\vert P_tX\vert^2.$$
Consequently, we obtain the following estimates of the norms of the matrices $P_t$:
$$\forall t\in(0,T^*],\quad \Vert P_t\Vert\le \frac 1{(ct^k)^{\frac 12}}.$$
It follows that there exists a positive constant $c_0>0$ such that
\begin{equation}\label{11092018E1}
	\forall t\in(0,T^*], \forall j,k\in\{1,\ldots,N\},\quad \big\vert (P_t)_{j,k}\big\vert\le \frac{c_0}{t^{\frac k2}},
\end{equation}
with $(P_t)_{j,k}$ the coefficients of the matrix $P_t$. We therefore deduce from \eqref{11092018E1} that for all $t\in(0,T^*]$,
$$\vert\det P_t\vert \le \sum_{\tau\in\mathfrak S_N}\vert\varepsilon(\tau)\vert\prod_{j=1}^N\big\vert(P_t)_{j,\tau(j)}\big\vert
\le\sum_{\tau\in\mathfrak{S}_N}\prod_{j=1}^N\frac{c_0}{t^{\frac k2}} = N!\left[\frac{c_0}{t^{\frac k2}}\right]^N,$$
where $\mathfrak{S}_N$ denotes the symmetric group and $\varepsilon(\tau)$ is the signature of the permutation $\tau\in\mathfrak{S}_N$. Setting $c_1 = N!\ c_0^N$, we proved that for all $t\in(0,T^*]$, 
\begin{equation}\label{18042018E5}
	\vert\det P_t\vert\le \frac{c_1}{t^{\frac{kN}2}}.
\end{equation}
\textbf{3.} The continuous dependence of the coefficients of the time-dependent quadratic form $\Reelle q_t$ with respect to the time variable $0\le t\le T$ implies that there exists a positive constant $c_2>0$ such that
\begin{equation}\label{12072018E1}
	\forall t\in[0,T^*], \forall X\in\mathbb R^N,\quad (\Reelle q_t)(X)\le c_2\vert X\vert^2.
\end{equation}
It follows from \eqref{23042018E5} and \eqref{12072018E1} that
$$\forall t\in(0,T^*], \forall X\in\mathbb R^N,\quad \vert X\vert^2 = (\Reelle q_t)(P_tX) \le c_2\vert P_tX\vert^2.$$
As a consequence, we have
\begin{equation}\label{18042018E6}
	\forall t\in(0,T^*],\quad \Vert P_t^{-1}\Vert\le c_2^{\frac12}.
\end{equation}
\textbf{4.} We deduce from \eqref{11092017E1} and \eqref{23042018E4} that for all $0<t\le T^*$ and $1\le j\le N$,
\begin{equation}\label{21062018E1}
	\vert\lambda_{j,t}\vert = \vert(\Imag q_t)(e_{j,t})\vert\le \Vert\Imag q_t\Vert\vert e_{j,t}\vert^2\le \frac{\Vert\Imag q_t\Vert}{ct^k}(\Reelle q_t)(e_{j,t}) = \frac{\Vert\Imag q_t\Vert}{ct^k}.
\end{equation}
Since the coefficients of the time-dependent quadratic form $q_t$ are continuous with respect to the time variable $0\le t\le T$, there exists a positive constant $c_3>0$ such that
\begin{equation}\label{21062018E2}
	\forall t\in[0,T^*],\quad \Vert \Imag q_t\Vert\le c_3.
\end{equation}
As a consequence of \eqref{21062018E1} and \eqref{21062018E2}, the following estimates hold:
\begin{equation}\label{19042018E3}
	\forall t\in(0,T^*], \forall j\in\{1,\ldots,N\},\quad \vert \lambda_{j,t}\vert\le \frac{c_3}{ct^k}.
\end{equation}
It follows from \eqref{20062018E8} and \eqref{19042018E3} that there exist some positive constants $0<t_0<\min(1,T^*)$ and $c_4,c_5>0$ such that
\begin{equation}\label{18042018E7}
	\forall t\in(0,t_0),\quad \Vert\Delta_t^{-\frac12}\Vert\le c_4\sum_{j=1}^N(1+\lambda^2_{j,t})^{\frac12}\le \frac{c_5}{t^k}.
\end{equation}
Finally, we deduce from \eqref{20062018E12}, \eqref{13072018E1}, \eqref{18042018E5}, \eqref{18042018E6} and \eqref{18042018E7} that for all $0<t<t_0$ and $\Xi\in\mathbb R^N$,
$$\big\vert\widehat{e^{-q_t}}(\Xi)\big\vert\le\frac{\pi^{\frac N2}c_1}{t^{\frac{kN}2}} e^{-t^{2k}\vert\Xi\vert^2/(2c_2c_5^2)}.$$
This ends the proof of Lemma \ref{18042018L2}.
\end{proof}

\begin{lem}\label{19042018L1} Let $T>0$ and $q_t:\mathbb R^N\rightarrow\mathbb{C}$ be a time-dependent complex-valued quadratic form satisfying \eqref{11092017E1}, and whose coefficients depend continuously on the time variable $0\le t\le T$. Then, there exist some positive constants $0<t_0<\min(1,T^*)$ and $C>1$ such that
$$\forall\alpha,\beta\in\mathbb N^N,\forall t\in(0,t_0),\quad \big\Vert X^{\alpha}\partial^{\beta}_X(e^{-q_t(X)})\big\Vert_{L^{\infty}(\mathbb R^N)}\le \frac{C^{1+\vert\alpha\vert+\vert\beta\vert}}{t^{\frac k2(\vert\alpha\vert+2\vert\beta\vert+s)}}\ (\alpha!)^{\frac12}\ (\beta!)^{\frac12},$$
where $0<T^*<T$ and $k\geq1$ are defined in \eqref{11092017E1}, and $s = 5N/4+2\lfloor N/2\rfloor+2$.
\end{lem}

\begin{proof} First, we deduce from \eqref{11092017E1} that
\begin{equation}\label{18042018E8}
	\forall t\in[0,T^*], \forall X\in\mathbb R^N,\quad \big\vert e^{-q_t(X)}\big\vert\le e^{-ct^k\vert X\vert^2}.
\end{equation}
Moreover, it follows from Lemma \ref{18042018L2} that there exist some positive constants $0<t_0<\min(1,T^*)$ and $c_1,c_2>0$ such that
\begin{gather}\label{18042018E9}
	\forall t\in(0,t_0),\forall\Xi\in\mathbb R^n,\quad\big\vert\widehat{e^{-q_t}}(\Xi)\big\vert
	\le \frac{c_1}{t^{\frac{kN}2}}e^{-c_2t^{2k}\vert\Xi\vert^2}.
\end{gather}
We can assume without lost of generality that $ct_0^k<1$ and $c_2t_0^k<1$ so that $ct^k\in(0,1)$ and $c_2t^k\in(0,1)$ for all $0<t<t_0$. It follows from \eqref{18042018E8}, \eqref{18042018E9} and Proposition \ref{28032018P1} that there exists a positive constant $C>1$ such that for all $\alpha,\beta\in\mathbb N^N$ and $0<t<t_0$,
\begin{multline*}
	\big\Vert X^{\alpha}\partial^{\beta}_X(e^{-q_t(X)})\big\Vert_{L^{\infty}(\mathbb R^N)} \le C^{1+\vert\alpha\vert+\vert\beta\vert}
	\left[\frac1{(ct^k)^{\vert\alpha\vert+\frac N4}}\frac{c_1}{t^{\frac{kN}2}}\frac{1}{(c_2t^{2k})^{\vert\beta\vert+\frac N4+\lfloor\frac N2\rfloor+1}}\ \alpha!\ \beta!\right]^{\frac12} \\[5pt]
	\le \frac{c_1^{\frac12}\ C^{1+\vert\alpha\vert+\vert\beta\vert}}{c^{\frac12(\vert\alpha\vert + \frac N4)}c_2^{\frac12(\vert\beta\vert+\frac N4+\lfloor\frac N2\rfloor+1)}}
	\frac1{t^{\frac k2(\vert\alpha\vert+2\vert\beta\vert+\frac{5N}4+2\lfloor\frac N2\rfloor+2)}}\ (\alpha!)^{\frac12}\ (\beta!)^{\frac12}.
\end{multline*}
This ends the proof of Lemma \ref{19042018L1}.
\end{proof}

\subsection{The composition formula} This subsection is devoted to the  computation of the derivatives of the symbol $x^{\alpha}\ \sharp\ \xi^{\beta}\ \sharp\ e^{-q_t(x,\xi)}$
by using standard features of the Weyl calculus together with the Leibniz formula, where $(\alpha,\beta)\in\mathbb N^{2n}$, and $q_t:\mathbb R^n_x\times\mathbb R^n_{\xi}\rightarrow\mathbb C$ is a time-dependent complex-valued quadratic form, with $T>0$, satisfying \eqref{08052018E2} and \eqref{19042018E5}, and whose coefficients depend continuously on the time variable $0\le t\le T$.

\begin{lem}\label{20092017L1} Let $T>0$ and $q_t:\mathbb R^n_x\times\mathbb R^n_{\xi}\rightarrow\mathbb C$ be a time-dependent complex-valued quadratic form satisfying \eqref{08052018E2} and \eqref{19042018E5}, and whose coefficients depend continuously on the time variable $0\le t\le T$. For all $(a,b), (\alpha,\beta)\in\mathbb N^n\times\mathbb N^n$, we consider the finite subset $E_{a,b,\alpha,\beta}\subset\mathbb N^{5n}$ whose elements $(\eta,\rho,\gamma,a',b')\in E_{a,b,\alpha,\beta}$ satisfy 
\begin{equation}\label{20042018E11}
\begin{array}{ll}
	0\le a'\le a, & 0\le b'\le b, \\[5pt]
	\gamma+a'+\rho \le \alpha, & \gamma+b'+\eta \le \beta.
\end{array}
\end{equation}
Then, we have that for all $(a,b), (\alpha,\beta)\in\mathbb N^n\times\mathbb N^n$, $0< t\le T^*$ and $(x,\xi)\in\mathbb R^{2n}$,
\begin{multline*}
	\big\vert\partial^a_x\partial^b_{\xi}\big(x^{\alpha}\ \sharp\ \xi^{\beta}\ \sharp\ e^{-q_t(x,\xi)}\big)\big\vert \\[5pt]
	\le 2^{\vert a\vert + \vert b\vert}\sum_{(\eta,\rho,\gamma,a',b')\in E_{a,b,\alpha,\beta}}\Lambda_{\eta,\rho,\gamma,a',b'}\ \big\vert x^{\alpha-\gamma-a'-\rho}\xi^{\beta-\gamma-b'-\eta}\partial^{a-a'+\eta}_x\partial^{b-b'+\rho}_{\xi}\big(e^{-q_t(x,\xi)}\big)\big\vert,
\end{multline*}
where $0<T^*<T$ is defined in \eqref{08052018E2} and \eqref{19042018E5} and the $\Lambda_{\eta,\rho,\gamma,a',b'}>0$ are given by
$$\Lambda_{\eta,\rho,\gamma,a',b'} = \frac{\alpha!}{(\alpha-\gamma-a'-\rho)!}\frac{\beta!}{(\beta-\gamma-b'-\eta)!}\frac{1}{\eta!\ \rho!\ \gamma!}.$$
\end{lem}

\begin{proof} Let $(\alpha,\beta)\in\mathbb N^{2n}$. We consider 
$$S = S\big(\langle(x,\xi)\rangle^{\vert \alpha\vert+\vert \beta\vert},\vert dx\vert^2+\vert d\xi\vert^2\big),$$
the symbol class of smooth functions $f\in C^{\infty}(\mathbb R^{2n})$ satisfying
$$\forall(a,b)\in\mathbb N^{2n}, \exists C>0,\forall(x,\xi)\in\mathbb R^{2n},\quad \vert\partial^a_x\partial^b_{\xi}f(x,\xi)\vert\le C\langle(x,\xi)\rangle^{\vert\alpha\vert+\vert\beta\vert}.$$ 
The Euclidean metric $\vert dx\vert^2+\vert d\xi\vert^2$ is admissible, that is, slowly varying, satisfying the uncertainty principle and temperate. The function $\langle(x,\xi)\rangle^{\vert\alpha\vert+\vert\beta\vert}$ is a $\big(\vert dx\vert^2+\vert d\xi\vert^2\big)$-slowly varying weight, see \cite{MR2599384} (Lemma 2.2.18). Therefore, $S$ is a symbol class with nice symbolic calculus. Clearly, $x^{\alpha},\xi^{\beta}\in S$, and we check that $e^{-q_t}\in S$ for all $0< t\le T^*$. It follows from a straightforward induction that for all $(a,b)\in\mathbb N^n\times\mathbb N^n$ and $t\in[0,T^*]$, there exists a time-dependent polynomial $P_{a,b,t}$ such that 
$$\forall (x,\xi)\in\mathbb R^n\times\mathbb R^n,\quad \partial^a_x\partial^b_{\xi}\big(e^{-q_t(x,\xi)}\big) = P_{a,b,t}(x,\xi)e^{-q_t(x,\xi)}.$$
We deduce from \eqref{19042018E5} that 
$$\forall(a,b)\in\mathbb N^n\times\mathbb N^n, \forall t\in[0,T^*], \forall(x,\xi)\in\mathbb R^n\times\mathbb R^n,\quad P_{a,b,t}(x,\xi) = P_{a,b,t}((x,\xi)_{I,J}),$$
where $(x,\xi)_{I,J}$ stands for the component in $\mathbb R^n_I\times\mathbb R^n_J$ of the vector $(x,\xi)\in\mathbb R^n\times\mathbb R^n$ according to the decomposition $\mathbb R^n\times\mathbb R^n = (\mathbb R^n_I\times\mathbb R^n_J)\oplus^{\perp}(\mathbb R^n_I\times\mathbb R^n_J)^{\perp}$, the orthogonality being taken with respect to the canonical Euclidean structure of $\mathbb R^n\times\mathbb R^n$. It follows from \eqref{08052018E2} that for all $(a,b)\in\mathbb N^n\times\mathbb N^n$, $0\le t\le T^*$ and $(x,\xi)\in\mathbb R^n$,
\begin{multline*}
	\big\vert\partial^a_x\partial^b_{\xi}\big(e^{-q_t(x,\xi)}\big)\big\vert = \big\vert P_{a,b,t}(x,\xi)e^{-q_t(x,\xi)}\big\vert \\
	= \vert P_{a,b,t}(x,\xi)\vert e^{-(\Reelle q_t)(x,\xi)}\le \vert P_{a,b,t}((x,\xi)_{I,J})\vert e^{-ct^k\vert(x,\xi)_{I,J}\vert^2}.
\end{multline*}
Since $k\geq1$, this implies that
$$\forall (a,b)\in\mathbb N^n\times\mathbb N^n, \forall t\in(0,T^*], \quad \big\Vert\partial^a_x\partial^b_{\xi}\big(e^{-q_t(x,\xi)}\big)\big\Vert_{L^{\infty}(\mathbb R^{2n})}<+\infty,$$
that is $e^{-q_t}\in C^{\infty}_b(\mathbb R^{2n})\subset S$ for all $0< t\le T^*$. In the following, we fix $0< t\le T^*$. Since $x^{\alpha}$ and $\xi^{\beta}$ are polynomials belonging to the symbol class $S$, it follows from \cite{MR2304165} (Theorem 18.5.4) that
\begin{align}\label{01092017e2}
	x^{\alpha}\ \sharp\ \xi^{\beta} 
	& = \sum_{l=0}^{\vert\min(\alpha,\beta)\vert}\left(\frac{i}{2}\right)^l\sum_{\vert\delta\vert+\vert\gamma\vert = l}\frac{(-1)^{\vert\gamma\vert}}{\delta!\ \gamma!}\partial^{\gamma}_x\partial^{\delta}_{\xi}\big(x^{\alpha}\big)\ \partial^{\delta}_x\partial^{\gamma}_{\xi}\big(\xi^{\beta}\big) \\[5pt]
	& = \sum_{l=0}^{\vert\min(\alpha,\beta)\vert}\left(\frac{i}{2}\right)^l\sum_{\vert\gamma\vert = l}\frac{(-1)^{\vert\gamma\vert}}{\gamma!}\frac{\alpha!\ x^{\alpha-\gamma}}{(\alpha-\gamma)!}\frac{\beta!\ \xi^{\beta-\gamma}}{(\beta-\gamma)!}\ \mathbbm{1}_{\gamma\le\min(\alpha,\beta)}, \nonumber
\end{align}
where $\min(\alpha,\beta)\in\mathbb N^n$ is defined for all $j\in\{1,\ldots,n\}$ by $\min(\alpha,\beta)_j = \min(\alpha_j,\beta_j)$ and
$$\mathbbm{1}_{\gamma\le\min(\alpha,\beta)} = \left\{\begin{array}{cc}
	1 & \text{if $\gamma\le\min(\alpha,\beta)$}, \\[5pt]
	0 & \text{otherwise}.
\end{array}\right.$$
As a consequence of \eqref{01092017e2}, $x^{\alpha}\ \sharp\ \xi^{\beta}$ is a polynomial of total degree $\vert\alpha\vert+\vert\beta\vert$, and since $e^{-q_t}\in S$, we deduce from \cite{MR2304165} (Theorem 18.5.4) an explicit formula for the symbol $x^{\alpha}\ \sharp\ \xi^{\beta}\ \sharp\ e^{-q_t(x,\xi)}$: 
\begin{equation}\label{12122017E3}
	x^{\alpha}\ \sharp\ \xi^{\beta}\ \sharp\ e^{-q_t(x,\xi)} = \sum_{k=0}^{\vert\alpha\vert+\vert\beta\vert}\left(\frac{i}{2}\right)^k\sum_{\vert\eta\vert+\vert\rho\vert = k}\frac{(-1)^{\vert\rho\vert}}{\eta! \rho!}\partial^{\rho}_x\partial^{\eta}_{\xi}\big(x^{\alpha}\ \sharp\ \xi^{\beta}\big)\ \partial^{\eta}_x\partial^{\rho}_{\xi}\big(e^{-q_t(x,\xi)}\big).
\end{equation}
Let $(a,b)\in\mathbb N^n\times\mathbb N^n$. It follows from \eqref{12122017E3} that for all $(x,\xi)\in\mathbb R^{2n}$,
$$\partial^a_x\partial^b_{\xi}\big(x^{\alpha}\ \sharp\ \xi^{\beta}\ \sharp\ e^{-q_t(x,\xi)}\big)
= \sum_{k=0}^{\vert\alpha\vert + \vert\beta\vert}\left(\frac{i}{2}\right)^k\sum_{\vert\eta\vert+\vert\rho\vert = k}\frac{(-1)^{\vert\rho\vert}}{\eta! \rho!}
\partial^a_x\partial^b_{\xi}\big(\partial^{\rho}_x\partial^{\eta}_{\xi}\big(x^{\alpha}\ \sharp\ \xi^{\beta}\big)\ \partial^{\eta}_x\partial^{\rho}_{\xi}\big(e^{-q_t(x,\xi)}\big)\big).$$
As a consequence of the Leibniz formula, this equality also writes as 
\begin{multline}\label{23042018E8}
	\partial^a_x\partial^b_{\xi}\big(x^{\alpha}\ \sharp\ \xi^{\beta}\ \sharp\ e^{-q_t(x,\xi)}\big)
	= \sum_{k=0}^{\vert\alpha\vert + \vert\beta\vert}\left(\frac i2\right)^k\sum_{\vert\eta\vert+\vert\rho\vert = k}\frac{(-1)^{\vert\rho\vert}}{\eta! \rho!} \\
	\sum_{a'\le a}\sum_{b'\le b}
	\binom a{a'}\binom b{b'}\partial^{a'+\rho}_x\partial^{b'+\eta}_{\xi}\big(x^{\alpha}\ \sharp\ \xi^{\beta}\big)\ \partial^{a-a'+\eta}_x\partial^{b-b'+\rho}_{\xi}\big(e^{-q_t(x,\xi)}\big).
\end{multline}
Moreover, we notice from \eqref{01092017e2} that the derivatives of the symbol $x^{\alpha}\ \sharp\ \xi^{\beta}$ are given for all $(x,\xi)\in\mathbb R^n\times\mathbb R^n$ by
\begin{align}\label{23042018E10}
	&\ \partial^{a'+\rho}_x\partial^{b'+\eta}_{\xi}(x^{\alpha}\ \sharp\ \xi^{\beta}) \\
	= &\ \sum_{l=0}^{\vert\min(\alpha,\beta)\vert}\left(\frac{i}{2}\right)^l\sum_{\vert\gamma\vert = l}\frac{(-1)^{\vert\gamma\vert}}{\gamma!}\frac{\alpha!}{(\alpha-\gamma)!}\frac{\beta!}{(\beta-\gamma)!}\partial^{a'+\rho}_x\big(x^{\alpha-\gamma}\big)\ \partial^{b'+\eta}_{\xi}\big(\xi^{\beta-\gamma}\big)\ \mathbbm{1}_{\gamma\le\min(\alpha,\beta)} \nonumberÊ\\
	= &\ \sum_{l=0}^{\vert\min(\alpha,\beta)\vert}\left(\frac{i}{2}\right)^l\sum_{\vert\gamma\vert = l}\frac{(-1)^{\vert\gamma\vert}}{\gamma!}\frac{\alpha!\ x^{\alpha-\gamma-a'-\rho}}{(\alpha-\gamma-a'-\rho)!}\frac{\beta!\ \xi^{\beta-\gamma-b'-\eta}}{(\beta-\gamma-b'-\eta)!}\mathbbm{1}_{\gamma+a'+\rho \le \alpha}\ \mathbbm{1}_{\gamma+b'+\eta \le \beta}, \nonumber
\end{align}
with
$$\mathbbm1_{\gamma+a'+\rho \le \alpha} = \left\{\begin{array}{cc}
	1 & \text{if $\gamma+a'+\rho \le \alpha$}, \\[5pt]
	0 & \text{otherwise},
\end{array}\right.\quad\text{and}\quad\mathbbm{1}_{\gamma+b'+\eta\le \beta} = \left\{\begin{array}{cc}
	1 & \text{if $\gamma+b'+\eta\le \beta$}, \\[5pt]
	0 & \text{otherwise}.
\end{array}\right.$$
We consider the finite subset $E_{a,b,\alpha,\beta}\subset\mathbb N^{5n}$ whose elements $(\eta,\rho,\gamma,a',b')\in E_{a,b,\alpha,\beta}$ satisfy:
$$\begin{array}{llll}
	0\le a'\le a, & 0\le b'\le b, && \\[5pt]
	\gamma+a'+\rho \le \alpha, & \gamma+b'+\eta \le \beta. &&
\end{array}$$
Moreover, for all $(\eta,\rho,\gamma,a',b')\in E_{a,b,\alpha,\beta}$, we set
$$\Gamma_{\eta,\rho,\gamma,a',b'} = \left(\frac{i}{2}\right)^{\vert\eta\vert + \vert\rho\vert + \vert\gamma\vert}\binom a{a'}\binom b{b'}\frac{\alpha!}{(\alpha-\gamma-a'-\rho)!}\frac{\beta!}{(\beta-\gamma-b'-\eta)!}\frac{(-1)^{\vert\rho\vert+\vert\gamma\vert}}{\eta!\ \rho!\ \gamma!}.$$
It follows from \eqref{23042018E8} and \eqref{23042018E10} that the derivatives of the symbol $x^{\alpha}\ \sharp\ \xi^{\beta}\ \sharp\ e^{-q_t(x,\xi)}$ satisfy the following estimates
\begin{multline}\label{21062019E4}
	\big\vert\partial^a_x\partial^b_{\xi}\big(x^{\alpha}\ \sharp\ \xi^{\beta}\ \sharp\ e^{-q_t(x,\xi)}\big)\big\vert \\[5pt]
	\le \sum_{(\eta,\rho,\gamma,a',b')\in E_{a,b,\alpha,\beta}}\vert\Gamma_{\eta,\rho,\gamma,a',b'}\vert\ \big\vert x^{\alpha-\gamma-a'-\rho}\xi^{\beta-\gamma-b'-\eta}\partial^{a-a'+\eta}_x\partial^{b-b'+\rho}_{\xi}\big(e^{-q_t(x,\xi)}\big)\big\vert.
\end{multline}
By using \eqref{20062018E1}, we finally notice that
\begin{equation}\label{23042018E9}
	\vert\Gamma_{\eta,\rho,\gamma,a',b'}\vert\le2^{\vert a\vert+\vert b\vert}\ \frac{\alpha!}{(\alpha-\gamma-a'-\rho)!}\frac{\beta!}{(\beta-\gamma-b'-\eta)!}\frac1{\eta!\ \rho!\ \gamma!}.
\end{equation}
Lemma \ref{20092017L1} is then a consequence of \eqref{21062019E4} and \eqref{23042018E9}.
\end{proof}

\subsection{Proof of Theorem \ref{18042018T1}} This subsection is devoted to the proof of Theorem \ref{18042018T1}. Let $T>0$ and $q_t:\mathbb R^{2n}\rightarrow\mathbb C$ be a time-dependent complex-valued quadratic form satisfying \eqref{08052018E2} and \eqref{19042018E5}, and whose coefficients depend continuously on the time variable $0\le t\le T$. In order to apply the Calder\'on-Vaillancourt theorem, we estimate the following norms for all $(a,b)\in\mathbb N^n\times\mathbb N^n$, $(\alpha,\beta)\in\mathbb N^n_I\times\mathbb N^n_J$ and $0< t\ll1$,
\begin{equation}\label{20062018E5}
	\big\Vert\partial^a_x\partial^b_{\xi}\big(x^{\alpha}\ \sharp\ \xi^{\beta}\ \sharp\ e^{-q_t(x,\xi)}\big)\big\Vert_{L^{\infty}(\mathbb R^{2n})},
\end{equation}
where $I,J\subset\{1,\ldots,n\}$ are defined in \eqref{08052018E2} and \eqref{19042018E5}. \\[5pt]
\textbf{1.} Since $q_t$ satisfies \eqref{19042018E5}, $q_t$ only depends on the variable of $\mathbb R^n_I\times\mathbb R^n_J$, that is
\begin{equation}\label{20062018E6}
	\forall t\in[0,T^*], \forall (x,\xi)\in\mathbb R^n\times\mathbb R^n,\quad q_t(x,\xi) = q_t((x,\xi)_{I,J}),
\end{equation}
with $0<T^*<T$ and where $(x,\xi)_{I,J}$ denotes the component in $\mathbb R^n_I\times\mathbb R^n_J$ of the vector $(x,\xi)\in\mathbb R^n\times\mathbb R^n$ according to the decomposition $\mathbb R^n\times\mathbb R^n = (\mathbb R^n_I\times\mathbb R^n_J)\oplus^{\perp}(\mathbb R^n_I\times\mathbb R^n_J)^{\perp}$, the orthogonality being taken with respect to the canonical Euclidean structure of $\mathbb R^n\times\mathbb R^n$. Moreover, it follows from \eqref{08052018E2} that the time-dependent quadratic form $\Reelle q_t$ enjoys the following coercivity estimate
$$\forall t\in[0,T^*], \forall (x,\xi)\in\mathbb R^n_I\times\mathbb R^n_J,\quad (\Reelle q_t)(x,\xi)\geq ct^k\vert(x,\xi)\vert^2,$$
with $c>0$ and $k\geq1$. Thus, we deduce from Lemma \ref{19042018L1} (by choosing a bijection between $\mathbb R^n_I\times\mathbb R^n_J$ and $\mathbb R^{\#I+\#J}$, $\#$ denoting the cardinality) that there exist some positive constants $0<t_0<\min(1,T^*)$ and $C_1>1$ such that for all $(\delta_1,\mu_1),(\delta_2,\mu_2)\in\mathbb N^n_I\times\mathbb N^n_J$, $0<t\le t_0$ and $(x,\xi)\in\mathbb R^n_I\times\mathbb R^n_J$,
\begin{equation}\label{23042018E13}
	\big\vert x^{\delta_1}\xi^{\mu_1}\partial^{\delta_2}_x\partial^{\mu_2}_{\xi}\big(e^{-q_t(x,\xi)}\big)\big\vert
\le\frac{C_1^{1+\vert\delta_1\vert+\vert\delta_2\vert+\vert\mu_1\vert + \vert\mu_2\vert}}{t^{\frac k2(\vert\delta_1\vert+\vert\mu_1\vert+2\vert\delta_2\vert + 2\vert\mu_2\vert + s)}}
((\delta_1)!\ (\delta_2)!\ (\mu_1)!\ (\mu_2)!)^{\frac12},
\end{equation}
where $$s = 5(2n)/4 + 2\lfloor (2n)/2\rfloor + 2 = 9n/2+2.$$ 
Moreover, it follows from \eqref{20062018E6} that \eqref{23042018E13} can be extended to all $(\delta_1,\mu_1)\in\mathbb N^n_I\times\mathbb N^n_J$, $(\delta_2,\mu_2)\in\mathbb N^{2n}$, $0<t\le t_0$ and $(x,\xi)\in\mathbb R^{2n}$,
since
$$\forall(\delta_1,\mu_1)\in\mathbb N^n_I\times\mathbb N^n_J,\forall(x,\xi)\in\mathbb R^{2n},\quad (x,\xi)^{(\delta_1,\mu_1)} = \left((x,\xi)_{I,J}\right)^{(\delta_1,\mu_1)}.$$
On the other hand, it follows from Lemma \ref{20092017L1} that for all $(a,b)\in\mathbb N^{2n}$ and $(\alpha,\beta)\in\mathbb N^n_I\times\mathbb N^n_J$, there exists a finite subset $E_{a,b,\alpha,\beta}\subset\mathbb N^{5n}$, whose elements $(\eta,\rho,\gamma,a',b')\in E_{a,b,\alpha,\beta}$ satisfy \eqref{20042018E11}, such that for all $0\le t\le T^*$ and $(x,\xi)\in\mathbb R^{2n}$,
\begin{multline}\label{20042018E1}
	\big\vert\partial^a_x\partial^b_{\xi}\big(x^{\alpha}\ \sharp\ \xi^{\beta}\ \sharp\ e^{-q_t(x,\xi)}\big)\big\vert \\[5pt]
	\le 2^{\vert a\vert + \vert b\vert}\sum_{(\eta,\rho,\gamma,a',b')\in E_{a,b,\alpha,\beta}}\Lambda_{\eta,\rho,\gamma,a',b'}\ \big\vert x^{\alpha-\gamma-a'-\rho}\xi^{\beta-\gamma-b'-\eta}\partial^{a-a'+\eta}_x\partial^{b-b'+\rho}_{\xi}\big(e^{-q_t(x,\xi)}\big)\big\vert,
\end{multline}
where
\begin{equation}\label{20062018E7}
	\Lambda_{\eta,\rho,\gamma,a',b'} = \frac{\alpha!}{(\alpha-\gamma-a'-\rho)!}\frac{\beta!}{(\beta-\gamma-b'-\eta)!}\frac{1}{\eta!\ \rho!\ \gamma!}.
\end{equation}
As a consequence of  \eqref{23042018E13}, we have that for all $(a,b)\in\mathbb N^{2n}$, $(\alpha,\beta)\in\mathbb N^n_I\times\mathbb N^n_J$, $(\eta,\rho,\gamma,a',b')\in E_{a,b,\alpha,\beta}$, $0<t<t_0$ and $(x,\xi)\in\mathbb R^{2n}$,
\begin{multline*}
	\big\vert x^{\alpha-\gamma-a'-\rho}\xi^{\beta-\gamma-b'-\eta}\partial^{a-a'+\eta}_x\partial^{b-b'+\rho}_{\xi}\big(e^{-q_t(x,\xi)}\big)\big\vert \\[5pt]
	\le \frac{C_1^{1+\vert\alpha-\gamma-a'-\rho\vert + \vert\beta-\gamma-b'-\eta\vert + \vert a-a'+\eta\vert + \vert b-b'+\rho\vert}}
	{t^{\frac{k}{2}(\vert\alpha-\gamma-a'-\rho\vert+\vert\beta-\gamma-b'-\eta\vert+2\vert a-a'+\eta\vert+2\vert b-b'+\rho\vert + s)}} \\[5pt]
	\left((\alpha-\gamma-a'-\rho)!\ (\beta-\gamma-b'-\eta)!\ (a-a'+\eta)!\ (b-b'+\rho)!\right)^{\frac12},
\end{multline*}
since it follows from \eqref{20042018E11} that for all $(a,b)\in\mathbb N^{2n}$, $(\alpha,\beta)\in\mathbb N^n_I\times\mathbb N^n_J$ and $(\eta,\rho,\gamma,a',b')\in E_{a,b,\alpha,\beta}$,
$$\alpha-\gamma-a'-\rho\in\mathbb N^n_I\quad\text{and}\quad\beta-\gamma-b'-\eta\in\mathbb N^n_J.$$
We deduce from the previous inequality that
for all $(a,b)\in\mathbb N^{2n}$, $(\alpha,\beta)\in\mathbb N^n_I\times\mathbb N^n_J$, $(\eta,\rho,\gamma,a',b')\in E_{a,b,\alpha,\beta}$, $0<t<t_0$ and $(x,\xi)\in\mathbb R^{2n}$,
\begin{multline*}
	\big\vert x^{\alpha-\gamma-a'-\rho}\xi^{\beta-\gamma-b'-\eta}\partial^{a-a'+\eta}_x\partial^{b-b'+\rho}_{\xi}\big(e^{-q_t(x,\xi)}\big)\big\vert \\[5pt]
	\le \frac{C_1^{1+2\vert\alpha\vert + 2\vert\beta\vert + 2\vert a\vert + 2\vert b\vert}}
	{t^{k(\vert\alpha\vert + \vert\beta\vert + \vert a\vert + \vert b\vert + s/2)}}
	\left((\alpha-\gamma-a'-\rho)!\ (\beta-\gamma-b'-\eta)!\ (a+\eta)!\ (b+\rho)!\right)^{\frac12},
\end{multline*}
since $C_1>1$, $0<t_0<1$, and from \eqref{20042018E11},
\begin{multline*}
	\vert\alpha-\gamma-a'-\rho\vert + \vert\beta-\gamma-b'-\eta\vert + 2\vert a-a'+\eta\vert + 2\vert b-b'+\rho\vert \\[5pt]
	= \vert\alpha-\gamma-a'\vert + \vert\beta-\gamma-b'\vert + \vert\rho\vert+\vert\eta\vert + 2\vert a-a'\vert + 2\vert b-b'\vert 
	\le 2\vert\alpha\vert + 2\vert\beta\vert + 2\vert a\vert + 2\vert b\vert.
\end{multline*}
Moreover, as a consequence of \eqref{20042018E11} and \eqref{22102017E1}, we have that for all $(a,b)\in\mathbb N^{2n}$, $(\alpha,\beta)\in\mathbb N^n_I\times\mathbb N^n_J$ and $(\eta,\rho,\gamma,a',b')\in E_{a,b,\alpha,\beta}$,
$$(a+\eta)!\le 2^{\vert a\vert + \vert\eta\vert}a!\ \eta!\le 2^{\vert a\vert + \vert \beta\vert}a!\ \eta!,$$
and
$$(b+\rho)!\le 2^{\vert b\vert + \vert\alpha\vert}b!\ \rho!.$$
This implies the following estimate for all $(a,b)\in\mathbb N^{2n}$, $(\alpha,\beta)\in\mathbb N^n_I\times\mathbb N^n_J$, $(\eta,\rho,\gamma,a',b')\in E_{a,b,\alpha,\beta}$, $0<t<t_0$ and $(x,\xi)\in\mathbb R^{2n}$,
\begin{multline}\label{20042018E6}
	\big\vert x^{\alpha-\gamma-a'-\rho}\xi^{\beta-\gamma-b'-\eta}\partial^{a-a'+\eta}_x\partial^{b-b'+\rho}_{\xi}\big(e^{-q_t(x,\xi)}\big)\big\vert \\[5pt]
	\le \frac{2^{\frac12(\vert a\vert + \vert b\vert + \vert\alpha\vert+\vert\beta\vert)}C_1^{1+2\vert\alpha\vert + 2\vert\beta\vert + 2\vert a\vert + 2\vert b\vert}}{t^{k(\vert\alpha\vert+\vert\beta\vert+\vert a\vert+\vert b\vert + s/2)}}
	\left((\alpha-\gamma-a'-\rho)!\ (\beta-\gamma-b'-\eta)!\ \eta!\ \rho!\ a!\ b!\right)^{\frac12}.
\end{multline}
As a consequence of \eqref{20042018E1} and \eqref{20042018E6}, we obtain a first estimate of the norms \eqref{20062018E5}: there exists a positive constant $C_2>1$ such that for all $(a,b)\in\mathbb N^n\times\mathbb N^n$, $(\alpha,\beta)\in\mathbb N^n_I\times\mathbb N^n_J$ and $0<t<t_0$,
\begin{multline*}
	\big\Vert\partial^a_x\partial^b_{\xi}\big(x^{\alpha}\ \sharp\ \xi^{\beta}\ \sharp\ e^{-q_t(x,\xi)}\big)\big\Vert_{L^{\infty}(\mathbb R^{2n})}
	\le \frac{C_2^{1+\vert \alpha\vert + \vert \beta\vert + \vert a\vert + \vert b\vert}}{t^{k(\vert\alpha\vert+\vert\beta\vert+\vert a\vert+\vert b\vert + s/2)}} \\
	\sum_{(\eta,\rho,\gamma,a',b')\in E_{a,b,\alpha,\beta}}\Lambda_{\eta,\rho,\gamma,a',b'}\left((\alpha-\gamma-a'-\rho)!\ (\beta-\gamma-b'-\eta)!\ \eta!\ \rho!\ a!\ b!\right)^{\frac12}.
\end{multline*}
Moreover, we notice from \eqref{20062018E7} that
\begin{multline*}
	\Lambda_{\eta,\rho,\gamma,a',b'}\left((\alpha-\gamma-a'-\rho)!\ (\beta-\gamma-b'-\eta)!\ \eta!\ \rho!\ a!\ b!\right)^{\frac12} \\[5pt]
	= \frac{\alpha!}{((\alpha-\gamma-a'-\rho)!)^{\frac12}}\frac{\beta!}{((\beta-\gamma-b'-\eta)!)^{\frac12}}\frac{1}{(\eta!)^{\frac12}(\rho!)^{\frac12}\gamma!}\ (a!)^{\frac12}\ (b!)^{\frac12}.
\end{multline*}
We therefore deduce the following inequality for all $(a,b)\in\mathbb N^n\times\mathbb N^n$, $(\alpha,\beta)\in\mathbb N^n_I\times\mathbb N^n_J$ and $0<t\le t_0$,
\begin{multline}\label{20062018E3}
	\big\Vert\partial^a_x\partial^b_{\xi}\big(x^{\alpha}\ \sharp\ \xi^{\beta}\ \sharp\ e^{-q_t(x,\xi)}\big)\big\Vert_{L^{\infty}(\mathbb R^{2n})}
	\le \frac{C_2^{1+\vert \alpha\vert + \vert \beta\vert + \vert a\vert + \vert b\vert}}{t^{k(\vert\alpha\vert+\vert\beta\vert+\vert a\vert+\vert b\vert + s/2)}} \\[5pt]
	(a!)^{\frac12}\ (b!)^{\frac12}\sum_{(\eta,\rho,\gamma,a',b')\in E_{a,b,\alpha,\beta}}\frac{\alpha!}{((\alpha-\gamma-a'-\rho)!)^{\frac12}}\frac{\beta!}{((\beta-\gamma-b'-\eta)!)^{\frac12}}\frac1{(\eta!)^{\frac12}(\rho!)^{\frac12}\gamma!}.
\end{multline}
\textbf{2.} Let $(a,b)\in\mathbb N^n\times\mathbb N^n$ and $(\alpha,\beta)\in\mathbb N^n_I\times\mathbb N^n_J$. We now derive an estimate of the quantity
$$\Delta_{\eta,\rho,\gamma,a',b'} = \frac{\alpha!}{((\alpha-\gamma-a'-\rho)!)^{\frac12}}\frac{\beta!}{((\beta-\gamma-b'-\eta)!)^{\frac12}}\frac{1}{(\eta!)^{\frac12}(\rho!)^{\frac12}\gamma!},\quad (\eta,\rho,\gamma,a',b')\in E_{a,b,\alpha,\beta},$$
appearing in \eqref{20062018E3}. First, we deduce from \eqref{20042018E11} and \eqref{22102017E1} the following inequalities for all $(\eta,\rho,\gamma,a',b')\in E_{a,b,\alpha,\beta}$,
\begin{multline*}
	(\alpha-\gamma-\rho)! = (\alpha-\gamma-\rho-a'+a')! \le 2^{\vert \alpha-\gamma-\rho\vert}\ (\alpha-\gamma-\rho-a')!\ (a')! \\[5pt]
	\le 2^{\vert \alpha\vert}\ (\alpha-\gamma-\rho-a')!\ a!,
\end{multline*}
and
$$(\beta-\gamma-\eta)! \le 2^{\vert \beta\vert}\ (\beta-\gamma-\eta-b')!\ b!.$$
As a consequence, we have that for all $(\eta,\rho,\gamma,a',b')\in E_{a,b,\alpha,\beta}$,
$$\Delta_{\eta,\rho,\gamma,a',b'}\le \frac{\alpha!}{((\alpha-\gamma-\rho)!)^{\frac12}}\frac{\beta!}{((\beta-\gamma-\eta)!)^{\frac12}}\frac{1}{(\eta!)^{\frac12}(\rho!)^{\frac12}\gamma!}\ 2^{\frac12\vert \alpha + \beta\vert}\ (a!)^{\frac12}\ (b!)^{\frac12}.$$
Moreover, it follows from the definition of the Binomial coefficients that for all $(\eta,\rho,\gamma,a',b')\in E_{a,b,\alpha,\beta}$,
\begin{align*}
	 & \frac{\alpha!}{((\alpha-\gamma-\rho)!)^{\frac12}}\frac{\beta!}{((\beta-\gamma-\eta)!)^{\frac12}}\frac{1}{(\eta!)^{\frac12}(\rho!)^{\frac12}\gamma!} \\[5pt]
	= &\ \left[\frac{\alpha!}{(\alpha-\gamma-\rho)!\ (\gamma+\rho)!}\ \frac{(\gamma+\rho)!}{\rho!\ \gamma!}\ \frac{\beta!}{(\beta-\gamma-\eta)!\ (\gamma+\eta)!}\ \frac{(\gamma+\eta)!}{\eta!\ \gamma!}\right]^{\frac12}\ (\alpha!)^{\frac12}\ (\beta!)^{\frac12} \\[5pt]
	= &\ \left[\binom{\alpha}{\gamma+\rho}\binom{\gamma+\rho}{\gamma}\binom{\beta}{\gamma+\eta}\binom{\gamma+\eta}{\eta}\right]^{\frac12}\ (\alpha!)^{\frac12}\ (\beta!)^{\frac12},
\end{align*}
and we deduce from this equality and \eqref{20062018E1} that
$$\Delta_{\eta,\rho,\gamma,a',b'} \le \left[2^{\vert \alpha\vert + \vert \gamma\vert+\vert\rho\vert + \vert \beta\vert + \vert\gamma\vert + \vert\eta\vert}\ 2^{\vert \alpha\vert + \vert \beta\vert}\right]^{\frac12}\ (a!)^{\frac12}\ (b!)^{\frac12}\ (\alpha!)^{\frac12}\ (\beta!)^{\frac12}.$$
As a consequence of \eqref{20042018E11}, each element $(\eta,\rho,\gamma,a',b')$ of $E_{a,b,\alpha,\beta}$ satisfies
$$\vert \gamma\vert+\vert\rho\vert\le\vert\alpha\vert\quad \text{and}\quad \vert\gamma\vert + \vert\eta\vert\le\vert\beta\vert,$$
and thus, the following inequality holds for all $(\eta,\rho,\gamma,a',b')\in E_{a,b,\alpha,\beta}$,
\begin{multline}\label{20042018E4}
	\Delta_{\eta,\rho,\gamma,a',b'} 
	\le \left[2^{\vert \alpha\vert + \vert\alpha\vert + \vert \beta\vert + \vert\beta\vert}\ 2^{\vert \alpha\vert + \vert \beta\vert}\right]^{\frac12}\ (a!)^{\frac12}\ (b!)^{\frac12}\ (\alpha!)^{\frac12}\ (\beta!)^{\frac12} \\[5pt]
	= (2\sqrt{2})^{\vert\alpha\vert+\vert\beta\vert}\ (a!)^{\frac12}\ (b!)^{\frac12}\ (\alpha!)^{\frac12}\ (\beta!)^{\frac12}.
\end{multline}
\textbf{3.} Let $(a,b)\in\mathbb N^n\times\mathbb N^n$ and $(\alpha,\beta)\in\mathbb N^n_I\times\mathbb N^n_J$. The last quantity we need to control is the cardinality of $E_{a,b,\alpha,\beta}$, denoted $\vert E_{a,b,\alpha,\beta}\vert$. It follows from \eqref{20042018E11} that for all $(\eta,\rho,\gamma,a',b')\in E_{a,b,\alpha,\beta}$,
$$\begin{array}{lll}
	\vert\eta\vert\le\vert\beta\vert, & \vert\rho\vert\le\vert\alpha\vert, & \vert\gamma\vert\le\min(\vert\alpha\vert,\vert\beta\vert), \\[10pt]
	\vert a'\vert\le\vert a\vert, & \vert b'\vert\le\vert b\vert, &
\end{array}$$
and by using \eqref{20062018E1} and \eqref{24042018E9}, we deduce that $\vert E_{a,b,\alpha,\beta}\vert$ satisfies the following estimate:
\begin{multline}\label{20042018E5}
	\vert E_{a,b,\alpha,\beta}\vert \le\binom{n+\vert\alpha\vert}{\vert\alpha\vert}\binom{n+\vert\beta\vert}{\vert\beta\vert}
	\binom{n+\min(\vert\alpha\vert,\vert\beta\vert)}{\min(\vert\alpha\vert,\vert\beta\vert)}
	\binom{n+\vert a\vert}{\vert a\vert}\binom{n+\vert b\vert}{\vert b\vert} \\[5pt]
	\le 2^{\vert\alpha\vert + \vert\beta\vert + \min(\vert\alpha\vert,\vert\beta\vert) + \vert a\vert + \vert b\vert + 5n}.
\end{multline}
\textbf{4.} As a consequence of \eqref{20062018E3}, \eqref{20042018E4} and \eqref{20042018E5}, there exists a positive constant $C_3>1$ such that for all $(a,b)\in\mathbb N^{2n}$, $(\alpha,\beta)\in\mathbb N^n_I\times\mathbb N^n_J$ and $0<t<t_0$,
\begin{equation}\label{20042018E7}
	\big\Vert\partial^a_x\partial^b_{\xi}\big(x^{\alpha}\ \sharp\ \xi^{\beta}\ \sharp\ e^{-q_t(x,\xi)}\big)\big\Vert_{L^{\infty}(\mathbb R^{2n})}\le
	\frac{C_3^{1+\vert a\vert + \vert b\vert + \vert \alpha\vert + \vert \beta\vert}}{t^{k(\vert\alpha\vert+\vert\beta\vert+\vert a\vert+\vert b\vert + s/2)}}\ a!\ b!\ (\alpha!)^{\frac12}\ (\beta!)^{\frac12}.
\end{equation}
It follows from \eqref{20042018E7} and the Calder\'on-Vaillancourt theorem, see \cite{MR1696697} (Theorem 1.2) which provides sharp estimates, that the operator $x^{\alpha}D^{\beta}_x(e^{-q_t})^w:L^2(\mathbb R^n)\rightarrow L^2(\mathbb R^n)$ is bounded on $L^2(\mathbb R^n)$, and that its operator norm is controlled in the following way
\begin{equation}\label{20042018E8}
	\big\Vert x^{\alpha}D^{\beta}_x(e^{-q_t})^w\big\Vert_{\mathcal{L}(L^2)}
\le C \sup_{\vert a\vert,\vert b\vert\le \lfloor n/2\rfloor+1}\big\Vert\partial^a_x\partial^b_{\xi}\big(x^{\alpha}\ \sharp\ \xi^{\beta}\ \sharp\ e^{-q_t(x,\xi)}\big)\big\Vert_{L^{\infty}(\mathbb R^{2n})},
\end{equation}
where $C>0$ is a positive constant and $\mathcal{L}(L^2)$ stands for the set of bounded operators on $L^2(\mathbb R^n)$. Theorem \ref{18042018T1} is then a consequence of \eqref{20042018E7} and \eqref{20042018E8}.

\section{Regularizing effects of quadratic operators}\label{refquad}

This section is devoted to the proof of Theorem \ref{20112017T1}. Let $q:\mathbb R^n_x\times\mathbb R^n_{\xi}\rightarrow\mathbb C$ be a complex-valued quadratic form with a non-negative real part $\Reelle q\geq0$. We consider, for all $0\le t\ll1$ small enough, the time-dependent quadratic form $q_t$ defined by
\begin{equation}\label{24042018E1}
	q_t:X\in\mathbb R^{2n}\mapsto\sigma(X,\tan(tF)X)\in\mathbb C,
\end{equation}
where $F$ is the Hamilton map of the quadratic form $q$ and $\tan$ denotes the matrix tangent function. In the following, we aim at studying the time-dependent quadratic form $q_t$ in a general setting in order to check that when the singular space $S$ of $q$ satisfies $S^{\perp} = \mathbb R^n_I\times\mathbb R^n_J$, with $I,J\subset\{1,\ldots,n\}$, the orthogonality being taken with respect to the canonical Euclidean structure of $\mathbb R^{2n}$, and that the inclusion $S\subset\Ker(\Imag F)$ holds, $q_t$ satisfies the conditions \eqref{08052018E2} and \eqref{19042018E5} for the pair $(I,J)$. More precisely, we aim at deriving some coercive estimates for the real time-dependent quadratic form $\Reelle q_t$ in subspaces of the phase space, and then to investigate the variables on which $q_t$ depends. We shall then see that Theorem \ref{20112017T1} can be deduced from Theorem \ref{18042018T1} and the Mehler formula.

\subsection{Coercive estimates} First, we establish that the quadratic form $\Reelle q_t$ is coercive on some subspaces of the phase space. To that end, we introduce the following auxiliary time-dependent quadratic form
\begin{equation}\label{30032018E8}
	Q_t:X\in\mathbb C^{2n}\mapsto -i\sigma\left(\overline{(e^{2itF}+I_{2n})X},(e^{2itF}-I_{2n})X\right)\in\mathbb C,\quad t\geq0.
\end{equation}
The following lemma is an adaptation of \cite{MR3880300} (Lemmas 3.1 and 3.2):

\begin{lem}\label{04042018L1} Let $q:\mathbb R^n_x\times\mathbb R^n_{\xi}\rightarrow\mathbb{C}$ be a complex-valued quadratic form with a non-negative real part $\Reelle q\geq0$ and $F$ be the Hamilton map of $q$. Then, for all $t\geq0$ and $X\in\mathbb{C}^{2n}$, we have
$$\Reelle \left[Q_t(X)\right] = 4\int_0^t(\Reelle q)(\Reelle(e^{2isF}X))\ ds + 4\int_0^t(\Reelle q)(\Imag(e^{2isF}X))\ ds\geq0,$$
where $Q_t$ is the time-dependent quadratic form associated to $q$ defined in \eqref{30032018E8}.
\end{lem}

\begin{proof} For $X\in\mathbb{C}^{2n}$, we consider the function $\varphi_X : t\geq0\mapsto \Reelle \left[Q_t(X)\right]$. We first notice from the skew-symmetry of the symplectic form that for all $t\geq0$,
\begin{align*}
	\varphi_X(t) & = \Reelle\left[-i\sigma(\overline{(e^{2itF}+I_{2n})X},(e^{2itF}-I_{2n})X)\right] \\[5pt]
	& = \Reelle\left[-i\left(\sigma(\overline{e^{2itF}X},e^{2itF}X) - \sigma(\overline{X},X)\right)\right]
	+ \Reelle\left[-i\left(\sigma(\overline{X},e^{2itF}X) - \sigma(\overline{e^{2itF}X},X)\right)\right] \\[5pt]
	& = -i\left(\sigma(\overline{e^{2itF}X},e^{2itF}X) - \sigma(\overline{X},X)\right)
	+ \Reelle\left[-i\left(\sigma(\overline{X},e^{2itF}X) + \overline{\sigma(\overline{X},e^{2itF}X)}\right)\right] \\[5pt]
	& = -i\left(\sigma(\overline{e^{2itF}X},e^{2itF}X) - \sigma(\overline{X},X)\right),
\end{align*}
as for all $Y\in\mathbb{C}^{2n}$,
\begin{align*}
	-i\sigma(\overline{Y},Y) & = \frac{1}{2i}(\sigma(\overline{Y},Y) + \sigma(\overline{Y},Y)) = \frac{1}{2i}(\sigma(\overline{Y},Y) - \sigma(Y,\overline{Y})) \\[5pt]
	& = \frac{1}{2i}(\sigma(\overline{Y},Y) - \overline{\sigma(\overline{Y},Y)}) = \Imag[\sigma(\overline{Y},Y)]\in\mathbb R.
\end{align*}
Moreover, the function $\varphi_X$ is smooth and its derivative is given for all $t\geq0$ by
\begin{align}\label{22062018E5}
	(\varphi_X)'(t) & = -i\sigma(\overline{2iFe^{2itF}X},e^{2itF}X) - i\sigma(\overline{e^{2itF}X},2iFe^{2itF}X) \\[5pt]
	& = -2\sigma(\overline{Fe^{2itF}X},e^{2itF}X) + 2\sigma(\overline{e^{2itF}X},Fe^{2itF}X) \nonumber \\[5pt]
	& = 2\sigma(\overline{e^{2itF}X},\overline{F}e^{2itF}X) + 2\sigma(\overline{e^{2itF}X},Fe^{2itF}X) \nonumber \\[5pt]
	& = 2\sigma(\overline{e^{2itF}X},(F+\overline{F})e^{2itF}X)
	= 4\sigma(\overline{e^{2itF}X},(\Reelle F)e^{2itF}X), \nonumber
\end{align}
since $F$ is skew-symmetric with respect to $\sigma$, see \eqref{04122017E6}. By using anew \eqref{04122017E6} and the skew-symmetry of the symplectic form, we deduce that for all $Y\in\mathbb{C}^{2n}$,
\begin{gather*}
	\sigma(\Reelle Y,(\Reelle F)(\Imag Y)) = -\sigma((\Reelle F)(\Reelle Y),\Imag Y) = \sigma(\Imag Y,(\Reelle F)(\Reelle Y)),
\end{gather*}
and therefore,
\begin{multline*}
	\sigma(\overline{Y},(\Reelle F)Y) = \sigma(\Reelle Y, (\Reelle F)(\Reelle Y)) + \sigma(\Imag Y,(\Reelle F)(\Imag Y)) \\[5pt]
	+ i\left[\sigma(\Reelle Y,(\Reelle F)(\Imag Y)) -\sigma(\Imag Y,(\Reelle F)(\Reelle Y))\right],
\end{multline*}
that is
\begin{equation}\label{25062018E1}
	\sigma(\overline{Y},(\Reelle F)Y) = \sigma(\Reelle Y, (\Reelle F)(\Reelle Y)) + \sigma(\Imag Y,(\Reelle F)(\Imag Y)).
\end{equation}
It follows from \eqref{04122017E4}, \eqref{22062018E5} and \eqref{25062018E1} that for all $t\geq0$,
\begin{align}\label{25062018E2}
	(\varphi_X)'(t) & = 4\sigma(\Reelle(e^{2itF}X),(\Reelle F)\Reelle(e^{2itF}X)) + 4\sigma(\Imag(e^{2itF}X),(\Reelle F)\Imag(e^{2itF}X)) \\[5pt]
	& = 4(\Reelle q)(\Reelle(e^{2itF}X)) + 4(\Reelle q)(\Imag(e^{2itF}X))\geq0, \nonumber
\end{align}
since $\Reelle q\geq0$. This ends the proof of Lemma \ref{04042018L1} since $\varphi_X(0) = 0$.
\end{proof}

The two following algebraic lemmas are instrumental in the following: 

\begin{lem}\label{18042018L1} Let $A$ be a real $n\times n$ symmetric positive semidefinite matrix. Then, we have
$$\forall X\in\mathbb R^n,\quad \langle X,AX\rangle = 0 \Leftrightarrow AX = 0.$$
\end{lem}

\begin{proof} Thanks to the spectral theorem for real symmetric matrices, we can consider $(e_1,\ldots,e_n)$ an orthonormal basis of $\mathbb R^n$ equipped with its Euclidean structure, where the $e_j$ are eigenvectors of $A$. For all $j=1,\ldots,n$, let $\lambda_j\geq0$ be the eigenvalue associated to $e_j$. The $\lambda_j$ are non-negative real numbers since $A$ is positive semidefinite. Let $X\in\mathbb R^n$ satisfying $\langle X,AX\rangle = 0$. Decomposing $X$ in the basis $(e_1,\ldots,e_n)$, $X = X_1e_1+\ldots+X_ne_n$,
with $X_1,\ldots,X_n\in\mathbb R$, we notice that
$$\langle X,AX\rangle = \sum_{j=1}^n\lambda_jX^2_j = 0,$$
since $(e_1,\ldots,e_n)$ is orthonormal. Moreover, the $\lambda_j$ are non-negative real numbers, and it follows that $\lambda_jX^2_j = 0$ for all $1\le j\le n$. Therefore, $\lambda_jX_j = 0$ for all $1\le j\le n$, and this implies that
$$AX = \sum_{j=1}^n\lambda_jX_je_j = 0.$$
Conversely, if $AX=0$, then $\langle X,AX\rangle = 0$.
\end{proof}

\begin{lem}\label{17042018L1} Let $F\in M_n(\mathbb{C})$ be a complex $n\times n$ matrix and $K\geq0$ be a non-negative integer. Then, for all $X\in\mathbb{C}^n$ satisfying
\begin{equation}\label{17042018E7}
	\forall k\in\{0,\ldots,K\},\quad (\Reelle F)(F^kX) = 0,
\end{equation}
we have
\begin{gather}\label{17042018E8}
	\forall k\in\{0,\ldots,K\},\quad (\Reelle F)(\Imag F)^kX = 0.
\end{gather}
\end{lem}

\begin{proof} We prove \eqref{17042018E8} by induction on the non-negative integer $0\le k\le K$. The induction hypothesis $k=0$ is straightforward from \eqref{17042018E7}. Let $0\le k\le K-1$ such that
\begin{equation}\label{17042018E5}
	\forall l\in\{0,\ldots,k\},\quad (\Reelle F)(\Imag F)^lX = 0.
\end{equation}
It follows from a direct computation that $(\Reelle F)(F^{k+1}X)$ writes as
\begin{equation}\label{10112017E2}
	(\Reelle F)(F^{k+1}X) = \sum i^j(\Reelle F)(\Imag F)^{j_1}(\Reelle F)^{j_2}\ldots(\Reelle F)^{j_{s-1}}(\Imag F)^{j_s}X,
\end{equation}
where the sum is finite with $0\le j,s\le k+1$, the $j_l$ are non-negative integers, where $1\le l\le s$, and each product is composed of $k+2$ matrices including $j$ terms $\Imag F$. When the product 
$$(\Imag F)^{j_1}(\Reelle F)^{j_2}\ldots(\Reelle F)^{j_{s-1}}(\Imag F)^{j_s},$$
contains at least one matrix $\Reelle F$, we can extract from 
$$i^j(\Reelle F)(\Imag F)^{j_1}(\Reelle F)^{j_2}\ldots(\Reelle F)^{j_{s-1}}(\Imag F)^{j_s}X$$
a subproduct of the form
$\Reelle F(\Imag F)^lX,$
with $0\le l\le k$. It follows from the induction hypothesis \eqref{17042018E5} that  $\Reelle F(\Imag F)^lX = 0,$
and therefore, we have
\begin{gather}\label{17042018E9}
	i^j(\Reelle F)(\Imag F)^{j_1}(\Reelle F)^{j_2}\ldots(\Reelle F)^{j_{s-1}}(\Imag F)^{j_s}X = 0.
\end{gather}
On the other hand, the only term such that the product 
$$(\Imag F)^{j_1}(\Reelle F)^{j_2}\ldots(\Reelle F)^{j_{s-1}}(\Imag F)^{j_s},$$
does not contain any matrix $\Reelle F$, writes as
\begin{gather}\label{17042018E10}
	i^j(\Reelle F)(\Imag F)^{j_1}(\Reelle F)^{j_2}\ldots(\Reelle F)^{j_{s-1}}(\Imag F)^{j_s}X = i^{k+1}(\Reelle F)(\Imag F)^{k+1}X.
\end{gather}
It follows from \eqref{17042018E7}, \eqref{10112017E2}, \eqref{17042018E9} and \eqref{17042018E10} that
$$i^{k+1}(\Reelle F)(\Imag F)^{k+1}X = (\Reelle F)(F^{k+1}X) = 0,$$
which ends the induction and the proof of Lemma \ref{17042018L1}.
\end{proof}

The corollary of the next lemma will be key to derive some positivity for the time-dependent quadratic form $\Reelle Q_t$.

\begin{lem}\label{25062018L1} Let $q:\mathbb R^n_x\times\mathbb R^n_{\xi}\rightarrow\mathbb{C}$ be a complex-valued quadratic form with a non-negative real part $\Reelle q\geq0$ and $F$ be its Hamilton map. Let $K\geq0$ be a non-negative integer. For all $X\in\mathbb{C}^{2n}$ satisfying
\begin{equation}\label{22062018E2}
	\forall k\in\{0,\ldots,K\},\quad \partial^{2k+1}_t\Reelle\left[Q_t(X)\right]\Big\vert_{t=0} = 0,
\end{equation}
where $Q_t$ is the time-dependent quadratic form associated to $q$ defined in \eqref{30032018E8}, we have
\begin{equation}\label{22062018E3}
	\forall k\in\{0,\ldots,K\},\quad (\Reelle F)(\Imag F)^kX = 0.
\end{equation}
\end{lem}

\begin{proof} Let $X\in\mathbb{C}^{2n}$ satisfying \eqref{22062018E2}. We first prove by induction that
\begin{equation}\label{25062018E9}
	\forall k\in\{0,\ldots,K\},\quad (\Reelle F)(F^kX) = 0.
\end{equation}
Let $\varphi_X$ be the function defined by 
\begin{equation}\label{25062018E7}
	\varphi_X(t) = \Reelle\left[Q_t(X)\right],\quad t\geq0.
\end{equation}
We recall from \eqref{25062018E2} that the function $\varphi_X$ is smooth and its derivative is given for all $t\geq0$ by
\begin{equation}\label{22062018E7}
	(\varphi_X)'(t) = 4(\Reelle q)(\Reelle(e^{2itF}X)) + 4(\Reelle q)(\Imag(e^{2itF}X)).
\end{equation}
Since the quadratic form $\Reelle q$ is non-negative, it follows from \eqref{22062018E2}, \eqref{25062018E7} and \eqref{22062018E7} applied with $t=0$ that 
\begin{equation}\label{25062018E8}
	(\Reelle q)(\Reelle X) = (\Reelle q)(\Imag X) = 0.
\end{equation}
We deduce from \eqref{21062018E6}, \eqref{29062018E3}, \eqref{25062018E8} and Lemma \ref{18042018L1} that
$$(\Reelle F)(\Reelle X) = (\Reelle F)(\Imag X) = 0,$$
and therefore, $(\Reelle F)X = 0$. It proves the induction hypothesis in the basic case. Now, we consider $0\le k\le K$ such that 
\begin{equation}\label{25062018E3}
	\forall l\in\{0,\ldots,k-1\},\quad (\Reelle F)(F^lX) = 0.
\end{equation}
Since $\Reelle F$ is a real matrix, we also have 
\begin{equation}\label{25062018E4}
	\forall l\in\{0,\ldots,k-1\},\quad (\Reelle F)(\overline{F^lX}) = 0.
\end{equation}
We recall from \eqref{22062018E5} that the derivative of $\varphi_X$ also writes as 
\begin{equation}\label{06022019E1}
	\forall t\geq0,\quad (\varphi_X)'(t) = 4\sigma(\overline{e^{2itF}X},(\Reelle F)e^{2itF}X).
\end{equation}
It therefore follows from the Leibniz formula applied to \eqref{06022019E1} that for all $t\geq0$,
\begin{align*}
	(\varphi_X)^{(2k+1)}(t) & = 4\sum_{p=0}^{2k}\binom{2k}{p}\sigma(\overline{(2iF)^pe^{2itF}X},(\Reelle F)(2iF)^{2k-p}e^{2itF}X) \\[5pt]
	& = (-1)^k 4^{k+1}\sum_{p=0}^{2k}(-1)^p\binom{2k}{p}\sigma(\overline{F^pe^{2itF}X},(\Reelle F)F^{2k-p}e^{2itF}X).
\end{align*}
This implies that
\begin{equation}\label{22062018E8}
	(\varphi_X)^{(2k+1)}(0) = (-1)^k 4^{k+1}\sum_{p=0}^{2k}(-1)^p\binom{2k}{p}\sigma(\overline{F^pX},(\Reelle F)(F^{2k-p}X)).
\end{equation}
Let $0\le p\le 2k$. When $0\le p\le k-1$, we deduce from \eqref{04122017E6} and \eqref{25062018E4} that 
\begin{equation}\label{25062018E5}
	\sigma(\overline{F^pX},(\Reelle F)(F^{2k-p}X)) = -\sigma((\Reelle F)(\overline{F^pX}),F^{2k-p}X) = 0.
\end{equation}
On the other hand, when $k+1\le p\le 2k$, we have $0\le 2k-p\le k-1$ and it follows from \eqref{25062018E3} that
\begin{equation}\label{25062018E6}
	\sigma(\overline{F^pX},(\Reelle F)(F^{2k-p}X)) = 0.
\end{equation}
As a consequence of \eqref{22062018E2}, \eqref{25062018E7}, \eqref{22062018E8}, \eqref{25062018E5} and \eqref{25062018E6}, we obtain
\begin{equation}\label{18072018E1}
	(\varphi_X)^{(2k+1)}(0) = 4^{k+1}\binom{2k}{k}\sigma(\overline{F^kX},(\Reelle F)(F^kX)) = 0.
\end{equation}
Then, it follows from \eqref{04122017E4}, \eqref{25062018E1} and \eqref{18072018E1} that
\begin{align*}
	\sigma(\overline{F^kX},(\Reelle F)(F^kX)) & = \sigma(\Reelle(F^kX),(\Reelle F)\Reelle(F^kX)) + \sigma(\Imag(F^kX),(\Reelle F)\Imag(F^kX))\\[5pt]
	& = (\Reelle q)(\Reelle(F^kX)) + (\Reelle q)(\Imag(F^kX)) = 0.
\end{align*}
Since $\Reelle q$ is non-negative, this implies that
\begin{equation}\label{06022019E2}
	(\Reelle q)(\Reelle(F^kX)) = (\Reelle q)(\Imag(F^kX)) = 0.
\end{equation}
As above, we deduce from \eqref{21062018E6}, \eqref{29062018E3}, \eqref{06022019E2} and Lemma \ref{18042018L1} that
$$(\Reelle F)(\Reelle (F^k X)) = (\Reelle F)(\Imag (F^kX)) = 0,$$
and therefore $(\Reelle F)(F^kX) = 0$. This ends the induction and proves that \eqref{25062018E9} holds. Then, \eqref{22062018E3} is a consequence of \eqref{25062018E9} and Lemma \ref{17042018L1}.
\end{proof}

\begin{cor}\label{22062018C1} Let $q:\mathbb R^n_x\times\mathbb R^n_{\xi}\rightarrow\mathbb{C}$ be a complex-valued quadratic form with a non-negative real part $\Reelle q\geq0$ and $S$ be its singular space. We assume that $S\ne\mathbb R^{2n}$. Let $0\le k_0\le 2n-1$ be the smallest integer such that \eqref{22062018E1} holds. Then, for all $X\in\mathbb{C}^{2n}\setminus(S+iS)$, there exists a non-negative integer $0\le k_X\le k_0$ such that
$$\partial^{2k_X+1}_t\Reelle\left[Q_t(X)\right]\Big\vert_{t=0} \ne 0,$$
where $Q_t$ is the time-dependent quadratic form associated to $q$ defined in \eqref{30032018E8}. 
\end{cor}

\begin{proof} Let $X\in\mathbb{C}^{2n}\setminus(S+iS)$. As a consequence of \eqref{22062018E1}, there exists $0\le\tilde{k}_X\le k_0$ such that
$$(\Reelle F)(\Imag F)^{\tilde{k}_X}X\ne0.$$
Then, we deduce from Lemma \ref{25062018L1} the existence of $0\le k_X\le k_0$ such that
$$\partial^{2k_X+1}_t\Reelle\left[Q_t(X)\right]\Big\vert_{t=0} \ne 0.$$
This ends the proof of Corollary \ref{22062018C1}.
\end{proof}

The proof of the following result is an adaptation of \cite{MR2664715} (Proposition 3.2):

\begin{prop}\label{25102017T1} Let $q:\mathbb R^n_x\times\mathbb R^n_{\xi}\rightarrow\mathbb C$ be a complex-valued quadratic form with a non-negative real part $\Reelle q\geq0$ and $S$ be its singular space. We assume that $S\ne\mathbb R^{2n}$. Let $0\le k_0\le 2n-1$ be the smallest integer such that \eqref{22062018E1} holds. Then, for all compact set $K$ of $\mathbb S^{4n-1}$ satisfying $(S+iS)\cap K = \emptyset$, there exist some positive constants $c>0$ and $0<T\le 1$ such that for all $0\le t\le T$ and $X \in K$,
$$\Reelle \left[Q_t(X)\right] \geq c t^{2k_0+1},$$
where $Q_t$ is the time-dependent quadratic form associated to $q$ defined in \eqref{30032018E8} and $\mathbb S^{4n-1}$ stands for the Euclidean unit sphere of $\mathbb{C}^{2n}$ identified to $\mathbb R^{4n}$.
\end{prop}

\begin{proof} Let $K$ be a compact set of $\mathbb S^{4n-1}$ satisfying $(S+iS)\cap K = \emptyset$. Let $X\in K$ and $0\le k_X\le k_0$ be a positive integer given by Corollary \ref{22062018C1} satisfying
\begin{equation}\label{25102017E4}
	\partial^{2k_X+1}_t\Reelle\left[Q_t(X)\right]\Big\vert_{t=0} \ne 0.
\end{equation}
We first prove that there exist some positive constants $c_X>0$, $0<t_X<1$ and an open neighborhood $V_X$ of $X$ in $\mathbb S^{4n-1}\cap(S+iS)^c$ such that for all $0<t\le t_X$ and $Y\in V_X$, 
\begin{equation}\label{25102017E3}
	\Reelle \left[Q_t(Y)\right]\geq c_X t^{2k_X+1}.
\end{equation}
We proceed by contradiction and assume that \eqref{25102017E3} does not hold. Then there exist some sequences $(t_N)_{N\geq0}$ of positive real numbers and $(Y_N)_{N\geq0}$ of unit vectors of $\mathbb S^{4n-1}$ satisfying
\begin{equation}\label{25102017E5}
	\lim_{N\rightarrow+\infty}t_N = 0,\quad \lim_{N\rightarrow+\infty}Y_N = X,\quad \lim_{N\rightarrow+\infty}\frac{1}{t_N^{2k_X+1}}\Reelle \left[Q_{t_N}(Y_N)\right] = 0.
\end{equation}
We deduce from \eqref{25102017E5} and Lemma \ref{04042018L1} that
\begin{equation}\label{25102017E6}
	\lim_{N\rightarrow+\infty}\Big(\sup_{0\le t\le t_N}\frac1{t_N^{2k_X+1}}\Reelle\left[Q_t(Y_N)\right]\Big) = 0,
\end{equation}
since the mapping $t\in\mathbb R_+\mapsto \Reelle\left[Q_t(Y_N)\right]$ is non-decreasing. The equality \eqref{25102017E6} can be reformulate as 
\begin{equation}\label{25102017E7}
	\lim_{N\rightarrow+\infty}\Big(\sup_{0\le x\le 1}\vert u_N(x)\vert\Big) = 0,
\end{equation}
with
\begin{equation}\label{25102017E8}
	\forall x\in[0,1],\quad u_N(x) = \frac{1}{t_N^{2k_X+1}}\Reelle \left[Q_{xt_N}(Y_N)\right]\geq0.
\end{equation}
It follows from the Taylor formula that
$$\Reelle \left[Q_t(Y_N)\right] = \sum_{k=0}^{2k_X+1}a_{k,N} t^k + \frac{t^{2k_X+2}}{(2k_X+1)!}\int_0^1(1-\theta)^{2k_X+1}\partial^{2k_X+2}_s\Reelle\left[Q_s(Y_N)\right]\Big\vert_{s=t\theta}d\theta,$$
where
$$\forall k\in\{0,\ldots,2k_X+1\},\quad a_{k,N} = \frac{1}{k!}\partial^k_t\Reelle\left[Q_t(Y_N)\right]\Big\vert_{t=0}.$$
Since $\partial^{2k_X+2}_s\Reelle\left[Q_s\right]$ is a quadratic form whose coefficients depend smoothly on the variable $0\le s\le 1$ and the $Y_N$ are elements of the unit sphere $\mathbb S^{4n-1}$, we notice that
$$\exists C>0,\forall t\in[0,1], \forall N\geq0,\quad \left\vert\int_0^1(1-\theta)^{2k_X+1}\partial^{2k_X+2}_s\Reelle\left[Q_s(Y_N)\right]\Big\vert_{s=t\theta}d\theta\right\vert\le C.$$
Therefore,
$$\Reelle \left[Q_t(Y_N)\right] \underset{t\rightarrow0}{=} \sum_{k=0}^{2k_X+1}a_{k,N} t^k + \mathcal{O}(t^{2k_X+2}),$$
where the term $\mathcal O(t^{2k_X+2})$ can be assumed to be independent on the integer $N$. As a consequence, the following Taylor expansion
\begin{equation}\label{25102017E9}
	u_N(x) = \sum_{k=0}^{2k_X+1}\frac{a_{k,N}}{t_N^{2k_X+1-k}}\ x^k + \mathcal{O}(t_N x^{2k_X+2}),
\end{equation}
holds. It follows from \eqref{25102017E5}, \eqref{25102017E7} and \eqref{25102017E9} that
\begin{equation}\label{25102017E10}
	\lim_{N\rightarrow+\infty}\left[\sup_{0\le x\le 1}\vert p_N(x)\vert\right] = 0,
\end{equation}
where the $p_N$ are the polynomials defined by 
\begin{equation}\label{25102017E11}
	\forall x\in[0,1],\quad p_N(x) = \sum_{k=0}^{2k_X+1}\frac{a_{k,N}}{t_N^{2k_X+1-k}}\ x^k.
\end{equation}
It follows from the equivalence of norms in finite-dimensional vector spaces that
\begin{equation}\label{25102017E12}
	\forall k\in\{0,\ldots,2k_X+1\},\quad \lim_{N\rightarrow+\infty}\frac{a_{k,N}}{t_N^{2k_X+1-k}} = 0.
\end{equation}
In particular, we obtain that
$$\lim_{N\rightarrow+\infty}a_{2k_X+1,N} = 0.$$
However, this is in contraction with the fact that
\begin{align*}
	\lim_{N\rightarrow+\infty}a_{2k_X+1,N} & = \lim_{N\rightarrow+\infty}\frac{1}{(2k_X+1)!}\partial^{2k_X+1}_t\Reelle \left[Q_t(Y_N)\right]\Big\vert_{t=0} \\[5pt]
	& = \frac{1}{(2k_X+1)!}\partial^{2k_X+1}_t\Reelle\left[Q_t(X)\right]\Big\vert_{t=0}\ne0,
\end{align*}
according to \eqref{25102017E4}. Covering the compact set $K$ by finitely many open neighborhoods of the form $V_{X_1},\ldots,V_{X_R}$, and letting 
$$c = \min_{1\le j\le R}c_{X_j}\quad \text{and}\quad T = \min_{1\le j\le R}t_{X_j},$$
we conclude that
$$\forall t\in[0,T], \forall Y\in K,\quad \Reelle \left[Q_t(Y)\right]\geq c t^{2k_0+1}.$$
It ends the proof of Proposition \ref{25102017T1}.
\end{proof}

We can now derive from Proposition \ref{25102017T1} that the time-dependent quadratic form $\Reelle q_t$ defined in \eqref{24042018E1} satisfies coercive estimates on subspaces of the phase space.

\begin{cor}\label{17042018C1} Let $q:\mathbb R^n_x\times\mathbb R^n_{\xi}\rightarrow\mathbb C$ be a complex-valued quadratic form with a non-negative real part $\Reelle q\geq0$, $F$ be its Hamilton map and $S$ its singular space. Let $0\le k_0\le 2n-1$ be the smallest integer such that \eqref{22062018E1} holds. Then, for all linear subspace $\Sigma$ of $\mathbb R^{2n}$ satisfying $S\cap \Sigma = \{0\}$, there exist some positive constants $c>0$ and $0<T\le 1$ such that for all $0\le t\le T$ and $X \in \Sigma$,
\begin{equation}\label{02072018E1}
	(\Reelle q_t)(X) \geq c t^{2k_0+1}\vert X\vert^2,
\end{equation}
where $q_t$ stands for the time-dependent quadratic form associated to $q$ defined in \eqref{24042018E1}.
\end{cor}

\begin{proof} We first assume that $S\ne\mathbb R^{2n}$ and consider $\Sigma$ a linear subspace of $\mathbb R^{2n}$ satisfying $S\cap\Sigma = \{0\}$. Let $t_0>0$ small enough such that
\begin{equation}\label{26062018E4}
	\forall t\in[0,t_0],\quad e^{2itF}+I_n\in\GL_n(\mathbb C).
\end{equation}
For all $0\le t\le t_0$, we define $K_t = \bigcup_{0\le s\le t}\Gamma_s\subset\mathbb C^{2n}$,
where the vector subspaces $\Gamma_s$ are given by
$$\Gamma_s = (e^{2isF}+I_{2n})^{-1}(\Sigma+i\Sigma)\subset\mathbb C^{2n},\quad 0\le s\le t.$$
We first check that for all $0\le t\le t_0$, $K_t$ is a closed subset of $\mathbb C^{2n}$. Let $(Y_p)_p$ be a sequence of $K_t$ converging to $Y\in\mathbb{C}^{2n}$. For all $p\geq0$, there exists $0\le s_p\le t$ and $X_p\in \Sigma+i\Sigma$ such that $Y_p = (e^{2is_pF}+I_{2n})^{-1}X_p$. Since $[0,t]$ is compact, there exists a subsequence $(p')$ such that $(s_{p'})_{p'}$ converges to $s_{\infty}\in[0,t]$. It follows from the continuity of the exponential function that 
$$\lim_{p'\rightarrow+\infty}e^{2is_{p'}F} = e^{2is_{\infty}F},\quad \text{and therefore,}\quad \lim_{p'\rightarrow+\infty}X_{p'} = X,\quad \text{where $X = (e^{2is_{\infty}F}+I_{2n})Y$}.$$ 
Moreover, $X_{p'}\in\Sigma+i\Sigma$ for all $p'$ and $\Sigma+i\Sigma$ is closed, so $X\in \Sigma+i\Sigma$. Finally, $Y = (e^{2is_{\infty}F}+I_{2n})^{-1}X\in K_t$, and $K_t$ is closed. Now, we prove that there exists $t_1>0$ such that
\begin{equation}\label{26062018E1}
	(S+iS)\cap K_{t_1} = \{0\}.
\end{equation}
To that end, we consider $\psi_s$ defined for all $0\le s\le t_0$ by $\psi_s = \dim(\Gamma_s + (S+iS))$. We observe that $\psi_s$ satisfies the estimate
\begin{equation}\label{26062018E3}
	\forall s\in[0,t_0],\quad \psi_s\le \dim \Gamma_s + \dim(S+iS)\le \dim(\Sigma+i\Sigma) + \dim(S+iS).
\end{equation}
Moreover, it follows from the Grassman formula that
\begin{equation}\label{26062018E5}
	\psi_0 = \dim((\Sigma+i\Sigma) + (S+iS)) = \dim(\Sigma+i\Sigma) + \dim(S+iS),
\end{equation}
since $(\Sigma+i\Sigma) \cap (S+iS) = \{0\}$. Since $\psi_s = \Rank M_s$, where $M_s\in M_{n,\psi_0}(\mathbb{C})$ is defined through its column vectors by
$$M_s = \begin{pmatrix}
	(e^{2isF}+I_{2n})^{-1}\mathcal{B}_1 & \mathcal{B}_2
\end{pmatrix},$$ 
with $\mathcal B_1$ a basis of $\Sigma+i\Sigma$ and $\mathcal B_2$ a basis of $S+iS$, we deduce from \eqref{26062018E5} and the lower semi-continuity of $\Rank$ that there exists $t_1>0$ such that 
\begin{equation}\label{26062018E2}
	\forall s\in[0,t_1],\quad \psi_s\geq \dim(\Sigma+i\Sigma) + \dim(S+iS).
\end{equation}
We deduce from \eqref{26062018E3} and \eqref{26062018E2} that for all $0\le s\le t_1$, $\psi_s = \dim \Gamma_s + \dim(S+iS)$,
and the Grassman formula implies that
$$\forall s\in[0,t_1],\quad (S+iS)\cap\Gamma_s = \{0\}.$$
Therefore, \eqref{26062018E1} holds. Since $K_{t_1}\cap\mathbb S^{4n-1}$ is a compact set of $\mathbb S^{4n-1}$ disjointed from $S+iS$, it follows from Proposition \ref{25102017T1} that there exist some positive constants $c_0>0$ and $0<t_2\le1$ such that for all $0\le t\le t_2$ and $X\in K_{t_1}\cap\mathbb S^{4n-1}$,
$$\Reelle\left[-i\sigma(\overline{(e^{2itF}+I_{2n})X},(e^{2itF}-I_{2n})X)\right]\geq c_0t^{2k_0+1}.$$
As a consequence, we have that for all $0\le t\le \min(t_1,t_2)$ and $X\in(\Sigma+i\Sigma)\setminus\{0\}$,
$$\Reelle\left[-i\sigma\left(\overline{\frac{(e^{2itF}+I_{2n})(e^{2itF}+I_{2n})^{-1}X}{\vert (e^{2itF}+I_{2n})^{-1}X\vert}},\frac{(e^{2itF}-I_{2n})(e^{2itF}+I_{2n})^{-1}X}{\vert (e^{2itF}+I_{2n})^{-1}X\vert}\right)\right]\geq c_0t^{2k_0+1},$$
that is, 
\begin{equation}\label{17042018E11}
	\Reelle\left[-i\sigma(\overline X,(e^{2itF}-I_{2n})(e^{2itF}+I_{2n})^{-1}X)\right]\geq c_0t^{2k_0+1}\vert (e^{2itF}+I_{2n})^{-1}X\vert^2.
\end{equation}
Furthermore, it follows from \eqref{26062018E4} that there exists a positive constant $c_1>0$ such that for all $0\le t\le \min(t_1,t_2)$ and $X\in\Sigma+i\Sigma$,
\begin{equation}\label{17042018E13}
	\vert (e^{2itF}+I_{2n})^{-1}X\vert^2\geq c_1\vert X\vert^2,
\end{equation}
since $0<t_1<t_0$. We deduce from \eqref{24042018E1}, \eqref{17042018E11} and \eqref{17042018E13} that there exist some positive constants $c>0$ and $0<T\le 1$ such that
$$\forall t\in[0,T], \forall X\in\Sigma,\quad (\Reelle q_t)(X)\geq ct^{2k_0+1}\vert X\vert^2,$$
since
$$-i(e^{2itF}-I_{2n})(e^{2itF}+I_{2n})^{-1} = \tan(tF),\quad 0\le t\ll1.$$
This ends the proof of Corollary \ref{17042018C1} when $S\ne\mathbb R^{2n}$. If $S=\mathbb R^{2n}$, then the only linear subspace $\Sigma\subset\mathbb R^{2n}$ satisfying $S\cap\Sigma = \{0\}$ is $\Sigma = \{0\}$ and \eqref{02072018E1} is trivial.
\end{proof}

\subsection{Variables of the Mehler symbol} In this subsection, we investigate the variables on which the time-dependent quadratic form $q_t$ depends, in order to check that $q_t$ satisfies the condition \eqref{19042018E5} for the pair $(I,J)$, with $I,J\subset\{1,\ldots,n\}$, when the singular space $S$ satisfies $S^{\perp} = \mathbb R^n_I\times\mathbb R^n_J$, the orthogonality being taken with respect to the canonical Euclidean structure of $\mathbb R^{2n}$.

\begin{lem}\label{25062018L3} Let $q:\mathbb R^n_x\times\mathbb R^n_{\xi}\rightarrow\mathbb C$ be a complex-valued quadratic form with a non-negative real part $\Reelle q\geq0$, $F$ be its Hamilton map and $S$ its singular space. Then,
\begin{equation}\label{25062018E10}
	\forall k\geq0, \forall X\in S,\quad \Reelle(F^{2k+1})X = 0.
\end{equation}
\end{lem}

\begin{proof} We first check that
\begin{equation}\label{25062018E13}
	(\Reelle F)S = \{0\}\quad \text{and}\quad (\Imag F)S\subset S.
\end{equation}
It follows from the definition of the singular space $S$, see \eqref{04122017E7}, that $(\Reelle F)S = \{0\}$. Then, the Cayley-Hamilton theorem applied to $\Imag F$ shows that
$$(\Imag F)^kX\in\Span(X,\ldots,(\Imag F)^{2n-1}X),\quad X\in\mathbb R^{2n},\ k\geq0,$$
and as a consequence, the singular space $S$ is actually equal to the infinite intersection of kernels
$$S = \bigcap_{j=0}^{+\infty}\Ker(\Reelle F(\Imag F)^j)\cap\mathbb R^{2n},$$
which proves that $(\Imag F)S\subset S$. Therefore, \eqref{25062018E13} holds. We can now derive \eqref{25062018E10} from \eqref{25062018E13}. Let $k\geq0$. A direct computation shows that
\begin{equation}\label{24042018E2}
	\Reelle (F^{2k+1}) = \sum (-1)^j(\Imag F)^{j_1}(\Reelle F)^{j_2}\ldots(\Reelle F)^{j_{s-1}}(\Imag F)^{j_s},
\end{equation}
where the sum is finite with $0\le j\le k$, the $j_l$ are non-negative integers, where $1\le l\le s$, and each product is composed of $2k+1$ matrices including $2j$ terms $\Imag F$. In particular, each product appearing in \eqref{24042018E2} contains at least one matrix $\Reelle F$. It follows from \eqref{25062018E13} that for all $X\in S$,
$$\Reelle (F^{2k+1})X = \sum (-1)^j(\Imag F)^{j_1}(\Reelle F)^{j_2}\ldots(\Reelle F)^{j_{s-1}}(\Imag F)^{j_s}X = 0.$$
This ends the proof of Lemma \ref{25062018L3}.
\end{proof}

\begin{lem}\label{24042018L1} Let $q:\mathbb R^n_x\times\mathbb R^n_{\xi}\rightarrow\mathbb{C}$ be a complex-valued quadratic form with a non-negative real part $\Reelle q\geq0$, $S$ be its singular space and $\Sigma$ be a linear subspace of $\mathbb R^{2n}$ satisfying $S+\Sigma=\mathbb R^{2n}$. Then, there exists $t_0>0$ such that for all $0\le t< t_0$ and all decomposition $X = X_S + X_{\Sigma}\in\mathbb R^{2n}$ with $X_S\in S$ and $X_{\Sigma}\in\Sigma$ (not unique),
$$(\Reelle q_t)(X) = (\Reelle q_t)(X_{\Sigma}),$$
where $q_t$ is the time-dependent quadratic form associated to $q$ defined in \eqref{24042018E1}.
\end{lem}

\begin{proof} Let $Y\in S$ and $Z\in\mathbb R^{2n}$. We recall that the tangent function $\tan$ is analytic and that its Taylor expansion writes for all matrices $M\in M_n(\mathbb{C})$ such that $\Vert M\Vert< \pi/2$ as
\begin{equation}\label{25062018E11}
	\tan M = \sum_{k=0}^{+\infty}a_kM^{2k+1},
\end{equation}
where all the coefficients $a_k$ are positive real numbers. The continuity of the symplectic form and \eqref{25062018E11} imply that
$$\forall t\in(-t_0,t_0),\quad \Reelle\left[\sigma(Z,\tan(tF)Y)\right] = \sum_{k=0}^{+\infty}a_k\sigma(Z,\Reelle(F^{2k+1})Y)t^{2k+1},$$
where $t_0 = \pi/(2\Vert F\Vert)$. Since $Y\in S$, we deduce from Lemma \ref{25062018L3} that $\Reelle(F^{2k+1})Y = 0$ for all $k\geq0$, and therefore, we have
\begin{equation}\label{2062018E12}
	\forall t\in(-t_0,t_0),\quad \Reelle\left[\sigma(Z,\tan(tF)Y)\right] = 0.
\end{equation}
By using the skew-symmetry of the Hamilton map with respect to the symplectic form, see \eqref{04122017E6}, and the skew-symmetry of the symplectic form, we obtain that for all $t\in(-t_0,t_0)$,
\begin{multline}\label{25062018E15}
	\sigma(Y,\tan(tF)Z) = \sum_{k=0}^{+\infty}a_k\sigma(Y,F^{2k+1}Z)t^{2k+1} \\[5pt]
	= - \sum_{k=0}^{+\infty}a_k\sigma(F^{2k+1}Y,Z)t^{2k+1} = - \sigma(\tan(tF)Y,Z)
	= \sigma(Z,\tan(tF)Y).
\end{multline}
It follows from \eqref{2062018E12} that
$$\forall t\in(-t_0,t_0),\quad \Reelle\left[\sigma(Y,\tan(tF)Z)\right] = 0.$$
As a consequence, we have that for all $0\le t\le t_0$ and $X = X_S+X_{\Sigma}$ with $X_S\in S$ and $X_{\Sigma}\in\Sigma$,
$$ (\Reelle q_t)(X) = \Reelle\left[\sigma(X,\tan(tF)X)\right] = \Reelle\left[\sigma(X_{\Sigma},\tan(tF)X_{\Sigma})\right] = (\Reelle q_t)(X_{\Sigma}),$$
by bilinearity of the symplectic form. This ends the proof of Lemma \ref{24042018L1}.
\end{proof}

Let $q:\mathbb R^n_x\times\mathbb R^n_{\xi}\rightarrow\mathbb C$ be a complex-valued quadratic form with a non-negative real part $\Reelle q\geq0$. We assume that there exist some subsets $I,J\subset\{1,\ldots,n\}$ such that $S^{\perp} = \mathbb R^n_I\times\mathbb R^n_J$, the orthogonality being taken with respect to the canonical Euclidean structure of $\mathbb R^{2n}$. We deduce from Lemma \ref{24042018L1}, that there exists $t_0>0$ such that for all $0\le t<t_0$ and $X\in\mathbb R^{2n}$,
$$(\Reelle q_t)(X) = (\Reelle q_t)(X_{I,J}),$$
where $q_t$ is the time-dependent quadratic form associated to $q$ defined in \eqref{24042018E1}, and where $X_{I,J}$ stands for the component in $\mathbb R^n_I\times\mathbb R^n_J$ of the vector $X\in\mathbb R^{2n}$ according to the orthogonal decomposition $S\oplus^{\perp}(\mathbb R^n_I\times\mathbb R^n_J)=\mathbb R^{2n}$. The condition \eqref{19042018E5} is therefore always satisfied for real-part the time-dependent quadratic form $q_t$. However, as pointed out by the following example, we observe that the condition \eqref{19042018E5} is not satisfied in general for the time-dependent quadratic form $q_t$, and therefore, Theorem \ref{18042018T1} cannot be directly applied.

\begin{ex}\label{ex4} We consider the Kolmogorov operator 
$$P = -\partial^2_v +  v\partial_x,\quad (x,v)\in\mathbb R^2.$$
The Weyl symbol of $P$ is given by the quadratic form
$$q(x,v,\xi,\eta) = \eta^2 + iv\xi,\quad (x,v,\xi,\eta)\in\mathbb R^4,$$
and a direct computation shows that the Hamilton map and the singular space of $q$ are respectively given by
$$F = \frac12\begin{pmatrix}
	0 & i & 0 & 0 \\
	0 & 0 & 0 & 2 \\
	0 & 0 & 0 & 0 \\
	0 & 0 & -i & 0
\end{pmatrix}\quad \text{and}\quad S = \mathbb R_x\times\mathbb R_v\times\{0_{\mathbb R_{\xi}}\}\times\{0_{\mathbb R_{\eta}}\}.$$
Moreover, its Mehler symbol \eqref{24042018E1} is given for all $t\geq0$ and $(x,v,\xi,\eta)\in\mathbb R^4$ by
$$q_t(x,v,\xi,\eta) = \frac{t^3}{12}\xi^2 + t\eta^2 + itv\xi.$$
Notice that $S\oplus^{\perp}(\mathbb R^2_{\emptyset}\times\mathbb R^2_{\{1,2\}}) = \mathbb R^4$, the orthogonality being taken with respect to the canonical Euclidean structure of $\mathbb R^4$. However, we have that for all $t>0$,
$$it = q_t(0,1,1,0)\ne q_t(0,0,1,0) = 0,$$
and $(0,0,1,0)$ is the component in $\mathbb R^2_{\emptyset}\times\mathbb R^2_{\{1,2\}}$ of the vector $(0,1,1,0)\in\mathbb R^2\times\mathbb R^2$ with respect to the above orthogonal decomposition of the phase space $\mathbb R^4$.
\end{ex}

The issue pointed out by Example \ref{ex4} is that the imaginary part $\Imag q_t$ may depend on variables in the singular space $S$, whereas the real part $\Reelle q_t$ cannot according to Lemma \ref{24042018L1}. In order to ensure that condition \eqref{19042018E5} actually holds, we add an extra assumption on the quadratic form $q$ so that this case does not occur. Before introducing this condition, we provide another example:

\begin{ex}\label{ex3} We consider the Kramers-Fokker-Planck operator without external potential
$$K = -\partial^2_v + \frac14v^2 + v\partial_x,\quad (x,v)\in\mathbb R^2.$$
The Weyl symbol of $K$ is the following quadratic form
$$q(x,v,\xi,\eta) = \eta^2+\frac14v^2+iv\xi,\quad (x,v,\xi,\eta)\in\mathbb R^4.$$
We recall from Example \ref{ex1} that the Hamilton map and the singular space of $q$ are given by 
$$F = \frac12\begin{pmatrix}
	0 & i & 0 & 0 \\
	0 & 0 & 0 & 2 \\
	0 & 0 & 0 & 0 \\
	0 & -\frac12 & -i & 0
\end{pmatrix}\quad \text{and}\quad S = \mathbb R_x\times\{0_{\mathbb R_v}\}\times\{0_{\mathbb R_{\xi}}\}\times\{0_{\mathbb R_{\eta}}\}.$$
We observe that $F$ and $S$ satisfy $S\subset\Ker(\Imag F)$, whereas this inclusion does not hold in Example \ref{ex4}. Moreover, an algebraic computation shows that the Mehler symbol \eqref{24042018E1} of $q$ is given for all $t\geq0$ and $(x,v,\xi,\eta)\in\mathbb R^4$ by
$$q_t(x,v,\xi,\eta) = \frac12\Big(\tanh\frac t2\Big)v^2 + \Big(t-2\tanh\frac t2\Big)\xi^2 + 2\Big(\tanh\frac t2\Big)\eta^2 + 2i\Big(\tanh\frac t2\Big)v\xi.$$
We notice in this case that the variables appearing in the imaginary part $\Imag q_t$ do appear also in the real part $\Reelle q_t$.
\end{ex}

In the following lemma, we prove that the condition $S\subset\Ker(\Imag F)$ pointed out in Example \ref{ex3} is actually sufficient to ensure that the condition \eqref{19042018E5} holds.

\begin{lem}\label{24042018L2} Let $q:\mathbb R^n_x\times\mathbb R^n_{\xi}\rightarrow\mathbb{C}$ be a complex-valued quadratic form with a non-negative real part $\Reelle q\geq0$, whose singular space $S$ satisfies $S\subset \Ker(\Imag F)$, where $F$ stands for the Hamilton map of $q$. Let $\Sigma$ be a linear subspace of $\mathbb R^{2n}$ satisfying $S+\Sigma = \mathbb R^{2n}$. Then, there exists $t_0>0$ such that for all $0\le t< t_0$ and all decomposition $X = X_S+X_{\Sigma}\in\mathbb R^{2n}$, with $X_S\in S$ and $X_{\Sigma}\in\Sigma$,
$$q_t(X) = q_t(X_{\Sigma}),$$
where $q_t$ is the time-dependent quadratic form associated to $q$ defined in \eqref{24042018E1}. 
\end{lem}

\begin{proof} Since $(\Reelle F)S = \{0\}$ according to the definition of the singular space \eqref{04122017E7}, the assumption $S\subset \Ker(\Imag F)$ implies that $S\subset\Ker F$. It then follows from \eqref{25062018E11} that 
\begin{equation}\label{25062018E14}
	\forall t\in(-t_0,t_0),\forall Y\in S,\quad \tan(tF)Y = \sum_{k=0}^{+\infty}a_k(tF)^{2k+1}Y = 0,
\end{equation}
where $t_0 = \pi/(2\Vert F\Vert)$ and all the coefficients $a_k$ are positive real numbers. As a consequence of \eqref{25062018E15} and \eqref{25062018E14}, we notice that for all $t\in(-t_0,t_0)$, $Y\in S$ and $Z\in\mathbb R^{2n}$,
\begin{equation}\label{25062018E16}
	\sigma(Z,\tan(tF)Y) = \sigma(Y,\tan(tF)Z) = 0.
\end{equation}
Finally, we deduce from \eqref{25062018E16} and the bilinearity of the symplectic form $\sigma$ that for all $0\le t<t_0$ and $X = X_S+X_{\Sigma}$ with $X_S\in S$ and $X_{\Sigma}\in\Sigma$,
$$q_t(X) = \sigma(X,\tan(tF)X) = \sigma(X_{\Sigma},\tan(tF)X_{\Sigma}) = q_t(X_{\Sigma}).$$
This ends the proof of Lemma \ref{24042018L2}.
\end{proof}

\subsection{Proof of Theorem \ref{20112017T1}} The aim of this subsection is to prove Theorem \ref{20112017T1}. Let $q:\mathbb R^n_x\times\mathbb R^n_{\xi}\rightarrow\mathbb{C}$ be a complex-valued quadratic form with a non-negative real part $\Reelle q\geq0$. We assume that there exist some subsets $I,J\subset\{1,\ldots,n\}$ such that $S^{\perp}=\mathbb R^n_I\times\mathbb R^n_J$, the orthogonality being taken with respect to the canonical Euclidean structure of $\mathbb R^{2n}$. We also assume that the inclusion $S\subset\Ker(\Imag F)$ holds, where $F$ denotes the Hamilton map of $q$. Notice that the coefficients of the time-dependent quadratic form
$q_t:X\in\mathbb R^{2n}\mapsto\sigma(X,\tan(tF)X)\in\mathbb C$,
defined for $0\le t\le t_0$, with $0<t_0\ll1$ small enough, depend continuously on the time variable $t$. Since $S\oplus^{\perp}(\mathbb R^n_I\times\mathbb R^n_J) = \mathbb R^{2n}$ and $S\subset\Ker(\Imag F)$, it follows from Corollary \ref{17042018C1} and Lemma \ref{24042018L2} that there exist some positive constants $c>0$ and $0<t_0<1$ such that
$$\forall t\in[0,t_0], \forall X\in\mathbb R^n_I\times\mathbb R^n_J,\quad (\Reelle q_t)(X_{I,J})\geq ct^{2k_0+1}\vert X\vert^2,$$
and 
$$\forall t\in[0,t_0], \forall X\in\mathbb R^{2n},\quad q_t(X) = q_t(X_{I,J}),$$
where $0\le k_0\le 2n-1$ is the smallest integer such that \eqref{22062018E1} holds and $X_{I,J}$ stands for the component in $\mathbb R^n_I\times\mathbb R^n_J$ of the vector $X\in\mathbb R^{2n}$ with respect to the decomposition $S\oplus^{\perp}(\mathbb R^n_I\times\mathbb R^n_J) = \mathbb R^{2n}$. As a consequence, we deduce from Theorem \ref{18042018T1} that there exist some positive constants $C>1$ and $0<t_1<t_0$ such that for all $(\alpha,\beta)\in\mathbb N^n_I\times\mathbb N^n_J$, $0< t\le t_1$ and $u\in L^2(\mathbb R^n)$,
\begin{equation}\label{20042018E9}
	\big\Vert x^{\alpha}\partial^{\beta}_x(e^{-q_t})^wu\big\Vert_{L^2(\mathbb R^n)}\le\frac{C^{1+\vert\alpha\vert+\vert\beta\vert}}{t^{(2k_0+1)(\vert\alpha\vert+\vert\beta\vert+s)}}\ (\alpha!)^{\frac12}\ (\beta!)^{\frac12}\ \Vert u\Vert_{L^2(\mathbb R^n)},
\end{equation}
where $s = 9n/4+2\lfloor n/2\rfloor + 3$. Moreover, $t_0>0$ is chosen such that $\det(\cos(tF))\ne0$ for all $0\le t\le t_0$, and the Mehler formula \cite{MR1339714} (Theorem 4.2) provides that
\begin{equation}\label{20042018E10}
	\forall t\in[0,t_0],\quad e^{-tq^w} = \frac{(e^{-q_t})^w}{\sqrt{\det(\cos(tF))}},
\end{equation}
with 
$$\sqrt{\det(\cos(tF))} = \exp\bigg(\frac12\Log\big(\det(\cos(tF))\big)\bigg),$$
where $\Log$ denotes the principal determination of the complex logarithm on $\mathbb C\setminus\mathbb R_-$. It follows from \eqref{20042018E9} and \eqref{20042018E10} that for all $(\alpha,\beta)\in\mathbb N^n_I\times\mathbb N^n_J$, $0< t\le t_1$ and $u\in L^2(\mathbb R^n)$,
$$\big\Vert x^{\alpha}\partial^{\beta}_x(e^{-tq^w}u)\big\Vert_{L^2(\mathbb R^n)}\le\frac{MC^{1+\vert\alpha\vert+\vert\beta\vert}}{t^{(2k_0+1)(\vert\alpha\vert+\vert\beta\vert+s)}}\ (\alpha!)^{\frac12}\ (\beta!)^{\frac12}\ \Vert u\Vert_{L^2(\mathbb R^n)},$$
where
$$M = \max_{0\le t\le t_0}\bigg\vert\frac1{\sqrt{\det(\cos(tF))}}\bigg\vert.$$
This ends the proof of Theorem \ref{20112017T1}.

\section{Null-controllability of quadratic differential equations}
\label{controllability}

This section is devoted to the proof of Theorem \ref{08122017T1}. Let $q:\mathbb R^n_x\times\mathbb R^n_{\xi}\rightarrow\mathbb{C}$ be a complex-valued quadratic form with a non-negative real part $\Reelle q\geq0$ and $\omega\subset\mathbb R^n$ be a thick subset (see Definition \ref{10092018D1}). We assume that $q$ is diffusive as defined in Definition \ref{2} and that the singular space $S$ and the Hamilton map $F$ of $q$ satisfy $S\subset\Ker(\Imag F)$. To establish the observability estimate \eqref{08122017E9}, we use the following theorem which is the result of \cite{BEPS} (Theorem 3.2) applied for fixed control supports and semigroups, whose proof is based on an adaptation of the Lebeau-Robbiano strategy.

\begin{thm}\label{08122017T3} Let $\tilde q:\mathbb R^n_x\times\mathbb R^n_{\xi}\rightarrow\mathbb C$ be a complex-valued quadratic form with a non-negative real part $\Reelle \tilde q\geq0$, $\omega$ be a Borel subset of $\mathbb R^n$ and $(\pi_k)_{k\geq1}$ be a family of orthogonal projections defined on $L^2(\mathbb R^n)$. Assume that there exist $c_1, c'_1, c_2, c'_2, a,b,t_0,m_1>0$ and $m_2\geq0$ some constants with $a<b$ such that the following spectral inequality 
\begin{equation}\label{08122017E3}
	\forall u\in L^2(\mathbb R^n), \forall k\geq1,\quad \Vert\pi_k u\Vert_{L^2(\mathbb R^n)}\le c'_1e^{c_1k^a}\Vert\pi_ku\Vert_{L^2(\omega)},
\end{equation}
and the following dissipation estimate
\begin{equation}\label{08122017E4}
	\forall u\in L^2(\mathbb R^n), \forall k\geq1, \forall 0<t<t_0,\quad \big\Vert(1-\pi_k)(e^{-t\tilde q^w}u)\big\Vert_{L^2(\mathbb R^n)}\le\frac1{c'_2t^{m_2}}e^{-c_2t^{m_1}k^b}\Vert u\Vert_{L^2(\mathbb R^n)},
\end{equation}
hold. Then, there exists a positive constant $C>1$ such that the following observability estimate holds
\begin{equation}\label{08122017E5}
	\forall T>0, \forall u\in L^2(\mathbb R^n),\quad \big\Vert e^{-T\tilde q^w}u\big\Vert^2_{L^2(\mathbb R^n)}\le C\exp\Big(\frac C{T^{\frac{am_1}{b-a}}}\Big)\int_0^T\big\Vert e^{-t\tilde q^w}u\big\Vert^2_{L^2(\omega)}\ dt.
\end{equation}
\end{thm}

Notice that under the assumptions of Theorem \ref{08122017T3}, a strict application of \cite{BEPS} (Theorem 3.2) only provides the existence of positive constants $C_1>0$ and $T_0>0$ such that for all $0<T\le T_0$ and $u\in L^2(\mathbb R^n)$,
\begin{equation}\label{29082019E1}
	\big\Vert e^{-T\tilde q^w}u\big\Vert^2_{L^2(\mathbb R^n)}\le\exp\Big(\frac{C_1}{T^{\frac{am_1}{b-a}}}\Big)\int_0^T\big\Vert e^{-t\tilde q^w}u\big\Vert^2_{L^2(\omega)}\ dt.
\end{equation}
However, setting $C_2 = \exp(C_1/T^{\frac{am_1}{b-a}}_0)>1$ and by using the contractivity property of the semigroup $(e^{-t\tilde q^w})_{t\geq0}$, we deduce from \eqref{29082019E1} that for all $T>T_0$ and $u\in L^2(\mathbb R^n)$,
\begin{multline}\label{29082019E2}
\Vert e^{-T\tilde q^w}u\big\Vert^2_{L^2(\mathbb R^n)}\le \Vert e^{-T_0\tilde q^w}u\big\Vert^2_{L^2(\mathbb R^n)}
\le C_2\int_0^{T_0}\big\Vert e^{-t\tilde q^w}u\big\Vert^2_{L^2(\omega)}\ dt \\
\le C_2\exp\Big(\frac{C_1}{T^{\frac{am_1}{b-a}}}\Big)\int_0^T\big\Vert e^{-t\tilde q^w}u\big\Vert^2_{L^2(\omega)}\ dt.
\end{multline}
The estimates \eqref{29082019E1} and \eqref{29082019E2} therefore imply that 
$$\forall T>0, \forall u\in L^2(\mathbb R^n),\quad \big\Vert e^{-T\tilde q^w}u\big\Vert^2_{L^2(\mathbb R^n)}\le C\exp\Big(\frac C{T^{\frac{am_1}{b-a}}}\Big)\int_0^T\big\Vert e^{-t\tilde q^w}u\big\Vert^2_{L^2(\omega)}\ dt,$$
with $C= \sup(C_1,C_2)>1$, which is the desired observability estimate \eqref{08122017E5}.

Let $\pi_k:L^2(\mathbb R^n)\rightarrow E_k$, $k\geq1$, be the orthogonal frequency cutoff projection onto the closed subspace
\begin{equation}\label{23042018E1}
	E_k = \big\{u\in L^2(\mathbb R^n),\quad \Supp\widehat u\subset[-k,k]^n\big\}.
\end{equation}
According to Theorem \ref{08122017T3}, it is sufficient to prove a spectral inequality as \eqref{08122017E3} and a dissipation estimate as \eqref{08122017E4} for the orthogonal projections $\pi_k$ to obtain the observability estimate \eqref{08122017E9}.

\subsection{Spectral inequality} The following theorem is proved by O. Kovrijkine in \cite{MR1840110} (Theorem 3):

\begin{thm}\label{08122017T2} There exists a universal constant $K$ depending only on the dimension $n$ that may be assumed to be greater or equal to $e$ such that for any parallelepiped $J$ with sides parallel to the coordinate axis and of positive lengths $b_1,\ldots,b_n$ and $\omega$ be a $(\gamma,a)$-thick set, then
$$\forall u\in L^2(\mathbb R^n),\ \Supp\widehat{u}\subset J,\quad\Vert u\Vert_{L^2(\mathbb R^n)}\le \left(\frac{K^n}{\gamma}\right)^{K(\langle a,b\rangle + n)}\Vert u\Vert_{L^2(\omega)},$$
where $b=(b_1,\ldots,b_n)$.
\end{thm}

Let $u\in L^2(\mathbb R^n)$ and $k\geq1$. It follows from \eqref{23042018E1} (the definition of $\pi_k$) that $\pi_ku\in L^2(\mathbb R^n)$ and $\widehat{\pi_ku}$ is supported in $[-k,k]^n$. Therefore, we deduce from Theorem \ref{08122017T2} that
\begin{equation}\label{08122017E21}
	\forall u\in L^2(\mathbb R^n), \forall k\geq1,\quad \Vert\pi_ku\Vert_{L^2(\mathbb R^n)}\le c'_1e^{c_1k}\Vert\pi_ku\Vert_{L^2(\omega)},
\end{equation}
where the two positive constants $c_1>0$ and $c'_1>0$ are given by
$$c_1 = 2K\vert a\vert\ln\bigg(\frac{K^n}{\gamma}\bigg)>0,\quad c'_1 = \bigg(\frac{K^n}{\gamma}\bigg)^{nK}>0.$$

\subsection{Dissipation estimate} In order to prove the dissipation estimate \eqref{08122017E4}, we need the following lemma, whose proof is taken from \cite{MR2668420} (Proposition 6.1.5) and recalled for the convenience of the reader:

\begin{lem}\label{19042018L2} There exist some positive constants $C_1>0$ and $C_2>0$ only depending on the dimension $n$ such that for all $\Lambda_1,\Lambda_2>0$ and $f\in L^2(\mathbb R^n)$ satisfying
\begin{equation}\label{19042018E4}
	\forall\alpha\in\mathbb N^n,\quad \Vert \partial_x^{\alpha}f\Vert_{L^2(\mathbb R^n)}\le \Lambda_1 \Lambda_2^{\vert\alpha\vert} (\alpha!)^{\frac12},
\end{equation}
we have
$$\big\Vert e^{C_2\Lambda_2^{-2}\vert D_x\vert^2}f\big\Vert_{L^2(\mathbb R^n)}\le C_1\Lambda_1.$$
\end{lem}

\begin{proof} Let $f\in L^2(\mathbb R^n)$ satisfying \eqref{19042018E4}, where $\Lambda_1,\Lambda_2>0$. First, it follows from \eqref{19042018E4}, \eqref{20062018E1}, \eqref{24112017E3}, \eqref{24042018E5} and the Plancherel theorem that for all $N\in\mathbb N$,
\begin{multline}\label{20032018E1}
	\big\Vert\vert\xi\vert^N\widehat f\big\Vert_{L^2(\mathbb R^n)}
	\le n^N\sum_{\vert\alpha\vert = N}\big\Vert\xi^{\alpha}\widehat f\big\Vert_{L^2(\mathbb R^n)} 
	\le n^N(2\pi)^{\frac n2}\sum_{\vert\alpha\vert = N}\big\Vert\partial^{\alpha}_xf\big\Vert_{L^2(\mathbb R^n)} \\[5pt]
	\le n^N(2\pi)^{\frac n2}\sum_{\vert\alpha\vert = N}\Lambda_1 \Lambda_2^{\vert\alpha\vert} (\alpha!)^{\frac 12}
	\le n^N(2\pi)^{\frac n2}\binom{N+n-1}N\Lambda_1 \Lambda_2^N (N!)^{\frac 12} \\
	\le n^N(2\pi)^{\frac n2}2^{n+N-1}\Lambda_1\Lambda_2^N (N!)^{\frac 12}
	\le 2^{n-1}(2\pi)^{\frac n2}\Lambda_1(2n\Lambda_2)^N N^{\frac N2},
\end{multline}
since $\alpha!\le N!$ for all $\alpha\in\mathbb N^n$ such that $\vert\alpha\vert = N$. We deduce from \eqref{20032018E1} that
\begin{multline}\label{11122017E1}
	\Big\Vert\exp\left[\frac{\vert\xi\vert^2}{16en^2\Lambda_2^2}\right]\widehat f\Big\Vert_{L^2(\mathbb R^n)}
	\le \sum_{N=0}^{+\infty}\frac1{2^N}\frac1{(2n\Lambda_2)^{2N}}\frac1{(2e)^NN!}\big\Vert\vert\xi\vert^{2N}\widehat f\big\Vert_{L^2(\mathbb R^n)} \\
	\le 2^{n-1}(2\pi)^{\frac n2}\Lambda_1 \sum_{N=0}^{+\infty}\frac1{2^N}\frac{(2N)^N}{(2e)^NN!}.
\end{multline}
Moreover, we have $N^N\le e^NN!$ for all $N\geq0$ from (0.3.12) in \cite{MR2668420}, and \eqref{11122017E1} implies that
\begin{equation}\label{20032018E2}
	\Big\Vert\exp\left[\frac{\vert \xi\vert^2}{16en^2\Lambda_2^2}\right]\widehat f\Big\Vert_{L^2(\mathbb R^n)}
	\le 2^n(2\pi)^{\frac n2}\Lambda_1.
\end{equation}
Setting $C_1 = 2^n$ and $C_2 = 1/(16en^2)$, we deduce from \eqref{20032018E2} and the Plancherel theorem that
$$\big\Vert e^{C_2\Lambda_2^{-2}\vert D_x\vert^2}f\big\Vert_{L^2(\mathbb R^n)}\le C_1\Lambda_1.$$
This ends the proof of Lemma \ref{19042018L2}.
\end{proof}

We can now derive a dissipation estimate for the semigroup $(e^{-tq^w})_{t\geq0}$. Since the quadratic form $q$ is diffusive and $S\subset\Ker(\Imag F)$, we deduce from Theorem \ref{20112017T1} and Definition \ref{2} that there exist some positive constants $C>1$ and $0<t_0<1$ such that for all $\alpha\in\mathbb N^n$, $0<t\le t_0$ and $u\in L^2(\mathbb R^n)$,
$$\big\Vert\partial^{\alpha}_x(e^{-tq^w}u)\big\Vert_{L^2(\mathbb R^n)}\le\frac{C^{1+\vert\alpha\vert}}{t^{(2k_0+1)(\vert\alpha\vert+s)}}\ (\alpha!)^{\frac12}\ \Vert u\Vert_{L^2(\mathbb R^n)},$$
where $0\le k_0\le 2n-1$ is the smallest integer satisfying \eqref{22062018E1} and $s = 9n/4+2\lfloor n/2\rfloor+3$ (see the remark before Theorem \ref{05122017T2}). It follows from Lemma \ref{19042018L2} that there exist some positive constants $C_1>0$ and $C_2>0$ such that for all $0<t\le t_0$ and $u\in L^2(\mathbb R^n)$,
$$\big\Vert e^{C_2t^{2(2k_0+1)}\vert D_x\vert^2}e^{-tq^w}u\big\Vert_{L^2(\mathbb R^n)}\le \frac{C_1}{t^{(2k_0+1)s}}\ \Vert u\Vert_{L^2(\mathbb R^n)}.$$
As a consequence, it follows that for all $k\geq1$, $0<t\le t_0$ and $u\in L^2(\mathbb R^n)$,
\begin{multline*}
	\big\Vert(1-\pi_k)e^{-tq^w}u\big\Vert_{L^2(\mathbb R^n)} 
	= \frac1{(2\pi)^{\frac n2}}\big\Vert\mathbbm1_{\mathbb R^{2n}\setminus[-k,k]^n}\ \widehat{e^{-tq^w}u}\big\Vert_{L^2(\mathbb R^n)} \\[5pt]
	\le e^{-C_2t^{2(2k_0+1)}k^2}\big\Vert e^{C_2t^{2(2k_0+1)}\vert D_x\vert^2}e^{-tq^w}u\big\Vert_{L^2(\mathbb R^n)}
	\le \frac{C_1}{t^{(2k_0+1)s}}\ e^{-C_2t^{2(2k_0+1)}k^2}\Vert u\Vert_{L^2(\mathbb R^n)}.
\end{multline*}
Setting $c_2 = C_2$ and $c'_2 = \frac1{C_1}$, we obtain that for all $k\geq1$, $0<t\le t_0$ and $u\in L^2(\mathbb R^n)$,
\begin{equation}\label{08122017E20}
	\big\Vert(1-\pi_k)e^{-tq^w}u\big\Vert_{L^2(\mathbb R^n)}\le\frac1{c'_2t^{(2k_0+1)s}}\ e^{-c_2t^{2(2k_0+1)}k^2}\Vert u\Vert_{L^2(\mathbb R^n)}.
\end{equation}

\subsection{Proof of Theorem \ref{08122017T1}} We deduce from \eqref{08122017E21}, \eqref{08122017E20} and Theorem \ref{08122017T3} that there exists a positive constant $C>1$ such that for all $T>0$ and $u\in L^2(\mathbb R^n)$,
$$\big\Vert e^{-Tq^w}u\big\Vert^2_{L^2(\mathbb R^n)}\le C\exp\Big(\frac C{T^{2(2k_0+1)}}\Big)\int_0^T\big\Vert e^{-tq^w}u\big\Vert^2_{L^2(\omega)}\ dt.$$
It proves the observability estimate \eqref{08122017E9} and ends the proof of Theorem \ref{08122017T1}.

\section{Application to generalized Ornstein-Uhlenbeck operators}
\label{GOU}

In this section, we consider the generalized Ornstein-Uhlenbeck operators
\begin{align}\label{20042018E17}
	P & = \frac12\sum_{i,j=1}^nq_{i,j}\partial^2_{x_i,x_j} - \frac12\sum_{i,j=1}^nr_{i,j}x_ix_j + \sum_{i,j=1}^nb_{i,j}x_j\partial_{x_i} \\
	& = \frac12\Tr(Q\nabla^2_x) - \frac12\langle Rx,x\rangle + \langle Bx,\nabla_x\rangle, \nonumber
\end{align}
equipped with the domain
$$D(P) = \big\{u\in L^2(\mathbb R^n),\quad Pu\in L^2(\mathbb R^n)\big\},$$
where $x\in\mathbb R^n$, $n\geq1$, $Q = (q_{i,j})_{1\le i,j\le n}$, $R=(r_{i,j})_{1\le i,j\le n}$ and $B=(b_{i,j})_{1\le i,j\le n}$ are real $n\times n$ matrices, with $Q$ and $R$ symmetric positive semidefinite. The notation $\Tr$ denotes the trace.

In the case when $R=0$, these operators are Ornstein-Uhlenbeck operators and have been studied in many works as \cite{MR2729292, MR2505366, MR2257846, MR2313847, MR1475774, MR1941990,  MR3342487}. We recall from these works that the assumption of hypoellipticity is then characterized by the following equivalent assertions: \\[5pt]
\textbf{1.} The Ornstein-Uhlenbeck operator $\frac12\Tr(Q\nabla^2_x) + \langle Bx,\nabla_x\rangle$ is hypoelliptic. \\[5pt]
\textbf{2.} The symmetric positive semidefinite matrices $$Q_t = \int_0^te^{sB}Qe^{sB^T}\ ds,$$ are nonsingular for some (equivalently, for all) $t>0$, i.e. $\det Q_t>0$. \\[5pt]
\textbf{3.} The Kalman rank condition holds: 
\begin{equation}\label{11122017E5}
	\Rank[B\ \vert\ Q^{\frac12}] = n,
\end{equation}
where $$[B\ \vert\ Q^{\frac12}] = [Q^{\frac12},BQ^{\frac12},\ldots,B^{n-1}Q^{\frac12}],$$
is the $n\times n^2$ matrix obtained by writing consecutively the columns of the matrices $B^jQ^{\frac12}$, with $Q^{\frac12}$ the symmetric positive semidefinite matrix given by the square root of $Q$. \\[5pt]
\textbf{4.} The H\"ormander condition holds: $$\forall x\in\mathbb R^n,\quad \Rank\mathcal{L}(X_1,X_2,\ldots,X_n,Y_0)(x) = n,$$
with $$Y_0 = \langle Bx,\nabla_x\rangle,\quad X_i = \sum_{j=1}^nq_{i,j}\partial_{x_j},\quad i=1,\ldots,n,$$
where $\mathcal{L}(X_1,X_2,\ldots,X_n,Y_0)(x)$ denotes the Lie algebra generated by the vector fields 
$$X_1,X_2,\ldots,X_n\quad \text{and}\quad Y_0$$
at point $x\in\mathbb R^n$. \\[5pt]
When the Ornstein-Uhlenbeck operator $\frac12\Tr(Q\nabla^2_x) + \langle Bx,\nabla_x\rangle$ is hypoelliptic, the associated Markov semigroup $(T(t))_{t\geq0}$ has the following explicit representation due to Kolmogorov \cite{MR1503147}:
$$(T(t)f)(x) = \frac1{(2\pi)^{\frac n2}\sqrt{\det Q_t}}\int_{\mathbb R^n}e^{-\frac12\langle Q^{-1}_ty,y\rangle}f(e^{tB}x-y)\ dy,\quad t>0.$$

We consider here the general case when $R$ is a symmetric positive semidefinite matrix and we study the partial Gelfand-Shilov smoothing properties of the semigroup $(e^{tP})_{t\geq0}$ and the null-controllability of generalized Ornstein-Uhlenbeck equations.

We first observe that $P$ writes as
\begin{equation}\label{20042018E16}
	P = -q^w(x,D_x) - \frac12\Tr(B),
\end{equation}
where $q^w(x,D_x)$ is the quadratic operator associated to the following complex-valued quadratic form:
\begin{equation}\label{20042018E15}
	q(x,\xi) = \frac12\vert Q^{\frac12}\xi\vert^2 + \frac12\vert R^{\frac12}x\vert^2 - i\langle Bx,\xi\rangle, \quad (x,\xi)\in\mathbb R^{2n}.
\end{equation}
Since the polar form of $q$ is given for all $(x,\xi),(y,\eta)\in\mathbb R^{2n}$ by
\begin{align*}
	q((x,\xi),(y,\eta)) & = \frac12\langle Q\xi,\eta\rangle + \frac12\langle Rx,y\rangle - \frac{i}{2}(\langle Bx,\eta\rangle + \langle By,\xi\rangle) \\[5pt]
	& = \frac12\ \sigma((x,\xi),(-iBy + Q\eta, iB^T\eta - Ry)),
\end{align*}
the Hamilton map and the singular space of $q$ are respectively given by
\begin{gather}\label{20042018E12}
F = \frac12\begin{pmatrix}
	-iB & Q \\
	-R & iB^T
\end{pmatrix}
\end{gather}
and
\begin{equation}\label{27062018E1}
	S = \Bigg[\bigcap_{j=0}^{2n-1}\Ker(RB^j)\Bigg]\times\Bigg[\bigcap_{j=0}^{2n-1}\Ker(Q(B^T)^j)\Bigg].
\end{equation}
The Cayley-Hamilton theorem applied to the matrices $B$ and $B^T$ shows that for all $j\geq0$ and $X\in\mathbb R^n$,
$$B^jX\in\Span(X,BX,\ldots,B^{n-1}X),$$
and
$$(B^T)^jX\in\Span(X,B^TX,\ldots,(B^T)^{n-1}X).$$
It follows that the singular space reduces to
\begin{equation}\label{24042018E6}
	S = \Bigg[\bigcap_{j=0}^{n-1}\Ker(RB^j)\Bigg]\times\Bigg[\bigcap_{j=0}^{n-1}\Ker(Q(B^T)^j)\Bigg].
\end{equation}

According to \eqref{20042018E16},  the smoothing properties of the semigroup $(e^{tP})_{t\geq0}$ are the very same as the ones of the semigroup $(e^{-tq^w})_{t\geq0}$, where $q:\mathbb R^n_x\times\mathbb R^n_{\xi}\rightarrow\mathbb{C}$ is the complex-valued quadratic form given by \eqref{20042018E15}, and Theorem \ref{20112017T1} allows to prove that under suitable algebraic conditions for the matrices $Q$, $R$ and $B$, the semigroup $(e^{-tq^w})_{t\geq0}$ enjoys partial Gelfand-Shilov regularizing effects:

\begin{thm}\label{13122017T1} Let $P$ be the generalized Ornstein-Uhlenbeck operator defined in \eqref{20042018E17}. We assume that there exist some subsets $I,J\subset\{1,\ldots,n\}$ such that
\begin{equation}\label{20042018E14}
	\Bigg(\Bigg[\bigcap_{j=0}^{n-1}\Ker(RB^j)\Bigg]\times\Bigg[\bigcap_{j=0}^{n-1}\Ker(Q(B^T)^j)\Bigg]\Bigg)^{\perp} = \mathbb R^n_I\times\mathbb R^n_J,
\end{equation}
the orthogonality being taken with respect to the canonical Euclidean structure of $\mathbb R^{2n}$. We also assume that the matrices $Q$, $R$ and $B$ satisfy
\begin{equation}\label{20042018E13}
	\Bigg[\bigcap_{j=0}^{n-1}\Ker(RB^j)\Bigg]\times\Bigg[\bigcap_{j=0}^{n-1}\Ker(Q(B^T)^j)\Bigg]\subset\Ker B\times\Ker B^T.
\end{equation}
Then, there exist some positive constants $C>1$ and $0<t_0<1$ such that for all $(\alpha,\beta)\in\mathbb N^n_I\times\mathbb N^n_J$, $0<t\le t_0$ and $u\in L^2(\mathbb R^n)$,
$$\big\Vert x^{\alpha}\partial^{\beta}_x(e^{tP}u)\big\Vert_{L^2(\mathbb R^n)}\le\frac{C^{1+\vert\alpha\vert + \vert \beta\vert}}{t^{(2k_0+1)(\vert\alpha\vert + \vert\beta\vert + s)}}\ (\alpha!)^{\frac12}\ (\beta!)^{\frac12}\ \Vert u\Vert_{L^2(\mathbb R^n)},$$
where $0\le k_0\le n-1$ is the smallest integer satisfying
$$\Bigg[\bigcap_{j=0}^{n-1}\Ker(RB^j)\Bigg]\times\Bigg[\bigcap_{j=0}^{n-1}\Ker(Q(B^T)^j)\Bigg] = \Bigg[\bigcap_{j=0}^{k_0}\Ker(RB^j)\Bigg]\times\Bigg[\bigcap_{j=0}^{k_0}\Ker(Q(B^T)^j)\Bigg],$$
and $s = 9n/4+2\lfloor n/2\rfloor + 3$.
\end{thm}

\begin{proof} We consider the quadratic form $q:\mathbb R^n_x\times\mathbb R^n_{\xi}\rightarrow\mathbb C$ defined in \eqref{20042018E15}. It follows from \eqref{24042018E6} and \eqref{20042018E14} that the singular space $S$ of $q$ satisfies $S^{\perp} = \mathbb R^n_I\times\mathbb R^n_J$. Moreover, \eqref{20042018E12}, \eqref{24042018E6} and \eqref{20042018E13} imply that $S\subset \Ker(\Imag F)$, where $F$ denotes the Hamilton map of the quadratic form $q$. Notice from \eqref{27062018E1} and \eqref{24042018E6} that the smallest integer $0\le k_0\le n-1$ satisfying
$$\Bigg[\bigcap_{j=0}^{n-1}\Ker(RB^j)\Bigg]\times\Bigg[\bigcap_{j=0}^{n-1}\Ker(Q(B^T)^j)\Bigg] = \Bigg[\bigcap_{j=0}^{k_0}\Ker(RB^j)\Bigg]\times\Bigg[\bigcap_{j=0}^{k_0}\Ker(Q(B^T)^j)\Bigg],$$
is also the smallest integer $0\le k_0\le 2n-1$ satisfying
$$S = \Bigg[\bigcap_{j=0}^{k_0}\Ker(RB^j)\Bigg]\times\Bigg[\bigcap_{j=0}^{k_0}\Ker(Q(B^T)^j)\Bigg].$$
It therefore follows from Theorem \ref{20112017T1} that there exist some positive constants $C>1$ and $0<t_0<1$ such that for all $(\alpha,\beta)\in\mathbb N^n_I\times\mathbb N^n_J$, $0<t\le t_0$ and $u\in L^2(\mathbb R^n)$,
\begin{equation}\label{20042018E18}
	\big\Vert x^{\alpha}\partial^{\beta}_x(e^{-tq^w}u)\big\Vert_{L^2(\mathbb R^n)}\le\frac{C^{1+\vert\alpha\vert + \vert\beta\vert}}{t^{(2k_0+1)(\vert\alpha\vert + \vert\beta\vert + s)}}\ (\alpha!)^{\frac12}\ (\beta!)^{\frac12}\ \Vert u\Vert_{L^2(\mathbb R^n)},
\end{equation}
with $s = 9n/4+2\lfloor n/2\rfloor + 3$. Theorem \ref{13122017T1} is then a consequence of \eqref{20042018E16} and \eqref{20042018E18}.
\end{proof}

As an application of Theorem \ref{05122017T2}, we study the null-controllability of the parabolic equation
$$\left\{\begin{array}{l}
	\partial_t f(t,x) -Pf(t,x)  = u(t,x)\mathbbm{1}_{\omega}(x),\quad (t,x)\in(0,+\infty)\times\mathbb R^n, \\[5pt]
	f(0) = f_0\in L^2(\mathbb R^n),
\end{array}\right.$$
where $\omega\subset\mathbb R^n$ is a measurable subset with a positive Lebesgue measure and $P$ is the generalized Ornstein-Uhlenbeck operator defined in \eqref{20042018E17}. When $R = 0$ and $P$ stands for a hypoelliptic Ornstein-Uhlenbeck operator, J. Bernier and the author proved in \cite{AB} (Theorem 1.3) that this equation is null-controllable in any positive time, once the control subset $\omega\subset\mathbb R^n$ is thick. When $R\ne0$, we derive from Theorem \ref{05122017T2} the following result:

\begin{thm}\label{11122017T1} Let $B,Q$ and $R$ be some real $n\times n$ matrices, with $Q$ and $R$ symmetric semidefinite positive, such that $B$ and $Q^{\frac12}$ satisfy the Kalman rank condition \eqref{11122017E5}. We assume that there exists a subset $I\subset\{1,\ldots,n\}$ such that
\begin{equation}\label{05062019E1}
	\Bigg[\bigcap_{j=0}^{n-1}\Ker(RB^j)\Bigg]^{\perp} = \mathbb R^n_I,
\end{equation}
the orthogonality being taken with respect to the canonical Euclidean structure of $\mathbb R^n$.
We also assume that the matrices $R$ and $B$ satisfy
\begin{equation}\label{13122017E2}
	\bigcap_{j=0}^{n-1}\Ker(RB^j)\subset\Ker B.
\end{equation}
If $\omega\subset\mathbb R^n$ is a thick subset, then the parabolic equation 
\begin{equation}\label{23042018E2}
\left\{\begin{array}{l}
	\partial_t f(t,x) -Pf(t,x)  = u(t,x)\mathbbm{1}_{\omega}(x),\quad (t,x)\in(0+\infty)\times\mathbb R^n, \\[5pt]
	f(0) = f_0\in L^2(\mathbb R^n),
\end{array}\right.
\end{equation}
is null-controllable from the set $\omega$ in any positive time $T>0$, where $P$ stands for the generalized Ornstein-Uhlenbeck operator associated to $B$, $Q$ and $R$ defined in \eqref{20042018E17}.
\end{thm}

\begin{proof} By using the change of unknowns $g(t,x)=e^{\frac12\Tr(B)t}f(t,x)$ and $v(t,x) = e^{\frac12\Tr(B)t}u(t,x)$, where $f$ is solution of \eqref{23042018E2} with control $u$, we notice from \eqref{20042018E16} that the result of Theorem \ref{11122017T1} is equivalent to the null-controllability of the equation 
\begin{equation}\label{23042018E3}
\left\{\begin{array}{l}
	\partial_t g(t,x) + q^w(x,D_x)g(t,x)  = v(t,x)\mathbbm{1}_{\omega}(x),\quad (t,x)\in(0,+\infty)\times\mathbb R^n, \\[5pt]
	g(0) = f_0\in L^2(\mathbb R^n).
\end{array}\right.
\end{equation}
Since $B$ and $Q^{\frac12}$ satisfy the Kalman rank condition \eqref{11122017E5}, we notice that $\Ran[B\ \vert\ Q^{\frac12}] = \mathbb R^n$, where $\Ran$ denotes the range, and it follows that
$$\bigcap_{j=0}^{n-1}\Ker(Q^{\frac12}(B^T)^j) = \Ker\big([B\ \vert\ Q^{\frac12}]^T\big) = \big(\Ran[B\ \vert\ Q^{\frac12}]\big)^{\perp} = \{0\},$$
where $\perp$ denotes the orthogonality with respect to the canonical Euclidean structure. Moreover, the equality of kernels $\Ker Q = \Ker Q^{\frac12}$ implies that
\begin{equation}\label{24042018E8}
	\bigcap_{j=0}^{n-1}\Ker(Q(B^T)^j) = \bigcap_{j=0}^{n-1}\Ker(Q^{\frac12}(B^T)^j) = \{0\}.
\end{equation}
According to \eqref{24042018E6} and \eqref{24042018E8}, the singular space of $q$ is given by
\begin{equation}\label{24042018E7}
	S = \Bigg[\bigcap_{j=0}^{n-1}\Ker(RB^j)\Bigg]\times\{0\}.
\end{equation}
We deduce from \eqref{05062019E1} and \eqref{24042018E7} that $S^{\perp} = \mathbb R^n_I\times\mathbb R^n_{\xi}$. The quadratic form $q$ is therefore diffusive. Moreover, it follows from \eqref{20042018E12}, \eqref{13122017E2} and \eqref{24042018E7} that $S\subset\Ker(\Imag F)$. The null-controllability of \eqref{23042018E3} is then a consequence of Theorem \ref{05122017T2}.
\end{proof}

\begin{ex} Let $a,b,c,d,e$ be some real numbers satisfying $b\ne0$ or $c\ne 0$, $d\ne0$ and $e>0$. We consider the three matrices of $M_2(\mathbb R)$ given by
$$Q = \begin{pmatrix}
	a & c \\
	c & b
\end{pmatrix},\quad R = \begin{pmatrix}
	0 & 0 \\
	0 & e
\end{pmatrix}\quad \text{and}\quad B = \begin{pmatrix}
	0 & d \\
	0 & 0
\end{pmatrix}.$$
Since $Q$ is a real symmetric matrix, $Q^2$ is a symmetric positive semidefinite matrix. Moreover, $e>0$ implies that $R$ is also a symmetric positive semidefinite matrix. The generalized Ornstein-Uhlenbeck operator associated to $B$, $Q^2$ and $R$ is given by
$$P = \frac12(a^2+c^2)\partial^2_x + \frac12(b^2+c^2)\partial^2_v + (ac+bc)\partial_x\partial_v - \frac12ev^2 + dv\partial_x,\quad (x,v)\in\mathbb R^2.$$
Notice that when $a=c=0$, $b^2=2$, $d=-1$ and $e=\frac12$, $-P=-\partial^2_v + \frac14v^2+v\partial_x$ is the Kramers-Fokker-Planck operator without external potential. Since $b\ne0$ or $c\ne0$, and $d\ne0$, we deduce that
$$\Rank[B\ \vert\ Q] =\Rank\left[\begin{array}{cccc}
	a & c & dc & db \\[5pt]
	c & b & 0 & 0
\end{array}\right] = 2,$$
that is $B$ and $Q$ satisfy the Kalman rank condition. Moreover, it follows from a straightforward computation that
$$\Ker R\cap\Ker(RB) = \mathbb R\times\{0\},$$
since $RB = 0$ and $e>0$. This implies in particular that $(\Ker R\cap\Ker(RB))^{\perp} = \mathbb R^2_I$, with $I = \{2\}$, the orthogonality being taken with respect to the canonical Euclidean structure of $\mathbb R^2$, and $\Ker R\cap\Ker(RB)\subset\Ker B$. It therefore follows from Theorem \ref{11122017T1} that when $\omega\subset\mathbb R^2$ is a thick subset, the parabolic equation
$$\left\{\begin{array}{l}
	\partial_t f(t,x,v) -Pf(t,x,v)  = u(t,x,v)\mathbbm1_{\omega}(x,v),\quad (t,x,v)\in(0,+\infty)\times\mathbb R^2, \\[5pt]
	f(0) = f_0\in L^2(\mathbb R^2),
\end{array}\right.$$
is null-controllable from $\omega$ in any positive time.
\end{ex}

\section{Appendix}
\label{Appendix}

\subsection{Miscellaneous estimates} The following factorial estimates are instrumental in this work, see for instance (0.3.3), (0.3.6) and (0.3.7) in \cite{MR2668420}:
\begin{equation}\label{26032018E1}
	\forall\alpha\in\mathbb N^n,\quad \vert\alpha\vert!\le n^{\vert\alpha\vert}\ \alpha!,
\end{equation}
\begin{equation}\label{22102017E1}
	\forall \alpha,\beta\in\mathbb N^n,\quad (\alpha+\beta)! \le 2^{\vert\alpha+\beta\vert}\ \alpha!\ \beta!\le 2^{\vert\alpha+\beta\vert}\ (\alpha+\beta)!.
\end{equation}
The inequality \eqref{26032018E1} as well as the left estimate in \eqref{22102017E1} are consequences of the multinomial formula, while the right one is straightforward. Another consequence of the multinomial formula is the following estimate
\begin{equation}\label{20062018E1}
	\forall \alpha\in\mathbb N^n,\quad\sum_{\beta\le\alpha}\binom{\alpha}{\beta} = 2^{\vert\alpha\vert},
\end{equation}
see (0.3.8) in \cite{MR2668420}. On the other hand, we notice that for all $N\in\mathbb N$ and $x\in\mathbb R^n$,
\begin{equation}\label{24112017E3}
	\vert x\vert^N\le n^N\sum_{\vert\alpha\vert = N}\vert x^{\alpha}\vert,
\end{equation}
where $\vert\cdot\vert$ stands for the Euclidean norm. Indeed, it follows from the multinomial theorem that for all $N\in\mathbb N$ and $x\in\mathbb R^n$,
\begin{equation}\label{26032018E2}
	\vert x\vert^N = \bigg(\sum_{j=1}^n\vert x_j\vert^2\bigg)^{\frac N2}\le\bigg(\sum_{j=1}^n\vert x_j\vert\bigg)^N = \sum_{\vert\alpha\vert = N}\frac{N!}{\alpha!}\vert x^{\alpha}\vert,
\end{equation}
since
$$\forall a,b\geq0,\quad (a+b)^{\frac12}\le a^{\frac12} + b^{\frac12}.$$
Yet, we derive from \eqref{26032018E1} that for all $\alpha\in\mathbb N^n$, $\vert\alpha\vert = N$,
$$N! = \vert\alpha\vert!\le n^{\vert\alpha\vert}\ \alpha! =  n^N\ \alpha!,$$
which, combined to \eqref{26032018E2}, leads to the desired estimate:
$$\forall N\in\mathbb N,\forall x\in\mathbb R^n,\quad \vert x\vert^N\le\sum_{\vert\alpha\vert = N}\frac{n^N\alpha!}{\alpha!}\vert x^{\alpha}\vert
=n^N\sum_{\vert\alpha\vert = N}\vert x^{\alpha}\vert.$$
Finally, we get from (0.3.15) and (0.3.16) in \cite{MR2668420} that for all $m\geq1$,
\begin{equation}\label{24042018E9}
	\#\big\{\alpha\in\mathbb N^n,\quad \vert\alpha\vert\le m\big\} = \binom{m+n}m,
\end{equation}
and
\begin{equation}\label{24042018E5}
	\#\big\{\alpha\in\mathbb N^n,\quad\vert\alpha\vert = m\big\} = \binom{m+n-1}m,
\end{equation}
where $\#$ denotes the cardinality.

\subsection{About the Gelfand-Shilov regularity}\label{GSreg} We refer the reader to \cite{MR2668420} (Chapter 6) for an extensive exposition of the Gelfand-Shilov regularity. The Gelfand-Shilov spaces $S^{\mu}_{\nu}(\mathbb R^n)$, with $\mu,\nu>0$, $\mu+\nu\geq1$, are defined as the spaces of smooth functions $f\in\mathscr S(\mathbb R^n)$ satisfying that there exist some positive constants $\varepsilon>0$ and $C>0$ such that
$$\begin{array}{ll}
	\forall x\in\mathbb R^n,\quad & \vert f(x)\vert\le Ce^{-\varepsilon\vert x\vert^{\frac{1}{\nu}}}, \\[7pt]
	\forall\xi\in\mathbb R^n,\quad & \vert \widehat{f}(x)\vert\le Ce^{-\varepsilon\vert x\vert^{\frac{1}{\mu}}}.
\end{array}$$
We recall from \cite{MR2668420} (Theorem 6.1.6) that for $f\in\mathscr S(\mathbb R^n)$, the following conditions are equivalent:
\begin{enumerate}[label=(\roman*),leftmargin=* ,parsep=2pt,itemsep=0pt,topsep=2pt]
\item $f\in S^{\mu}_{\nu}(\mathbb R^n)$,
\item There exists a positive constant $C>1$ such that 
$$\begin{array}{ll}
	\forall x\in\mathbb R^n,\forall\alpha\in\mathbb N^n,\quad & \big\Vert x^{\alpha}f(x)\big\Vert_{L^{\infty}(\mathbb R^n)}\le C^{1+\vert\alpha\vert}\ (\alpha!)^{\nu}, \\[7pt]
	\forall \xi\in\mathbb R^n,\forall\beta\in\mathbb N^n,\quad & \big\Vert\xi^{\beta}\widehat f(\xi)\big\Vert_{L^{\infty}(\mathbb R^n)}\le C^{1+\vert\beta\vert}\ (\beta!)^{\mu}.
\end{array}$$
\item There exists a positive constant $C>1$ such that 
$$\begin{array}{ll}
	\forall x\in\mathbb R^n,\forall\alpha\in\mathbb N^n,\quad & \big\Vert x^{\alpha}f(x)\big\Vert_{L^2(\mathbb R^n)}\le C^{1+\vert\alpha\vert}\ (\alpha!)^{\nu}, \\[7pt]
	\forall \xi\in\mathbb R^n,\forall\beta\in\mathbb N^n,\quad & \big\Vert\partial^{\beta}_xf(x)\big\Vert_{L^2(\mathbb R^n)}\le C^{1+\vert\beta\vert}\ (\beta!)^{\mu}.
\end{array}$$
\item There exists a positive constant $C>1$ such that
$$\forall(\alpha,\beta)\in\mathbb N^{2n},\quad \big\Vert x^{\alpha}\partial^{\beta}_xf(x)\big\Vert_{L^2(\mathbb R^n)}\le C^{1+\vert\alpha\vert+\vert\beta\vert}\ (\alpha!)^{\nu}\ (\beta!)^{\mu},$$
\item There exists a positive constant $C>1$ such that
$$\forall(\alpha,\beta)\in\mathbb N^{2n},\quad \big\Vert x^{\alpha}\partial^{\beta}_xf(x)\big\Vert_{L^{\infty}(\mathbb R^n)}\le C^{1+\vert\alpha\vert+\vert\beta\vert}\ (\alpha!)^{\nu}\ (\beta!)^{\mu}.$$
\end{enumerate}
In the following, we need of to go through the proofs (i) $\Rightarrow$ (iv) $\Rightarrow$ (v) given in \cite{MR2668420} (Proposition 6.1.5 and Theorem 6.1.6) in order to make explicit the dependence of the different constants. For our purpose, we only consider the case when $\mu=\nu=\frac12$.

\begin{lem}\label{13072018L1} We have that for all non-negative integer $p\geq0$ and $c>0$,
$$\big\Vert x^pe^{-cx^2}\big\Vert_{L^2(\mathbb R)}\le \frac{\pi^{\frac14}}{c^{\frac p2+\frac14}}\ (p!)^{\frac12}.$$
\end{lem}

\begin{proof} For all non-negative integer $p\geq0$, we consider the integral
$$I_p = \int_{\mathbb R}\vert x\vert^{2p}e^{-x^2}dx.$$
It follows from an integration by parts that for all $p\geq0$,
\begin{multline*}
	I_{p+1} = \int_{\mathbb R}\vert x\vert^{2p+2}e^{-x^2}dx = 2\int_0^{+\infty}x^{2p+2}e^{-x^2}dx \\
	= \left[-x^{2p+1}e^{-x^2}\right]^{+\infty}_0 + (2p+1)\int_0^{+\infty}x^{2p}e^{-x^2}dx = \frac12(2p+1)I_p.
\end{multline*}
Moreover, we have $I_0 = \int_{\mathbb R}e^{-x^2}dx = \pi^{\frac12}$ and it follows from a straightforward induction that for all non-negative integer $p\geq0$,
\begin{equation}\label{16072018E3}
	I_p = \frac1{4^p}\frac{(2p)!}{p!}\pi^{\frac12}.
\end{equation}
We deduce from \eqref{22102017E1} that for all $p\geq0$,
\begin{equation}\label{16072018E4}
	\frac{(2p)!}{p!}\le \frac{2^{2p}(p!)^2}{p!} = 4^pp!,
\end{equation}
and as a consequence of \eqref{16072018E3} and \eqref{16072018E4}, we notice that
\begin{equation}\label{16072018E2}
	\forall p\geq0,\quad I_p\le \pi^{\frac12} p!.
\end{equation}
Let $p\geq0$ be a non-negative integer and $c>0$. It follows from \eqref{22102017E1}, \eqref{16072018E2} and the substitution rule that
$$\big\Vert x^pe^{-cx^2}\big\Vert^2_{L^2(\mathbb R)} = \int_{\mathbb R}\vert x\vert^{2p}e^{-2cx^2}dx = \frac1{(2c)^{p+\frac12}}\int_{\mathbb R}\vert x\vert^{2p}e^{-x^2}dx \\
\le \frac{\pi^{\frac12}p!}{(2c)^{p+\frac12}} 
\le \frac{\pi^{\frac12}p!}{c^{p+\frac12}}.$$
This ends the proof of Lemma \ref{13072018L1}.
\end{proof}

\begin{prop}\label{28032018P1} There exists a positive constant $C>1$ only depending on the dimension $N$ such that for all positive constants $0<c_1<1$, $0<c_2<1$, $C_1>0$, $C_2>0$ and Schwartz functions $f\in\mathscr S(\mathbb R^N)$ satisfying
\begin{gather}\label{25042018E3}
	\forall x\in\mathbb R^N,\quad \vert f(x)\vert\le C_1e^{-c_1\vert x\vert^2},
\end{gather}
and
\begin{gather}\label{25042018E4}
	\forall\xi\in\mathbb R^N,\quad \vert\widehat f(\xi)\vert\le C_2e^{-c_2\vert\xi\vert^2},
\end{gather}
we have
$$\forall\alpha,\beta\in\mathbb N^N,\quad \big\Vert x^{\alpha}\partial^{\beta}_xf(x)\big\Vert_{L^{\infty}(\mathbb R^N)}\le
C^{1+\vert\alpha\vert+\vert\beta\vert}\ \left[\frac{C_1}{c_1^{\vert\alpha\vert+\frac N4}}\ \frac{C_2}{c_2^{\vert\beta\vert+\lfloor\frac N2\rfloor+1+\frac N4}}\ \alpha!\ \beta!\right]^{\frac12}.$$
\end{prop}

\begin{proof} Let $f\in\mathscr S(\mathbb R^N)$ be a Schwartz function satisfying \eqref{25042018E3} and \eqref{25042018E4}. We first deduce from \eqref{25042018E3}, Lemma \ref{13072018L1} and the Fubini theorem that for all $\alpha\in\mathbb N^N$,
\begin{multline}\label{13072018E3}
	\big\Vert x^{\alpha}f(x)\big\Vert_{L^2(\mathbb R^N)}\le C_1\big\Vert x^{\alpha}e^{-c_1\vert x\vert^2}\big\Vert_{L^2(\mathbb R^N)} 
	= C_1\prod_{j=1}^N\big\Vert x_j^{\alpha_j}e^{-c_1 x_j^2}\big\Vert_{L^2(\mathbb R)} \\
	\le C_1\prod_{j=1}^N\frac{\pi^{\frac14}}{c_1^{\frac{\alpha_j}2+\frac14}}\ (\alpha_j!)^{\frac12}
	= \pi^{\frac N4}\ \frac{C_1}{c_1^{\frac{\vert\alpha\vert}2+\frac N4}}\ (\alpha!)^{\frac12}.
\end{multline}
Similarly, we deduce from \eqref{25042018E4}, Lemma \ref{13072018L1}, the Plancherel theorem and the Fubini theorem that for all $\beta\in\mathbb N^N$,
\begin{equation}\label{16072018E1}
	\big\Vert\partial^{\beta}_xf(x)\big\Vert_{L^2(\mathbb R^N)} = \frac1{(2\pi)^{\frac N2}}\big\Vert\xi^{\beta}\widehat f(\xi)\big\Vert_{L^2(\mathbb R^N)}
	\le \frac{\pi^{\frac N4}}{(2\pi)^{\frac N2}}\ \frac{C_2}{c_2^{\frac{\vert\beta\vert}2+\frac N4}}\ (\beta!)^{\frac12}.
\end{equation}
Let $\alpha,\beta\in\mathbb N^N$. With an integration par parts, we notice that
$$\big\Vert x^{\alpha}\partial^{\beta}_xf(x)\big\Vert^2_{L^2(\mathbb R^N)} = \langle\partial^{\beta}_xf(x),x^{2\alpha}\partial^{\beta}_xf(x)\rangle_{L^2(\mathbb R^N)}
= (-1)^{\vert\beta\vert}\langle f(x),\partial^{\beta}_x(x^{2\alpha}\partial^{\beta}_xf(x))\rangle_{L^2(\mathbb R^N)},$$
while it follows from the Leibniz formula that for all $x\in\mathbb R^N$,
\begin{align*}
	\partial^{\beta}_x(x^{2\alpha}\partial^{\beta}_xf(x)) 
	& = \sum_{\gamma\le\beta}\binom{\beta}{\gamma}\ \partial^{\gamma}_x(x^{2\alpha})\ \partial^{\beta-\gamma}_x(\partial^{\beta}_xf(x)) \\[5pt]
	& = \sum_{\gamma\le\beta,\ \gamma\le 2\alpha}\binom{\beta}{\gamma}\frac{(2\alpha)!}{(2\alpha-\gamma)!}\ x^{2\alpha-\gamma}\ \partial^{2\beta-\gamma}_xf(x) \\[5pt]
	& = \sum_{\gamma\le\beta,\ \gamma\le 2\alpha}\binom{\beta}{\gamma}\binom{2\alpha}{\gamma}\gamma!\ x^{2\alpha-\gamma}\ \partial^{2\beta-\gamma}_xf(x).
\end{align*}
We therefore deduce from the Cauchy-Schwarz inequality and \eqref{20062018E1} that
\begin{multline}\label{25042018E7}
	\big\Vert x^{\alpha}\partial^{\beta}_xf(x)\big\Vert^2_{L^2(\mathbb R^N)}
	\le\sum_{\gamma\le\beta,\ \gamma\le 2\alpha}\binom{\beta}{\gamma}\binom{2\alpha}{\gamma}\ \gamma!\ \vert\langle x^{2\alpha-\gamma}f(x),\partial^{2\beta-\gamma}_xf(x)\rangle_{L^2(\mathbb R^N)}\vert \\[5pt]
	\le 2^{2\vert\alpha\vert + \vert\beta\vert}\sum_{\gamma\le\beta,\ \gamma\le 2\alpha}\gamma!\ \big\Vert x^{2\alpha-\gamma}f(x)\big\Vert_{L^2(\mathbb R^N)}\big\Vert\partial^{2\beta-\gamma}_xf(x)\big\Vert_{L^2(\mathbb R^N)}.
\end{multline}
It follows from \eqref{13072018E3} and \eqref{16072018E1} that
\begin{multline}\label{25042018E8}
	\gamma!\ \big\Vert x^{2\alpha-\gamma}f(x)\big\Vert_{L^2(\mathbb R^N)}\big\Vert\partial^{2\beta-\gamma}_xf(x)\big\Vert_{L^2(\mathbb R^N)} \\[5pt]
	\le\frac{\pi^{\frac N2}}{(2\pi)^{\frac N2}}\ \frac{C_1}{c_1^{\frac{\vert2\alpha-\gamma\vert}{2}+\frac N4}}\ \frac{C_2}{c_2^{\frac{\vert2\beta-\gamma\vert}{2}+\frac N4}}\
	\gamma!\ ((2\alpha-\gamma)!)^{\frac12}\ ((2\beta-\gamma)!)^{\frac12}.
\end{multline}
Moreover, the estimate \eqref{22102017E1} implies that
\begin{multline}\label{25042018E9}
	\gamma!\ ((2\alpha-\gamma)!)^{\frac12}\ ((2\beta-\gamma)!)^{\frac12} = \left[\gamma!\ (2\alpha-\gamma)!\ \gamma!\ (2\beta-\gamma)!\right]^{\frac12} \\[5pt]
	\le\left[(2\alpha)!\ (2\beta)!\right]^{\frac12}
	\le\left[4^{\vert\alpha\vert+\vert\beta\vert}\ (\alpha!)^2\ (\beta!)^2\right]^{\frac12} = 2^{\vert\alpha\vert+\vert\beta\vert}\ \alpha!\ \beta!.
\end{multline}
We therefore obtain from \eqref{25042018E7}, \eqref{25042018E8} and \eqref{25042018E9} that
$$\big\Vert x^{\alpha}\partial^{\beta}_xf(x)\big\Vert^2_{L^2(\mathbb R^N)}
\le \frac1{2^{\frac N2}}\ \left[\sum_{\gamma\le\beta,\ \gamma\le 2\alpha}1\right]\ 2^{3\vert\alpha\vert + 2\vert\beta\vert}\ 
\frac{C_1}{c_1^{\vert\alpha\vert+\frac N4}}\ \frac{C_2}{c_2^{\vert\beta\vert+\frac N4}}\ \alpha!\ \beta!,$$
since $0<c_1,c_2<1$. It follows from \eqref{20062018E1} that the sum satisfies the following estimate
$$\sum_{\gamma\le\beta,\ \gamma\le 2\alpha}1\le \sum_{\gamma\le\beta}1\le\sum_{\gamma\le\beta}\binom{\beta}{\gamma} = 2^{\vert\beta\vert},$$
and as a consequence, we deduce that for all $\alpha,\beta\in\mathbb N^N$,
\begin{equation}\label{25042018E10}
	\big\Vert x^{\alpha}\partial^{\beta}_xf(x)\big\Vert^2_{L^2(\mathbb R^N)}
	\le 2^{3\vert\alpha\vert+3\vert\beta\vert-\frac N2}\
	\frac{C_1}{c_1^{\vert\alpha\vert+\frac N4}}\ \frac{C_2}{c_2^{\vert\beta\vert+\frac N4}}\ \alpha!\ \beta!.
\end{equation}
Then, the Sobolev embeddings give the existence of a positive constant $C>0$ only depending on the dimension $N$ such that for all $\alpha,\beta\in\mathbb N^N$,
$$\big\Vert x^{\alpha}\partial^{\beta}_xf(x)\big\Vert_{L^{\infty}(\mathbb R^N)}\le C\sum_{\vert\gamma\vert\le s}\big\Vert\partial^{\gamma}_x(x^{\alpha}\partial^{\beta}_xf(x))\big\Vert_{L^2(\mathbb R^N)},$$
where $s=\lfloor N/2\rfloor+1$. By using anew the Leibniz formula, we obtain that
\begin{gather}\label{25042018E11}
	\big\Vert x^{\alpha}\partial^{\beta}_xf(x)\big\Vert_{L^{\infty}(\mathbb R^N)}\le C\sum_{\vert\gamma\vert\le s}\sum_{\delta\le\gamma,\ \delta\le\alpha}\binom{\gamma}{\delta}\binom{\alpha}{\delta}\ \delta!\ \big\Vert x^{\alpha-\delta}\partial^{\beta+\gamma-\delta}_xf(x)\big\Vert_{L^2(\mathbb R^N)}.
\end{gather}
Let $\alpha,\beta\in\mathbb N^N$ and $\gamma,\delta\in\mathbb N^N$ such that $\vert\gamma\vert\le s$, $\delta\le\gamma$ and $\delta\le\alpha$. On the one hand, we deduce from \eqref{22102017E1} and \eqref{20062018E1} the following estimate:
\begin{equation}\label{20062018E2}
	\binom{\gamma}{\delta}\binom{\alpha}{\delta}\ \delta!\le 2^{\vert\gamma\vert+\vert\alpha\vert}\gamma!\le 2^s\ s!\ 2^{\vert\alpha\vert}.
\end{equation}
On the other hand, it follows from \eqref{22102017E1} and \eqref{25042018E10} that
\begin{multline}\label{25042018E12}
	\big\Vert x^{\alpha-\delta}\partial^{\beta+\gamma-\delta}_xf(x)\big\Vert_{L^2(\mathbb R^N)} \\[5pt]
	\le \left[2^{3\vert\alpha-\delta\vert+3\vert\beta+\gamma-\delta\vert-\frac N2}\
	\frac{C_1}{c_1^{\vert\alpha-\delta\vert+\frac N4}}\ \frac{C_2}{c_2^{\vert\beta+\gamma-\delta\vert+\frac N4}}\ (\alpha-\delta)!\ (\beta+\gamma-\delta)!\right]^{\frac12},
\end{multline}
with
$$(\alpha-\delta)!\le \alpha!\quad \text{and}\quad (\beta+\gamma-\delta)!\le (\beta+\gamma)!\le 2^{\vert\beta\vert+\vert\gamma\vert}\ \beta!\ \gamma!
\le 2^{\vert\beta\vert + s}\ s!\ \beta!.$$
Moreover, we deduce from \eqref{20062018E1} and \eqref{24042018E9} that
\begin{equation}\label{25042018E13}
	\sum_{\vert\gamma\vert\le s}\sum_{\delta\le\gamma,\ \delta\le\alpha}1
	\le \sum_{\vert\gamma\vert\le s}\sum_{\delta\le\gamma}1
	\le \sum_{\vert\gamma\vert\le s}\sum_{\delta\le\gamma}\binom{\gamma}{\delta} \\[5pt]
	\le \sum_{\vert\gamma\vert\le s}2^{\vert\gamma\vert}
	\le 2^s\binom{s+N}s
	\le 2^{2s + N}.
\end{equation}
Finally, we deduce from \eqref{25042018E11}, \eqref{20062018E2}, \eqref{25042018E12} and \eqref{25042018E13} that for all $\alpha,\beta\in\mathbb N^N$,
$$\big\Vert x^{\alpha}\partial^{\beta}_xf(x)\big\Vert_{L^{\infty}(\mathbb R^N)} \le 
C\ 2^{2s+N}\ 2^s\ s!\ 2^{\vert\alpha\vert}
\left[2^{3\vert\alpha\vert+3\vert\beta\vert + 3s -\frac N2}\ \frac{C_1}{c_1^{\vert\alpha\vert+\frac N4}}\ \frac{C_2}{c_2^{\vert\beta\vert+s+\frac N4}}\ 2^{\vert\beta\vert+s}\ s!\ \alpha!\ \beta!\right]^{\frac12},$$
that is
$$\big\Vert x^{\alpha}\partial^{\beta}_xf(x)\big\Vert_{L^{\infty}(\mathbb R^N)}
\le C\ (s!)^{\frac32}\ 2^{\frac52\vert\alpha\vert+\frac32\vert\beta\vert+\frac92s+\frac34N}
\left[\frac{C_1}{c_1^{\vert\alpha\vert+\frac N4}}\ \frac{C_2}{c_2^{\vert\beta\vert+ s +\frac N4}}\ \alpha!\ \beta!\right]^{\frac12}.$$
This ends the proof of Proposition \ref{28032018P1}.
\end{proof}

\nocite{*}
\bibliographystyle{siam}
\bibliography{BiblioControlabilite}

\end{document}